\documentclass[11pt]{amsart}

\usepackage{epsf,amssymb,times,overpic,amscd}
\usepackage[usenames]{color}
\usepackage[all]{xy}
\usepackage{hyperref}

\usepackage{amsfonts}
\usepackage{amsmath}
\usepackage{graphicx}
\usepackage{epsfig}
\usepackage{epstopdf}

\newcommand{\ba}{\begin{array}}
\newcommand{\ea}{\end{array}}
\newcommand{\be}{\begin{enumerate}}
\newcommand{\ee}{\end{enumerate}}

\newtheorem{thm}{Theorem}[section]
\newtheorem{prop}[thm]{Proposition}
\newtheorem{lemma}[thm]{Lemma}

\theoremstyle{definition}
\newtheorem{defn}[thm]{Definition}

\theoremstyle{remark}

\newtheorem{rmk}[thm]{Remark}
\newtheorem{example}[thm]{Example}

\newcommand{\C}{\mathcal{C}}

\newcommand{\Z}{\mathbb{Z}}

\newcommand{\K}{\mathbf{k}}

\newcommand{\Hom}{\operatorname{Hom}}
\newcommand{\RHom}{\operatorname{RHom}}
\newcommand{\End}{\operatorname{End}}
\newcommand{\Kom}{\operatorname{Kom}}
\newcommand{\Ind}{\operatorname{Ind}}
\newcommand{\Res}{\operatorname{Res}}
\newcommand{\Par}{\operatorname{Par}}
\newcommand{\Ker}{\operatorname{Ker}}
\newcommand{\wtpr}{\widetilde{\operatorname{pr}}}

\newcommand{\n}{\noindent}
\newcommand{\ot}{\otimes}
\newcommand{\bt}{\boxtimes}

\newcommand{\ra}{\rightarrow}
\newcommand{\xra}{\xrightarrow}

\newcommand{\mf}{\mathbf}
\newcommand{\op}{\operatorname}
\newcommand{\cal}{\mathcal}
\newcommand{\es}{\emptyset}
\newcommand{\lan}{\langle}
\newcommand{\ran}{\rangle}
\newcommand{\wt}{\widetilde}
\newcommand{\ov}{\overline}

\newcommand{\ovlm}{\overline{\lambda}}
\newcommand{\ovmu}{\overline{\mu}}
\newcommand{\ove}{\overline{\eta}}
\newcommand{\ovx}{\overline{x}}
\newcommand{\ovy}{\overline{y}}
\newcommand{\ovz}{\overline{z}}
\newcommand{\wtq}{\overline{Q}}
\newcommand{\wta}{\widetilde{a}}

\newcommand{\wtp}{\widetilde{P}_{-1}}
\newcommand{\bmn}{B(m,n)}
\newcommand{\mb}{\mathbf{1}}
\newcommand{\ch}{\mathcal{DH}}
\newcommand{\cl}{\mathcal{CL}}
\newcommand{\cf}{\mathcal{F}}
\newcommand{\cg}{\mathcal{G}}
\newcommand{\cb}{\mathcal{B}}
\newcommand{\cim}{\mathcal{IM}}
\newcommand{\cm}{\mathcal{M}}
\newcommand{\ii}{\rho_i}

\newcommand{\lm}{\lambda}
\newcommand{\lmn}{\lambda\cup1^n}

\newcommand{\vf}{V_{F}}
\newcommand{\br}{\overline{F}}

\newcommand{\sn}{\bigoplus\limits_{n=0}^{\infty}\mathbf{k}[S(n)]}
\newcommand{\cs}{\mathcal{S}}

\begin{document}
%%%%%%%%%%%%%%%%%%%%%%%%%%%%%%%%%%%%%%%%%%%%%%%%%%%%%%%%%
\title{Towards a categorical boson-fermion correspondence}

\author{Yin Tian}
\address{Yau Mathematical Sciences Center, Tsinghua University, Beijing 100084, China}
\email{ytian@math.tsinghua.edu.cn}

\date{October 31, 2017.}



\begin{abstract}
We construct the Heisenberg counterpart of a Clifford categorification in \cite{Tian}.
It is a modification of Khovanov's Heisenberg categorification.
We express generators of the Heisenberg category as a complex of generators of the Clifford category.
Certain vertex operators associated to the Clifford algebra are lifted to endofunctors of the Fock space categorification.
\end{abstract}

\maketitle

%%%%%%%%%%%%%%%%%%%%%%%%%%%%%%%%%%%%%%%%%%%%%%%%%%%%%%%%%
\tableofcontents

%%%%%%%%%%%%%%%%%%%%%%%%%%%%%%%%%%%%%%%%%%%%%%%%%%%%%%%%%
\section{Introduction}
The boson-fermion correspondence establishes an isomorphism between a bosonic Fock space and a fermionic Fock space.
%Here, $\wt{V_B} \cong \mathbb{C}[x_1, x_2, \dots; q, q^{-1}]$, and $V_F$ has a basis of semi-infinite sequences.
A Heisenberg algebra acts on the bosonic Fock space, and a Clifford algebra acts on the fermionic Fock space.
%Here, $Cl'$ is generated by $\psi_j, \psi_j^*$ for $j\in \Z$, and subject to the relations:
%$$\{\psi_i, \psi_j^*\}=\psi_i\psi_j^*+\psi_j^*\psi_i=\delta_{ij}, \quad \{\psi_i, \psi_j\}=0, \quad \{\psi_i^*, \psi_j^*\}=0.$$
The correspondence also provides maps between the Heisenberg and Clifford algebras via vertex operators  \cite{Fr} \cite[Section 14.9]{Kac}.

%The vertex operators are given by
%$$\psi(z)=\sum\limits_{j \in \Z} \psi_j z^j, \qquad \psi^*(z)=\sum\limits_{j \in \Z} \psi_j^* z^{-j}.$$

On the categorical level, the author constructed a DG categorification of a Clifford algebra \cite{Tian}.
The motivation is from studying three dimensional contact topology on $(\mathbb{R} \times [0,1]) \times [0,1]$.
The infinite strip $\mathbb{R} \times [0,1]$ is the universal cover of a punctured disk.
The boson-fermion correspondence is an example of a conformal field theory for a punctured disk.
For any oriented surface $\Sigma$, Honda defines a {\em contact category} $\C(\Sigma)$ constructed from contact structures on $\Sigma \times [0,1]$ \cite{HT}.
Contact categories and its close relative {\em bordered Heegaard Floer homology} \cite{LOT} can be viewed as a part of (2+1+1) dimensional topological field theory.
We hope this paper provides a small step towards categorifying conformal field theory via three dimensional geometric structures.

The goal of this paper is to give a representation theoretic interpretation of the geometric structure underlying the Clifford categorification.
Khovanov in \cite{Kh1} constructed a $\K$-linear additive categorification of the Heisenberg algebra, where $\K$ is a field of characteristic zero.
The Heisenberg category acts on the category of $\sn$-modules, where $S(n)$ is the $n$-th symmetric group.
We provide a modification of Khovanov's Heisenberg category, and show that it is the Heisenberg counterpart of the Clifford category under a categorical boson-fermion correspondence.

On the Heisenberg side, we construct a $\K$-algebra $B$ which contains $\sn$ as a subalgebra.
The homotopy category $\cb=\Kom(B)$ of finite dimensional projective $B$-modules categorifies the bosonic Fock space.
The derived category $D(B^e)$ of $B$-bimodules admits a monoidal structure given by the derived tensor product over $B$.
%The algebra $B$ as a $B$-bimodule is isomorphic to the unit object of $\mf{D}(B^e)$.
There are two distinguished bimodules $P$ and $Q$ which correspond to the induction and restriction functors of $\cb$.
Our Heisenberg category $\ch$ is defined as a full triangulated monoidal subcategory of $D(B^e)$ which is generated by $B, P$ and $Q$.

On the Clifford side, we generalize the construction in \cite{Tian} from $\mathbb{F}_2$ to $\K$ of characteristic zero.
We define a DG $\K$-algebra $R=\bigoplus\limits_{k\in\Z}R_{k}$, where all $R_{k}$'s are isomorphic to each other.
A homotopy category of certain DG $R$-modules categorifies the fermionic Fock space.
%The category $\cf$ is a disjoint union $\bigcup\limits_{i\in\Z}\cf_{k}$ of subcategories.
%Each $\cf_{k}$ is equivalent to the homotopy category $\Kom(F)$ of finite dimensional projective $F$-modules.
There are a family of distinguished DG $R$-bimodules $T(i)$ for $i \in \Z$ which correspond to certain contact geometric objects.
Our Clifford category $\cl$ is defined as a full triangulated monoidal subcategory of the derived category $D(R^e)$ which is generated by $R$ and $T(i)$'s.

The DG algebra $R_0$ is formal, i.e. quasi-isomorphic to its cohomology algebra $H(R_0)$ with the trivial differential.
It is derived Morita equivalent to a DG algebra $\wt{H}(R_0)$ with the trivial differential and concentrated at degree zero.
Then $\wt{H}(R_0)$ can be viewed as an ordinary algebra.
We show that it is isomorphic to a quiver algebra $F$, and the algebras $B$ and $F$ are Morita equivalent.
To sum up, we have a chain of the algebras:
\begin{gather} \label{eq chain rel1}
R_0 \leftrightarrow H(R_0) \leftrightarrow \wt{H}(R_0) \cong F \leftrightarrow B.
\end{gather}
Certain categories of $B$-modules and $R_0$-modules are equivalent.
This equivalence categorifies the isomorphism of the Fock spaces, see Theorem \ref{thm bf fock}.

The key feature of the category $\cb$ is that the algebra $B$ is not semisimple.
On the negative side, the Heisenberg category $\ch$ does not admit bi-adjunction.
As a result, it does not have an easy diagrammatic presentation.
It is expected that there are some difficulties in describing DG or triangulated monoidal categories using diagrams.
On the positive side, there are some $B$-bimodule homomorphisms and extensions between $B, P$ and $Q$ which do not exist in Khovanov's Heisenberg category.
These extra morphisms enable us to construct an infinite chain of adjoint pairs in $\ch$ which contains the bimodules $P$ and $Q$, see Theorem \ref{thm adj}.

Correspondingly, the bimodules $T(i)$ for $i \in \Z$ form a chain of adjoint pairs in the Clifford category.
Their classes $t_i=[T(i)]$ in the Grothendieck group generate a Clifford algebra $Cl$ with the relation:
$$t_i t_j + t_j t_i= \delta_{|i-j|,1}1.$$
Note that the basis $\{t_i\}$ is slightly different from the basis $\{\psi_i, \psi_i^*\}$ of creating and annihilating operators.

Using a variation of vertex operator construction, we can express the Heisenberg generators $p, q$ satisfying $qp-pq=1$ in terms of the Clifford generators as
\begin{gather*}
g(q)= \sum \limits_{i\le0}t_{2i}t_{2i-1} - \sum \limits_{i>0}t_{2i-1}t_{2i},\\
g(p)= \sum \limits_{i\le0}t_{2i+1}t_{2i} - \sum \limits_{i>0}t_{2i+1}t_{2i}.
\end{gather*}
We construct two objects $\wtq, \ov{P}$ in $D(R_0^e)$ which lift the expressions $g(q), g(p)$.
The chain (\ref{eq chain rel1}) induces an equivalence $\cg: D(B^e) \ra D(R_0^e)$ of categories.
Our first main result is to show that $\cg(Q)$ and $\cg(P)$ are isomorphic to $\wtq$ and $\ov{P}$, respectively, see Theorem \ref{thm bf HCl}.

Consider two generating series
$$\overline{t}(z)=\sum\limits_{i\in \Z} t_{2i+1} z^{i}, \qquad t(z)=\sum\limits_{i\in \Z} t_{2i} z^{-i},$$
associated to $Cl$.
The expressions $\overline{t}(z)|_{z=-1}$ and $t(z)|_{z=-1}$ define two linear operators of the Fock space.
Our second main result is to categorify these operators to certain endofunctors of $\cb$, see Theorems \ref{thm bar sn} and \ref{thm sn}.
The endofunctors are related to {\em Serre functors}.
We hope that this gives an idea towards categorifying vertex operators in general.
At this moment we do not know what the categorical meaning of $z$ is when treating $z$ as a complex variable.

\vspace{.2cm}
\n{\bf Comparison with other works.}
A similar non-semisimplicity appears in the work of Frenkel, Penkov and Serganova \cite{FPS}.
They gave a categorification of the boson-fermion correspondence via the representation theory of $\mathfrak{sl}(\infty)$.
Khovanov \cite{Kh2} pointed out that $\cb$ is also related to a category of certain $\op{Sym}^*(V) \rtimes \mathfrak{gl}(\infty)$-modules, where $V$ is the infinite dimensional vector representation of $\mathfrak{gl}(\infty)$.

Cautis and Sussan \cite{CS} constructed another categorical version of the correspondence whose Heisenberg side is Khovanov's categorification.
They define some infinite complexes of the Heisenberg generators which satisfy a categorical Clifford relation.
The construction is based on the previous work of Cautis and Licata \cite{CL1, CL2} on categorification of quantum affine algebras.
Although the bi-adjunction is missing in our case, we do have a chain of adjoint pairs in both the Heisenberg and Clifford categories.
We can first use the complexes in \cite{CS} in one direction, and then apply the adjunction map to get the complexes in the other direction.
We conjecture that the resulting complexes of the Heisenberg generators are isomorphic to certain Clifford generators. 
In this paper we express the Heisenberg generators as some complexes of the Clifford generators.

The algebra $B$ also appears in studying stability of representation of symmetric groups \cite{CEF}.
They are interested in some right $B$-modules which are infinite dimensional and satisfy certain stability conditions.
We focus on finite dimensional left $B$-modules and their homological properties.
%The Grothendieck rings of $\ch$ and $\cl$ are conjectured to be isomorphic to $H$ and $Cl$, respectively. The difficulty of proving the conjectures is partially from the fact that the two categories are triangulated. For the same reason, we can only give a partial graphic description of $\ch$. On the other hand, $\cl$ has a better diagrammatic presentation, see \cite{Tian}.

\vspace{.2cm}
\n{\em Convention:} All modules are left modules unless specified otherwise.

\vspace{.2cm}
\n{\bf Acknowledgements.}
The author is grateful to Ko Honda for introducing me to the contact category.
Thank Mikhail Khovanov for valuable suggestions and encouragement.
Also thank Bangming Deng and Aaron Lauda for useful discussions.
The author is partially supported by NSFC grant 11601256.

\section{The bosonic Fock space}

We provide a modification of the symmetric group algebras as follows.

\begin{defn} \label{def B}
Define a $\K$-algebra $B$ by generators $1(n)$ for $n \ge 0$, $s_i(n)$ for $1 \le i \le n-1$, and $v_i(n)$ for $1 \le i \le n$, subject to the relations consisting of three groups:

\n(1)  Relation of idempotents:
\begin{align*}
1(n) 1(m)=\delta_{n,m}1(n), \quad 1(n) s_i(n) = s_i(n) 1(n) = s_i(n), \quad 1(n-1) v_i(n) = v_i(n) 1(n) = v_i(n).
\end{align*}

\n(2) Relation of symmetric groups:
\begin{align*}
s_i(n) s_i(n) =1(n),& \\ s_i(n)s_{i+1}(n)s_i(n)=s_{i+1}(n)s_i(n)s_{i+1}(n),& \\
 s_i(n) s_{j}(n) = s_{j}(n) s_i(n),& \qquad \mbox{if}~~~|i-j|>1.
\end{align*}

\n(3) Relation of short strands:
\begin{align*}
v_i(n) s_j(n) = s_{j}(n-1) v_i(n), & \qquad \mbox{if}~i > j+1, \\
v_i(n) s_j(n) = s_{j-1}(n-1) v_i(n), & \qquad \mbox{if}~i <j,  \\
v_i(n) v_{j}(n+1)=v_{j}(n) v_{i+1}(n+1), & \qquad \mbox{if}~i \ge j, \\
v_i(n) s_i(n) = v_{i+1}(n) . & \qquad \mbox{(Slide relation)}
\end{align*}
\end{defn}

\begin{figure}[h]
\begin{overpic}
[scale=0.4]{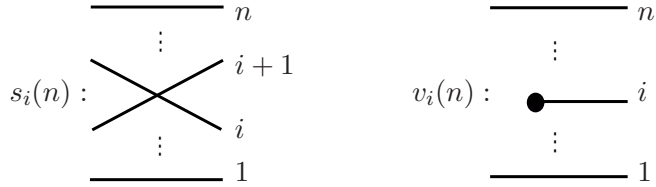}
\put(27,0){$1$}
\put(27,8){$i$}
\put(27,20){$i+1$}
\put(27,30){$n$}
\put(102,30){$n$}
\put(102,15){$i$}
\put(102,0){$1$}
\put(60,15){$v_i(n):$}
\put(-15,15){$s_i(n):$}
\end{overpic}
\caption{Generators of $B$.}
\label{s1}
\end{figure}

The algebra $B$ can be described diagrammatically, see figure \ref{s1}.
The idempotent $1(n)$ is denoted by $n$ horizontal strands.
In particular, $1(0)$ is denoted by the empty diagram.
The generator $s_i(n)$ is denoted by $n$ strands with a $(i,i+1)$ crossing.
The generator $v_i(n)$ is denoted by a diagram with $n-1$ horizontal strands and one short strand in the $i$-th position.
Here, the short strand has no endpoint on the left and one endpoint on the right.
The product $ab$ of two diagrams $a$ and $b$ is a horizontal concatenation of $a$ and $b$, where $a$ is on the left, $b$ is on the right.
The product is zero unless the number of their endpoints match.

\begin{rmk}
The short strand $v_1(1)$ corresponds to a contact structure on $(\mathbb{R} \times [0,1]) \times [0,1]$.
\end{rmk}

The relations of the second group are the defining relation of the symmetric groups.
The relations of the third group are about short strands.
The first three lines are isotopy relations of disjoint diagrams.
The last line says that the short strand can slide over the crossing.
We call it the {\em slide relation}.
It induces a relation
\begin{gather} \label{eq sym rel}
v_i(n)v_{i+1}(n+1)s_{i}(n+1)=v_i(n)v_{i+1}(n+1),
\end{gather}
which is called the {\em symmetric relation}.
In addition to the isotopy relations of disjoint diagrams, other local relations are drawn in figure \ref{s2}.

Let $a \bt b \in B$ denote a vertical concatenation of $a$ and $b$, where $a$ is at the bottom, $b$ is at the top.
The element $a \bt b$ does not depend on the relative positions of $a$ and $b$ by the isotopy relation of disjoint diagrams.
\begin{figure}[h]
\begin{overpic}
[scale=0.25]{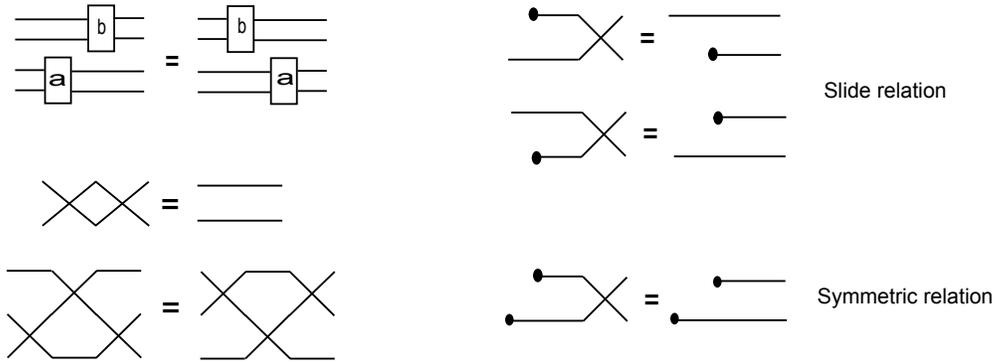}
\end{overpic}
\caption{Local relations of $B$.}
\label{s2}
\end{figure}

For $i<n$, generators $v_i(n)=v_n(n) g_i(n)$ for some $g_i(n) \in S(n)$ by the slide relation.
So the algebra $B$ is generated by $\K[S(n-1)]$ and $v_n(n)$ for $n \ge 1$.
Relations involving $v_n(n)$ are the isotopy relation with elements of symmetric groups, and the symmetric relation in (\ref{eq sym rel}) for $i=n$.
It is straightforward to verify the following equivalent description of $B$.
\begin{lemma} \label{lem def B}
The algebra $B$ is generated by $\K[S(n-1)]$ and $v_n(n)$ for $n \ge 1$, subject to the relations:
\begin{align*}
1(n-1) v_n(n) = v_n(n) 1(n) = v_n(n),\\
v_n(n) s_i(n) = s_{i}(n-1) v_n(n) ~~\mbox{for}~ i<n-1,\\
v_n(n)v_{n+1}(n+1)s_{n}(n+1)=v_n(n)v_{n+1}(n+1).
\end{align*}
\end{lemma}

The algebra $B$ is idempotented, i.e. has a complete system of mutually orthogonal idempotents $\{1(n)\}_{n\ge 0},$ so that
$$ B =  \bigoplus\limits_{m,n\ge 0} 1(m) B 1(n).$$
Let $\bmn$ denote its component $1(m) B 1(n)$.
It is spanned by diagrams with $m$ endpoints on the left and $n$ endpoints on the right.
The subalgebra $B(n,n)$ is naturally isomorphic to the symmetric group algebra $\K[S(n)]$.
The generator $v_i(n) \in B(n-1,n)$.
By definition, $\bmn=0$ if $n<m$.

There is a family of subalgebras
$$B(k)= \bigoplus\limits_{m,n \leq k} \bmn$$
of $B$ which is spanned by diagrams with at most $k$ endpoints on the right.
The algebra $B(k)$ is finite dimensional with the unit $\mb(k)'=\sum\limits_{0\leq n \leq k} 1(n)$.
There are the canonical nonunital inclusions $B(k) \subset B(k+1)$ which are compatible with the canonical nonunital inclusions $B(k) \subset B$.

\begin{lemma} \label{lem bmn}
For $0\le m \le n$, the vector space $\bmn$ is of dimension $\binom{n}{m}\cdot m!$.
As a left $\K[S(m)]$-module, it is free of rank $\binom{n}{m}$; as a right $\K[S(n)]$-module, it is isomorphic to the induction module $\op{Ind}_{n-m}^{n} \mb_{n-m}$, where $\mb_{n-m}$ is the one dimensional trivial representation of $\K[S(n-m)]$.
\end{lemma}
\begin{proof}
As a left $\K[S(m)]$-module, $\bmn$ is free with a basis of diagrams, each of which consists of $m$ horizontal strands, $n-m$ short strands, and no crossings.
There are $\binom{n}{m}$ such diagrams.
As a right $\K[S(n)]$-module, $\bmn$ is generated by any single element of the basis above.
The stabilizer of such an element is isomorphic to $\K[S(n-m)]$ by the symmetric relation, see figure \ref{s2}.
\end{proof}

For each partition $\mu \vdash n$, we fix a Young symmetrizer $e_{\mu} \in \K[S(n)] \cong B(n,n)$ throughout the paper.
We view $e_{\mu}$ as idempotents of $B(k)$ and $B$ under the inclusions $B(n,n) \subset B(k) \subset B$ for any $k \geq n$.
In particular, $e_{(0)} \in B$ is represented by the empty diagram for the partition $(0) \vdash 0$.
%The collection $\{e_{\mu}; ~ \mu \vdash n, n\leq k\}$ is a complete set of non-isomorphic primitive idempotents of $B_k$.
%The algebra $B_k$ is of finite global dimension.
Let $\Par$ denote the set of all partitions.
Define $$V_{\mu}=B \cdot e_{\mu}$$ as the left projective $B$-module.
Since each $1(n) \in \K[S(n)]$ which is a direct sum of matrix algebras determined by partitions, there is an isomorphism
$$B \cong \bigoplus\limits_{\mu \in \Par} V_{\mu}^{\oplus m(\mu)}$$
of left projective $B$-modules, where $m(\mu)$ is the multiplicity.
It follows that
$$V_P:=\bigoplus\limits_{\mu \in \Par} V_{\mu}$$
is a {\em projective generator} of the category $\{B-\op{mod}\}$ of left $B$-modules.
In particular, $B$ is Morita equivalent to $\End_{B}(V_P)^{op}$.

Let $L_{\mu}$ denote the simple $\K[S(n)]$-module for $\mu \vdash n$.
Let $M \bt N \in \{\K[S(m+n)]-\op{mod}\}$ be the external product of $M \in \{\K[S(m)]-\op{mod}\}$ and $N \in \{\K[S(n)]-\op{mod}\}$.

\begin{lemma} \label{lem hom B}
For $\lambda \vdash m, \mu \vdash n$, there is a natural isomorphism
$$\Hom_{B}(V_{\lambda}, V_{\mu}) \cong \Hom_{\K[S(n)]}(L_{\lambda}\bt \mb_{n-m}, L_{\mu}),$$
if $m \le n$.
If $m>n$, the space $\Hom_{B}(V_{\lambda}, V_{\mu})$ is zero.
Moreover, the morphism space is one dimensional if $\mu$ is obtained from $\lambda$ by adding at most one box on each column; it is zero otherwise.
\end{lemma}
\begin{proof}
It follows from the natural isomorphisms:
$$\Hom_{B}(V_{\lambda}, V_{\mu}) \cong e_{\lambda}Be_{\mu} \cong (e_{\lambda} \bt e_{(n-m)}) Be_{\mu} \cong \Hom_{\K[S(n)]}(L_{\lambda}\bt \mb_{n-m}, L_{\mu}),$$
where $e_{(n-m)}$ is the idempotent for the partition $(n-m)$ of $n-m$ corresponding to the trivial representation $\mb_{n-m}$ of $\K[S(n-m)]$.
\end{proof}

Define a quiver $\Gamma$ as follows.
The set of vertices of $\Gamma$ is $\Par$, the set of all partitions.
There is an arrow $\lm \ra \mu$ if $\lambda \subset \mu$ and $|\mu|=|\lm|+1$.
We write $\lm \ra \mu \oplus \nu \ra \eta$ if there is a square
\begin{gather} \label{eq def lmmunueta}
\xymatrix{
\lm \ar[r]\ar[d]   &  \mu \ar[d]     \\
  \nu \ar[r]       & \eta
}
\end{gather}
Let $\K \Gamma$ denote the path algebra of $\Gamma$, $(\lambda) \in \K \Gamma$ denote the idempotent, and $(\lambda|\mu) \in \K\Gamma$ denote the generator for each arrow $\lambda \ra \mu$.

\begin{defn} \label{def F}
Define a $\K$-algebra $F$ as a quotient of $\K \Gamma$, modulo the relations
\begin{gather*}
(\lambda | \mu) (\mu | \eta)=(\lambda | \nu) (\nu | \eta) ~~\mbox{if}~ \lm \ra \mu\oplus\nu \ra \eta;\\
(\lambda | \mu) (\mu | \eta)=0 ~~\mbox{if}~ \eta \backslash \lm = (1^2).
\end{gather*}
\end{defn}

%A part of the quiver $\Gamma$ with the relations is the following:
%$$
%\xymatrix{
%            &           &                  &(3) \\
%            &                    &    (2) \ar[dr]\ar[ur] & \\
%(0)\ar[r]\ar@/_2pc/[drr]_{= 0}   & (1) \ar[ur]\ar[dr] \ar@/_2pc/[ddrr]_{= 0}&                & (2,1) \\
%            &                    &   (1^2) \ar[ur]\ar[dr]& \\
%             &            &                & (1^3)
%}$$

The $\K$-vector space $(\lm)F(\mu)$ is at most one dimensional by the first relation in Definition \ref{def F}.
The second relation implies that $(\lm)F(\mu)$ is one dimensional if and only if $\lm \subset \mu$ and $\mu$ is obtained from $\lm$ by adding at most one box on each column.
We summarize this in the following lemma.
\begin{lemma} \label{lem lmRmu}
For $\lm=(\lm_i), \mu=(\mu_i)$, the space $(\lm)F(\mu)$ is one dimensional if $\mu_i \geq \lm_i \geq \mu_{i+1}$ for all $i \geq 1$; otherwise, the space is zero.
\end{lemma}

Let $\lm \ra \lm^1 \ra \dots \ra \lm^k \ra \mu$ denote any path from $\lm$ to $\mu$ in $\Gamma$.
When $(\lm)F(\mu)$ is nonzero, define
\begin{gather} \label{eq gen R}
(\lm || \mu)=(\lm|\lm^1)(\lm^1|\lm^2)\cdots(\lm^k|\mu) \in (\lm)F(\mu),
\end{gather}
which is independent of the choices of pathes.
It is a generator of $(\lm)F(\mu)$.
Note that $(\lm || \mu)=(\lm)$ if $\lm=\mu$, and $(\lm || \mu)=(\lm|\mu)$ if $\lm \ra \mu$.

\begin{lemma} \label{lem R}
The algebra $F$ has a $\K$-basis $\{(\lm || \mu)\}$, and the multiplication is
$(\lambda || \mu) (\mu || \eta)=(\lm|| \eta)$ when they all exist.
\end{lemma}

Lemmas \ref{lem hom B} and \ref{lem lmRmu} implies that $\End_B(V_P)^{op}$ is naturally isomorphic to $F$ as vector spaces.
\begin{prop} \label{prop Morita BF}
The algebras $\End_B(V_P)^{op}$ and $F$ are naturally isomorphic. In particular, $B$ is Morita equivalent to $F$.
\end{prop}
\begin{proof}
It is enough to show that we can choose a generator $b_{\lambda \mu}$ of $\Hom_{B}(V_{\lambda}, V_{\mu})$ for each $\lambda \ra \mu$ such that
\begin{align} \label{eq iso B F}
b_{\mu\eta} \circ b_{\lambda \mu}=b_{\nu\eta} \circ b_{\lambda \nu} \in \Hom_{B}(V_{\lambda}, V_{\eta}), \qquad \mbox{if}~~\lm \ra \mu \oplus \nu \ra \eta.
\end{align}
The second relation in Definition \ref{def F} is always preserved since $\Hom_{B}(V_{\lambda}, V_{\eta})=0$ if $\eta \backslash \lm = (1^2)$.
We prove (\ref{eq iso B F}) in the rest of this section.
\end{proof}

For $\lambda \ra \mu$, Lemma \ref{lem hom B} implies that
\begin{gather} \label{eq iso B F'}
\Hom_{B}(V_{\lambda}, V_{\mu}) \cong \Hom_{\K[S(n)]}(L_{\lambda}\bt \mb_{1}, L_{\mu}) \cong \Hom_{\K[S(n-1)]}(L_{\lambda}, \op{Res}_{n-1}^{n}L_{\mu}).
\end{gather}
Recall the following description of representations of symmetric groups from Vershik and Okounkov \cite{VO}.
For $\mu \vdash n$, $L_{\mu}$ has a $\K$-basis $\{v_{T};~T ~\mbox{Young tableaux of}~ \mu\}$.
They are common eigenvectors of the {\em Jucys-Murphy elements}.
The collection $\alpha(T)=(a_1, \dots, a_n) \in \Z^n$ of eigenvalues $a_i$ which are called {\em contents} of $T$.
If $\alpha(T')=(a_1, \dots, a_{i+1}, a_{i}, \dots, a_n)$, then we write $T'=s_i T$, where $s_i \in S(n)$ is the transposition of $i$ and $i+1$.
For each $\mu$, there is a special Young tableaux $T^{\mu}$ where $1$ through $\mu_1$ are in the first row, $\mu_1+1$ to $\mu_1+\mu_2$ are in the second row, and so on.
Each $T$ can be expressed as $s\cdot T^{\mu}$ for some $s \in S_n$.
Let $l(T)=l(s)$ denote the minimal number of transpositions needed to obtain $T$ from $T^{\mu}$.

The main result we need is \cite[Proposition 7.1]{VO} of Vershik and Okounkov.
If $T'=s_i T$ and $l(T')>l(T)$, then there exists a basis $\{v_{T}\}$ of $L_{\mu}$ such that
\begin{align*}
s_i \cdot v_{T} &= v_{T'} + \frac{1}{a_{i+1}-a_{i}} v_{T},\\
s_i \cdot v_{T'}&= (1-\frac{1}{(a_{i+1}-a_{i})^2})v_{T} - \frac{1}{a_{i+1}-a_{i}} v_{T'}.
\end{align*}

We rescale the basis $\{v_{T}\}$ by setting $\wt{v}_T=c_T v_T$, where the numbers $c_T \in \mathbb{Q} \subset \K$ are defined by induction on $l(T)$ as
$$c_{T'}=\frac{a_{i+1}-a_{i}}{a_{i+1}-a_{i}-1} c_T, \quad \mbox{if}~ T'=s_iT, ~l(T')>l(T); \qquad c_{T^{\mu}}=1.$$
When $T'=s_is_{i+1}s_iT=s_{i+1}s_is_{i+1}T$, i.e. there are two pathes from $T$ to $T'$, and $\alpha(T)=(\dots, a_i, a_{i+1}, a_{i+2}, \dots)$, we have
$$c_{T'}=\frac{a_{i+1}-a_{i}}{a_{i+1}-a_{i}-1}\cdot \frac{a_{i+2}-a_{i+1}}{a_{i+2}-a_{i+1}-1}\cdot \frac{a_{i+2}-a_{i}}{a_{i+2}-a_{i}-1} c_T$$
which is independent of choices of pathes.
Similarly, when $T'=s_is_jT=s_js_iT$ for $|i-j|>1$, $c_{T'}$ does not depend on choices of pathes from $T$ to $T'$.
Therefore, the collection $\{c_T\}$ is well-defined.

In terms of the rescaled basis $\{\wt{v}_{T}\}$, the action of $s_i$ is given by
\begin{align} \label{eq VO rescale}
s_i \cdot \wt{v}_{T} &= \frac{1}{a_{i+1}-a_{i}} \wt{v}_{T}+\frac{a_{i+1}-a_{i}-1}{a_{i+1}-a_{i}}v_{T'},
\end{align}
\begin{align} \label{eq VO rescale'}
s_i \cdot \wt{v}_{T'}&= \frac{a_{i+1}-a_{i}+1}{a_{i+1}-a_{i}}\wt{v}_{T} - \frac{1}{a_{i+1}-a_{i}} \wt{v}_{T'}.
\end{align}
%In particular, $s_i \cdot (\wt{v}_{T}-\wt{v}_{T'})=-\wt{v}_{T}+\wt{v}_{T'}$ so that
%\begin{gather} \label{eq VO si}
%(1+s_i) \cdot \wt{v}_{T}=(1+s_i) \cdot \wt{v}_{T'}.
%\end{gather}
The expression in terms of the rescaled basis still hold if $l(T')<l(T)$.

\begin{defn} \label{def choice f}
For $\lambda \ra \mu$, define $f_{\lambda\mu} \in \Hom_{\K[S_{n-1}]}(L_{\lambda}, \op{Res}_{n-1}^{n}L_{\mu})$ on the rescaled basis by
\begin{gather}
f_{\lambda\mu}(\wt{v}_{T})=\wt{v}_{T\cup\{n\}},
\end{gather}
where $T\cup\{n\}$ is the tableaux of $\mu$ obtained from $T$ of $\lambda$ by filling $n$ to the extra box.
\end{defn}

For $\lm \ra \mu \oplus \nu \ra \eta$, $\lm \vdash n-1, \mu,\nu \vdash n$ and $\eta \vdash n+1$,
$$f_{\mu\eta} \circ f_{\lambda \mu}(\wt{v}_{T})=\wt{v}_{T_1}, \qquad
f_{\nu\eta} \circ f_{\lambda \nu}(\wt{v}_{T})=\wt{v}_{T_2},$$
such that $T_2=s_nT_1$.
%Therefore, $(1+s_n) \cdot \wt{v}_{T_1}=(1+s_n) \cdot \wt{v}_{T_2}$ by (\ref{eq VO si}).
It follows from (\ref{eq VO rescale}, \ref{eq VO rescale'}) that
\begin{gather} \label{eq sn f}
s_n \cdot (f_{\mu\eta} \circ f_{\lm\mu})=\frac{1}{d}~(f_{\mu\eta} \circ f_{\lm\mu})+\frac{d-1}{d}~(f_{\nu\eta} \circ f_{\lm\nu}),
\end{gather}
\begin{gather} \label{eq sn f'}
s_n \cdot (f_{\nu\eta} \circ f_{\lm\nu})=\frac{d+1}{d}~(f_{\mu\eta} \circ f_{\lm\mu})+\frac{-1}{d}~(f_{\nu\eta} \circ f_{\lm\nu}),
\end{gather}
where $d=a_{n+1}-a_{n}$, and $\alpha(T_1)=(\dots, a_n , a_{n+1})$.
In particular,
$$(1+s_n)\cdot (f_{\mu\eta} \circ f_{\lambda \mu})=(1+s_n)\cdot (f_{\nu\eta} \circ f_{\lambda \nu})$$
in $\Hom_{\K[S(n-1)]}(L_{\lambda}, \op{Res}_{n-1}^{n+1}L_{\eta}) \cong \Hom_{\K[S(n+1)]}(L_{\lambda} \bt \K[S_2], L_{\eta})$.
Since $(1+s_n)$ is a projector onto the summand $\mb_2 \subset \K[S(2)]$, we have
\begin{gather} \label{eq iso B F''}
\op{pr}(f_{\mu\eta} \circ f_{\lambda \mu})=\op{pr}(f_{\nu\eta} \circ f_{\lambda \nu}) \in \Hom_{\K[S(n+1)]}(L_{\lambda} \bt \mb_2, L_{\eta}).
\end{gather}

The map $f_{\lm\mu}$ can be identified with an element of $(e_{\lm}\bt 1(1)) \K[S(n)] e_{\mu} \subset \K[S(n)]$.
Define
\begin{gather} \label{eq def b}
b_{\lm\mu}=v_n(n)\cdot f_{\lm\mu} \in B.
\end{gather}
It is a generator of the one dimensional space $e_{\lm}B e_{\mu}$.
Figure \ref{se1} describes the generators $f_{\lm\mu} \in \K[S(n)]$ and $b_{\lm\mu} \in B$, where a box with $\lm$ indside denotes the idempotent $e_{\lm}$, and a horizontal strand labeled $n$ denote $n$ horizontal strands $1(n)$.

The element $b_{\lm\mu}$ induces a generator of $\Hom_{B}(V_{\lambda}, V_{\mu})$, which is still denoted by $b_{\lm\mu}$.
The map $b_{\lm\mu}$ is the image of the map $f_{\lambda \mu}$ under the isomorphism (\ref{eq iso B F'}).
The equation (\ref{eq iso B F''}) implies (\ref{eq iso B F}) under the isomorphism $\Hom_{B}(V_{\lambda}, V_{\eta}) \cong \Hom_{\K[S(n+1)]}(L_{\lambda} \bt \mb_2, L_{\eta}).$
We complete the proof of Proposition \ref{prop Morita BF}.

\begin{figure}[h]
\begin{overpic}
[scale=0.3]{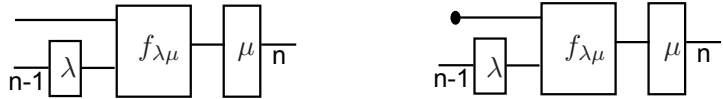}
\put(7,3){$\lm$}
\put(18,6){$f_{\lm\mu}$}
\put(32,6){$\mu$}
\put(67,3){$\lm$}
\put(78,6){$f_{\lm\mu}$}
\put(92,6){$\mu$}
\end{overpic}
\caption{The generators $f_{\lm\mu}$ and $b_{\lm\mu}$ on the left and right, respectively.}
\label{se1}
\end{figure}

%Let $\Kom(B_k)$ denote the homotopy category of finitely generated projective $B_k$-modules. There is a family of inclusions of triangulated categories $i_k: \Kom(B_k) \ra \Kom(B_{k+1})$ which take $V_{\mu,k}$ to $V_{\mu,k+1}$ for $\mu \vdash n, n\leq k$. There is a family of inclusions of triangulated categories $j_k: \Kom(B_k) \ra \Kom(B)$ which take $V_{\mu,k}$ to $V_{\mu}$. Define $\Kom(B)$ as the smallest full triangulated subcategory of $\Kom(B)$ containing $\{V_{\mu}; \mu \vdash n, n\geq 0\}$.
\vspace{.2cm}
According to \cite[Lemma 2.4]{ASS}, $\{1_{\mu}~;~ \mu \in \Par\}$ is a complete set of nonisomorphic primitive orthogonal idempotents of the quiver algebra $F$.
Therefore, $\{F\cdot 1_{\mu}\}$ is a complete set of nonisomorphic finite dimensional projective $F$-module.
By the Morita equivalence of $B$ and $F$, $\{V_{\mu}=B\cdot 1_{\mu}\}$ is a complete set of nonisomorphic finite dimensional projective $B$-module.
Let $\Kom(B)$ denote the homotopy category of finite dimensional projective $B$-modules.
Let $K_0(B)$ denote its Grothendieck group.

%\begin{lemma} There is an equivalence of categories $\Kom(B) \cong \varinjlim \Kom(B_k)$ with respect to $i_k$. \end{lemma}

\begin{prop} \label{prop K0 vb}
There is an isomorphism of abelian groups $K_0(B) \cong V_B$ which maps $[V_{\mu}]$ to the Schur polynomial associated to $\mu$.
\end{prop}

\begin{rmk}
There exists a filtration $\{F_n\}$ of $F$ such that the left global dimension of $F_n$ is $n$, see (\ref{eq Fn filtration}) and Lemma \ref{lem sn finite}.
Each finite dimensional $F$-module is an $F_n$-module for sufficiently large $n$ so that it has a finite projective resolution.
As a result, $\Kom(F)$ is equivalent to the derived category of finite dimensional $F$-modules.
The analogue for $\Kom(B)$ holds as well.
\end{rmk}

\section{The Heisenberg algebra} \label{sec H}
We consider endofunctors of $\Kom(B)$ which are given by tensoring with $B$-bimodules.
In this section, we write $\ot$ for $\ot_{B}$.

\subsection{The Heisenberg category $\ch$}
Define a linear map
$$\begin{array}{cccc}
\rho: & B & \ra & B \\
 & a & \mapsto & a \bt 1(1),
\end{array}$$
which is given by adding a horizontal strand on the top of any diagram $a \in B$.
The map $\rho: B \ra B$ is an inclusion of nonunital algebras.
In particular, $\rho: B(m,n) \ra B(m+1,n+1)$.
Restricted to the unital subalgebras $B(k)$, the induced inclusion $\rho: B(k) \ra B(k+1)$ is not unital.

Consider two $B$-bimodules corresponding to the induction and restriction functors of $\Kom(B)$ with respect to the inclusion $\rho$.

\begin{defn}
Define a $\K$-vector space $P= \bigoplus\limits_{m \geq 0, n \geq 1} B(m,n)$ with the $B$-bimodule structure given by
$a\cdot v \cdot b=a~v~\rho(b)$, where $a~v~\rho(b)$ is the multiplication in $B$, for $a,b \in B$ and $v \in P$.

Dually, define a $\K$-vector space $Q= \bigoplus\limits_{m,n \geq 1} B(m,n)$ with the $B$-bimodule structure given by
$a\cdot v \cdot b=\rho(a)~v~b$, where $\rho(a)~v~b$ is the multiplication in $B$, for $a,b \in B$ and $v \in Q$.
\end{defn}

See figure \ref{s3} for a diagrammatic description of $P$ and $Q$.
Since the top strand of $P$ is unchanged under the right multiplication, we call it the frozen strand of $P$ and add a little bar at its right end.
Similarly, we call the top strand of $Q$ the frozen strand of $Q$ and add a little bar at its left end.

\begin{figure}[h]
\begin{overpic}
[scale=0.25]{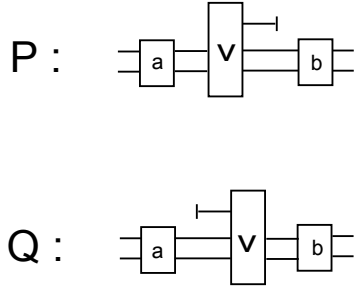}
\end{overpic}
\caption{The induction and restriction bimodules $P$ and $Q$.}
\label{s3}
\end{figure}

For a partition $\mu$, let
$$\Ind(\mu)=\{\eta ~|~ \eta \supset \mu, |\eta|=|\mu|+1\}, \quad \Res(\mu)=\{\lambda ~|~ \lambda \subset \mu, |\lambda|=|\mu|-1\}.$$
The endofunctors of tensoring with $P$ and $Q$ are the analogue of the induction and restriction functors on the representations of the symmetric groups.

\begin{lemma} \label{lem tensor PQ}
(1) $P \ot V_{\mu} \cong \bigoplus\limits_{\eta \in \Ind(\mu)}V_{\eta}$.

\n (2) $Q \ot V_{\mu} \cong \bigoplus\limits_{\lambda \in \Res(\mu)}V_{\lambda}$; in particular $Q \ot V_{(0)} = 0$.
\end{lemma}
\begin{proof}
The projective $V_{\mu}$ is determined by the idempotent $e_{\mu} \in \K[S(n)]$.
The representation theory of symmetric groups tells that $e_{\mu} \bt \mb_1=\sum\limits_{\eta \in \Ind(\mu)}e'_{\eta}$, where $e'_{\eta}$ is isomorphic to $e_{\eta}$.
This implies (1). The proof for (2) is similar.
\end{proof}

For the later use, we also consider tensoring $P$ and $Q$ with right $B$-modules.
Let $\wt{V}_{\mu}=e_{\mu}\cdot B$ denote the right projective $B$-module.
Note that any $\wt{V}_{\mu}$ is infinite dimensional.
\begin{lemma} \label{lem tensor PQ rt}
(1) $ \wt{V}_{\mu} \ot P \cong \wt{V}_{\mu} \oplus \bigoplus\limits_{\lambda \in \Res(\mu)}\wt{V}_{\lambda}$; in particular $\wt{V}_{(0)} \ot P \cong \wt{V}_{(0)}$.

\n (2) $\wt{V}_{\mu} \ot Q \cong \bigoplus\limits_{\eta \in \Ind(\mu)}\wt{V}_{\eta}$.
\end{lemma}
\begin{proof}
When tensoring with the right $B$-module, $Q$ is the induction bimodule which implies (2).
But $P$ is not exactly the restriction bimodule since the short strand is asymmetric with respect to left and right.
The tensor product $ \wt{V}_{\mu} \ot P=W \oplus U$, where $W$ is the right submodule spanned by diagrams with the frozen strand being the short strand, and $U$ is the complement.
Then $W \cong \wt{V}_{\mu}$, and $U \cong \bigoplus\limits_{\lambda \in \Res(\mu)}\wt{V}_{\lambda}$.
\end{proof}

\begin{lemma} \label{lem proj}
The $B$-bimodules $P$ and $Q$ are projective as both left and right $B$-modules.
\end{lemma}
\begin{proof}
The free left $B$-module $B$ is isomorphic to $\bigoplus\limits_{\mu}V_{\mu}^{\oplus m(\mu)}$, where $m(\mu)$ is the multiplicity.
So $$P \cong P \ot B \cong P \ot (\bigoplus\limits_{\mu}V_{\mu}^{\oplus m(\mu)}) \cong \bigoplus\limits_{\mu}(P \ot V_{\mu})^{\oplus m(\mu)}$$
which is left projective by Lemma \ref{lem tensor PQ}.
The same argument shows that $Q$ is left projective.

Similarly, it follows from Lemma \ref{lem tensor PQ rt} that $P$ and $Q$ are right projective.
\end{proof}

Let $D(B^e)$ denote the derived category of $B$-bimodules.
The derived tensor product gives a monoidal structure on it.
Any $B$-bimodule is viewed as an object of $D(B^e)$ by placing it in the cohomological degree zero.
The unit object is isomorphic to the $B$-bimodule $B$.

\begin{defn} \label{def H}
Define $\ch$ as the Karoubi envelope of the smallest full monoidal triangulated subcategory of $D(B^e)$ which contains $B,P,Q$.
\end{defn}

Since $B, P, Q$ are all left projective and right projective, the derived tensor product between them is equivalent to the ordinary tensor product.
Therefore, the monoidal structure on $\ch$ can be calculated using the ordinary tensor product.
We will use $MN$ to denote $M \ot N$ for objects $M, N \in \ch$, and simply write $\Hom$ for $\Hom_{\ch}$.

Since $\Kom(B)$ is generated by $V_{\mu}$'s, Lemma \ref{lem tensor PQ} implies that tensoring with $P$ and $Q$ induces endofunctors of $\Kom(B)$.
Thus, the category $\ch$ acts on $\Kom(B)$ by tensoring with bounded complexes of $B$-bimodules in $\ch$.

\vspace{.2cm}
The Heisenberg relation holds in $\ch$.
The proof is due to Khovanov \cite{Kh2}.
We rewrite it in terms of diagrams.
A diagram in $QP$ is called type (1) if the frozen strand of $Q$ connects to the frozen strand of $P$; otherwise, it is called type (2).
Any diagram is either of type (1) or type (2) since the short strand can slide over the crossing.
Let $X$ and $Y$ be the subspaces of $QP$ spanned by diagrams of type (1) and (2), respectively.
Since the left and right multiplication do not change the frozen strand, both $X$ and $Y$ are sub-bimodules of $QP$.
There is a decomposition $QP \cong X \oplus Y$ of $B$-bimodules.
There are natural isomorphisms of $B$-bimodules: $X \cong B$ and $Y \cong PQ$, see figure \ref{s4}.

\begin{figure}[h]
\begin{overpic}
[scale=0.25]{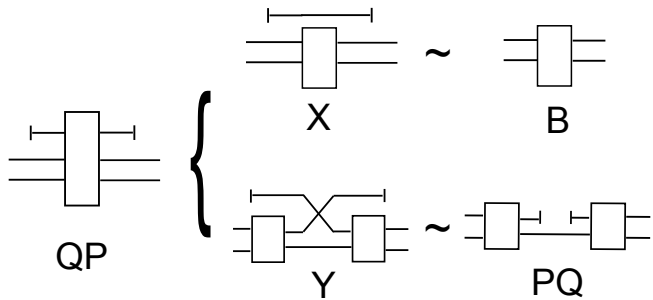}
\end{overpic}
\caption{The isomorphism $QP \cong B ~\oplus ~ PQ$.}
\label{s4}
\end{figure}

\begin{prop} \label{prop H relation}
There is a canonical isomorphism
\begin{gather} \label{eq H relation}
QP \cong B ~\oplus ~ PQ \in \ch.
\end{gather}
\end{prop}

We can partially describe $\ch$ using certain diagrams as in the literature.
Following \cite{Kh1}, we denote $B$ by the empty diagram, $P$ and $Q$ by an upward and downward vertical arrows, respectively.
Certain morphisms in $\ch$ can be described via diagrams with prescribed boundary conditions on the top and bottom.
The composition of these morphism is given by the vertical concatenation.

The elementary morphisms are three U-turn maps and four crossing maps:
\be
\item a U-turn map $\op{cup}_{QP} \in \Hom(B, QP)$ given by the inclusion;

\item a U-turn map $\op{cap}_{QP} \in \Hom(QP, B)$ given by the projection;

\item a crossing map $\op{cr}_{QP} \in \Hom(PQ, QP)$ given by the inclusion;

\item a crossing map $\op{cr}_{PQ} \in \Hom(QP, PQ)$ given by the projection.

\item a U-turn map $\op{cap}_{PQ} \in \Hom(PQ, B)$ given by connecting the frozen strands in $P$ and $Q$;

\item a crossing map $\op{cr}_{PP} \in \Hom(PP, PP)$ given by adding a crossing between the two frozen strands in $PP$;

\item a crossing map $\op{cr}_{QQ} \in \Hom(QQ, QQ)$ given by adding a crossing between the two frozen strands in $QQ$.
\ee
The first four maps are induced by two inclusions and two projections with respect to the canonical Heisenberg isomorphism (\ref{eq H relation}).
The remaining three maps are described in figure \ref{s5}.

\begin{figure}[h]
\begin{overpic}
[scale=0.25]{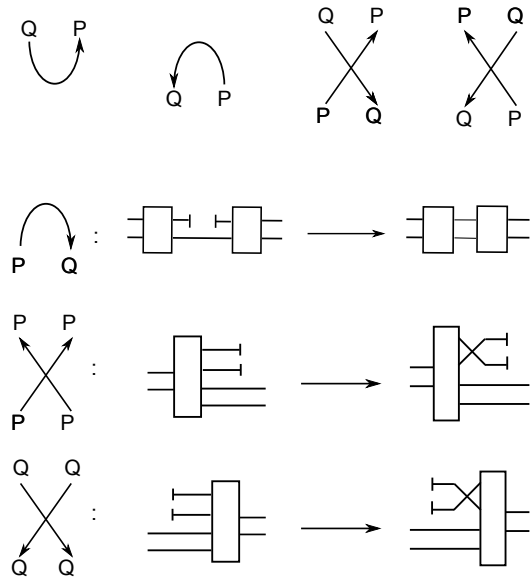}
\end{overpic}
\caption{Some elementary morphisms in $\ch$.}
\label{s5}
\end{figure}

\vspace{.2cm}
One of the main differences between $\ch$ and Khovanov's Heisenberg category is that there is no U-turn map in $\Hom(B, PQ)$.
\begin{lemma}
The morphism space $\Hom(B, PQ)=0$.
\end{lemma}
\begin{proof}
Suppose $f \in \Hom(B, PQ)$. Since the $B$-bimodule $B$ is generated by $1(n)$'s for $n \ge 0$, $f$ is determined by images $f(1(n)) \in 1(n)\cdot PQ \cdot 1(n)$.
We prove that $f(1(n))=0$ by induction on $n$ as follows.
For $n=0$, it is true since $Q\cdot 1(0)=0$.
Suppose $f(1(n-1))=0$.
For the short strand, we have $f(v_n(n))=f(1(n-1)v_n(n))=0$.
On the other hand, $f(v_n(n))=f(v_n(n)1(n))=v_n(n)f(1(n))$.
Left multiplication by $v_n(n)$ induces a map $1(n)\cdot PQ \cdot 1(n) \ra 1(n-1) \cdot PQ \cdot 1(n)$.
This linear map is injective which implies that $f(1(n))=0$.
\end{proof}

\begin{figure}[h]
\begin{overpic}
[scale=0.25]{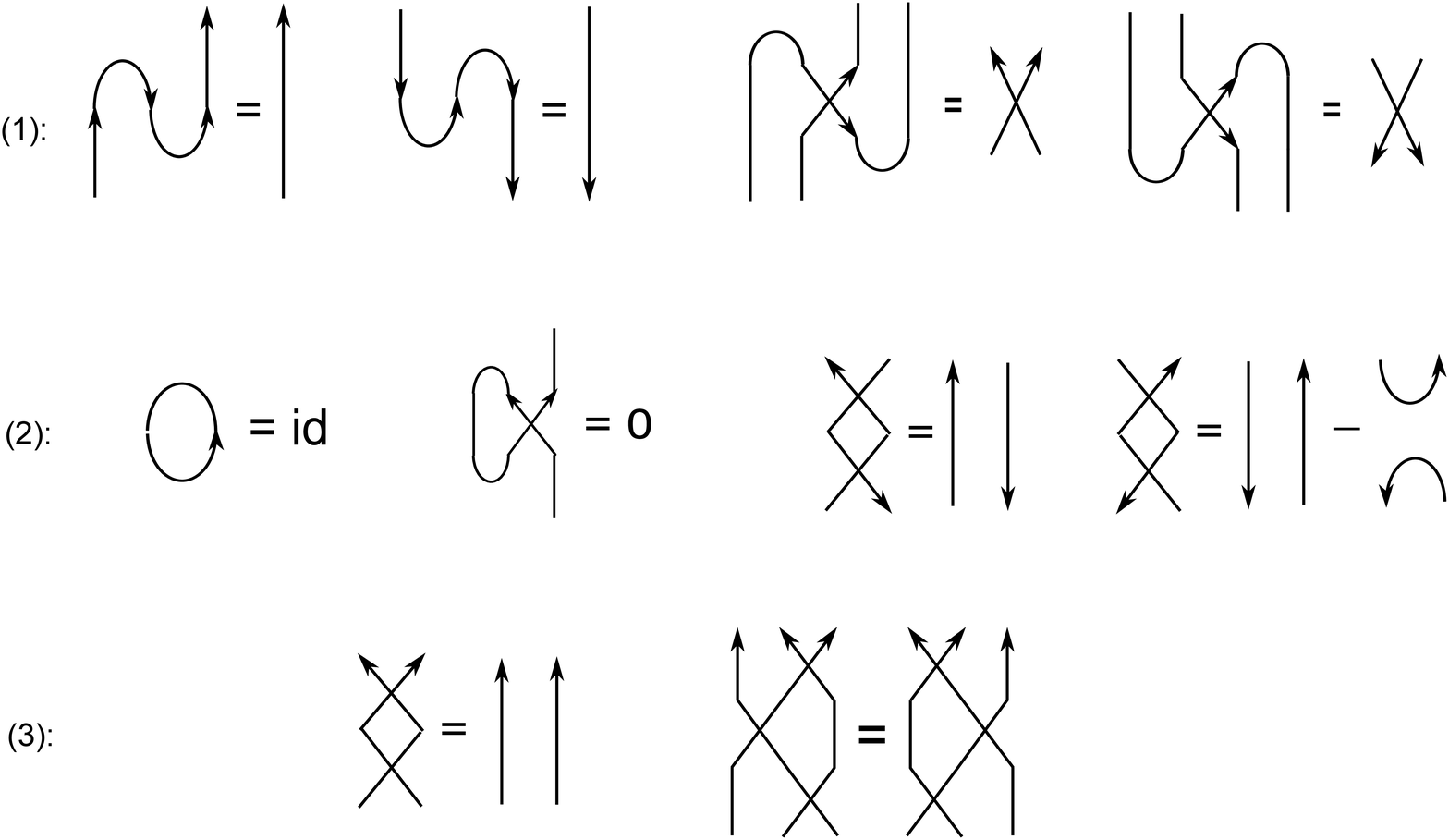}
\end{overpic}
\caption{Some local relations in $\ch$.}
\label{s6}
\end{figure}

Some local relations in $\ch$ can be described diagrammatically, see figure \ref{s6}.
The proofs of the local relations are the same as in \cite{Kh1}, and we leave them to the reader.
Relations (1) are the isotopy relations.
Note that the zigzag relation only holds for one direction since there is no U-turn map in $\Hom(B,PQ)$.
Relations (2) are essentially from the Heisenberg relation (\ref{eq H relation}).
Relations (3) imply that $\K[S(n)] \subset \End(P^n)$.
Since $\ch$ is Karoubi closed, we obtain an object $P^{(\mu)} \in \ch$ as a direct summand of $P^n$ for each partition $\mu \vdash n$. Similarly, we have $Q^{(\mu)} \in \ch$.

\begin{prop}[Khovanov \cite{Kh2}] \label{prop Kh}
The following relations hold in $\ch$:

\n(1) $P^{(\mu)}P^{(\lm)} \cong P^{(\lm)}P^{(\mu)}, Q^{(\mu)}Q^{(\lm)} \cong Q^{(\lm)}Q^{(\mu)}$ for any partitions $\lm, \mu$.

\n(2) $Q^{(n)}P^{(m)} \cong \bigoplus\limits_{k\ge0}P^{(m-k)}Q^{(n-k)}$, and $Q^{(1^n)}P^{(1^m)} \cong \bigoplus\limits_{k\ge0}P^{(1^{m-k})}Q^{(1^{n-k})}$.

\n(3) $Q^{(n)}P^{(1^m)} \cong P^{(1^m)}Q^{(n)} \oplus P^{(1^{m-1})}Q^{(n-1)}$, and $Q^{(1^n)}P^{(m)} \cong P^{(m)}Q^{(1^n)} \oplus P^{(m-1)}Q^{(1^{n-1})}$.
\end{prop}
\begin{proof}
Notice that no U-turn map in $\Hom(B,PQ)$ is needed in the proof of \cite[Proposition 1]{Kh2}.
The same proof applies here.
\end{proof}

Define the Heisenberg algebra $H$ with generators $p^{(m)}, q^{(m)}$ for $m \ge 0$, and relations
\begin{gather} \label{eq rel H}
p^{(n)}p^{(m)}=p^{(m)}p^{(n)}, \qquad q^{(n)}q^{(m)}=q^{(m)}q^{(n)}, \qquad q^{(n)}p^{(m)}=\sum\limits_{k\ge0}p^{(m-k)}q^{(m-k)}.
\end{gather}
The proposition implies that there is a ring homomorphism $\gamma: H \ra K_0(\ch)$.
This map is injective by considering the categorical action of $\ch$ on $\Kom(B)$.
We conjecture that this map is an isomorphism.

\subsection{New morphisms in $\ch$}
In the rest of Section \ref{sec H}, we discuss some new ingredients in $\ch$ which do not appear in Khovanov's Heisenberg category.

The isotopy relations (1) in figure \ref{s6} implies that $P$ is left adjoint to $Q$ in $\ch$.
Since there is no U-turn map in $\Hom(B,PQ)$, $P$ is not right adjoint to $Q$.
To find an object which is right adjoint to $Q$, we need more morphisms in $\ch$.
There is a $B$-bimodule homomorphism $f'_0 \in \Hom(B, P)$ defined by $f'_0(a)=a \bt v_1(1)$.
Diagrammatically, $f'_0$ adds a short strand as the frozen strand in $P$ on the top of any diagram $a \in B$.

Recall $e_{(1^n)}=\frac{1}{n!}\sum\limits_{g \in S(n)}~(-1)^{\op{sign}(g)}g$
is the idempotent associated to the partition $(1^n)$.

\begin{defn} \label{def fn}
For $n \geq 0$, define a family of $B$-bimodule homomorphisms $f_n \in \Hom(B, P)$ by setting $f_n(1(m))$ as
\begin{align*}
 \left\{
\begin{array}{rl}
0, & ~\mbox{if}~ m<n, \\
\frac{n+1}{k!}\sum\limits_{\scriptscriptstyle g \in S(m)}v_{m+1}(m+1) (g \bt 1(1)) (1(k) \bt e_{(1^{n+1})}) (g^{-1} \bt 1(1)), & ~\mbox{if}~ m=n+k, k\ge 0,
\end{array}\right.
\end{align*}
and extending $f_n(a\cdot 1(m) \cdot b)=a\cdot f_n(1(m)) \cdot b$, for $a,b \in B$.
\end{defn}

For $n=0$, $f_0(1(m))=v_{m+1}(m+1)=1(m) \bt v_1(1)$ which agrees with $f'_0(1(m))$.
For $m=n$, $$f_n(1(n))=(n+1)!~v_{n+1}(n+1) ~e_{(1^{n+1})}.$$
We express $f_n$ in figure \ref{s7}, where a horizontal strand with a label $n$ denotes $n$ horizontal strands $1(n)$, a horizontal strand without a label denotes one horizontal strand $1(1)$, and a box with $1^{n+1}$ inside to denote $e_{(1^{n+1})}$.

\begin{figure}[h]
\begin{overpic}
[scale=0.25]{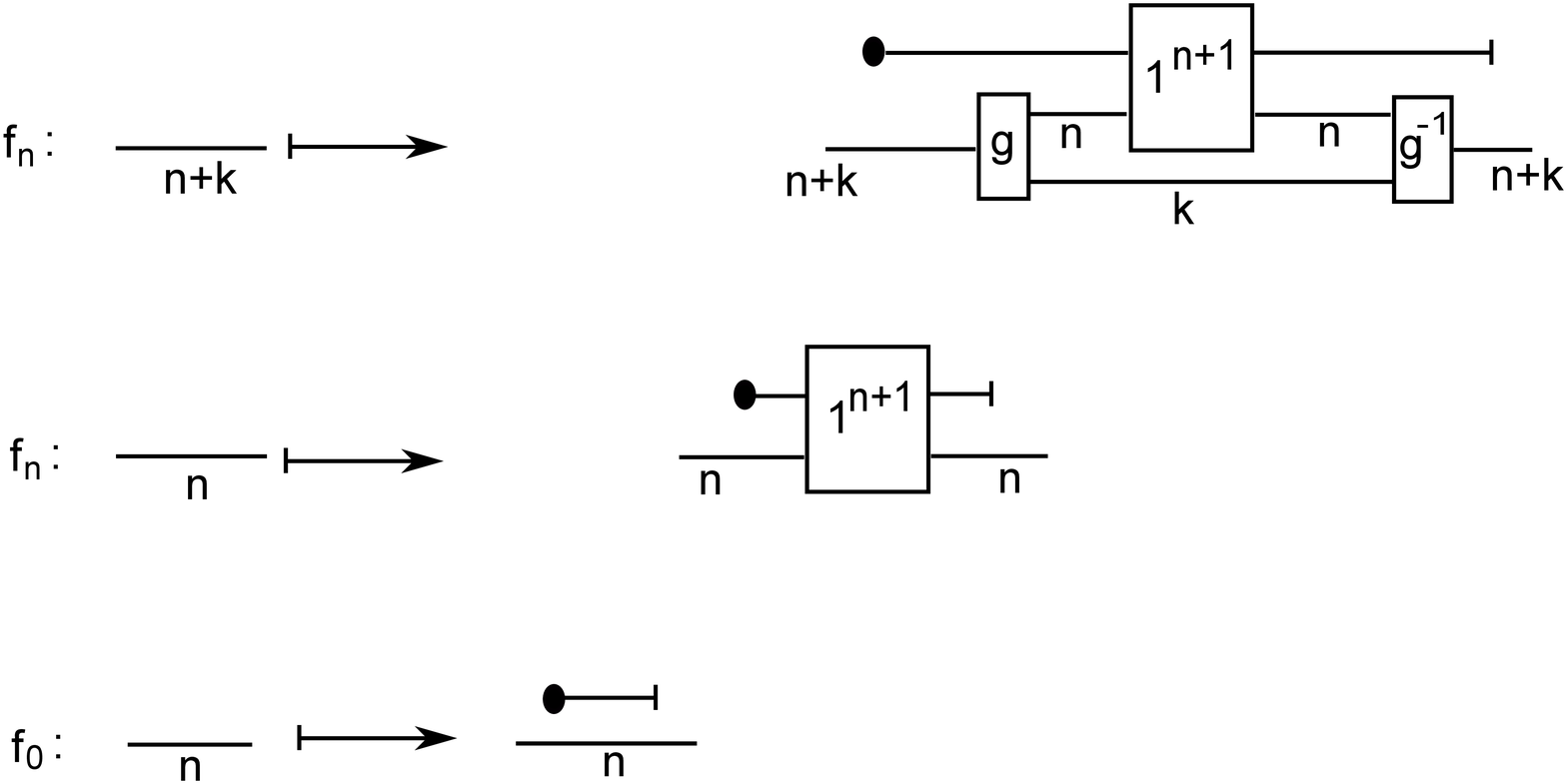}
\put(30,40){$\frac{n+1}{k!}\sum\limits_{g\in S(m)}$}
\put(30,20){${\scriptstyle (n+1)!}$}
\end{overpic}
\caption{The definition of $f_n \in \Hom(B,P)$.}
\label{s7}
\end{figure}

\begin{lemma} \label{lem fn welldefined}
The map $f_n$ is a $B$-bimodule homomorphism.
\end{lemma}
\begin{proof}
The bimodule $B$ is generated by $1(m)$, subject to the relations
$$a1(m)=1(m)a, ~\mbox{for}~~a \in \K[S(m)], \qquad 1(m)v_1(m+1)=v_1(m+1)1(m+1).$$
It suffices to show that: $a f_n(1(m))=f_n(1(m))  a$, $f_n(1(m))v_1(m+1)=v_1(m+1)f_n(1(m+1))$.
The first equality is true since the definition of $f_n$ is a sum of conjugation over $S(m)$.
The second equality is proved in figure \ref{s8}, where a relation induced from the symmetric relation is used.
\end{proof}

\begin{figure}[h]
\begin{overpic}
[scale=0.25]{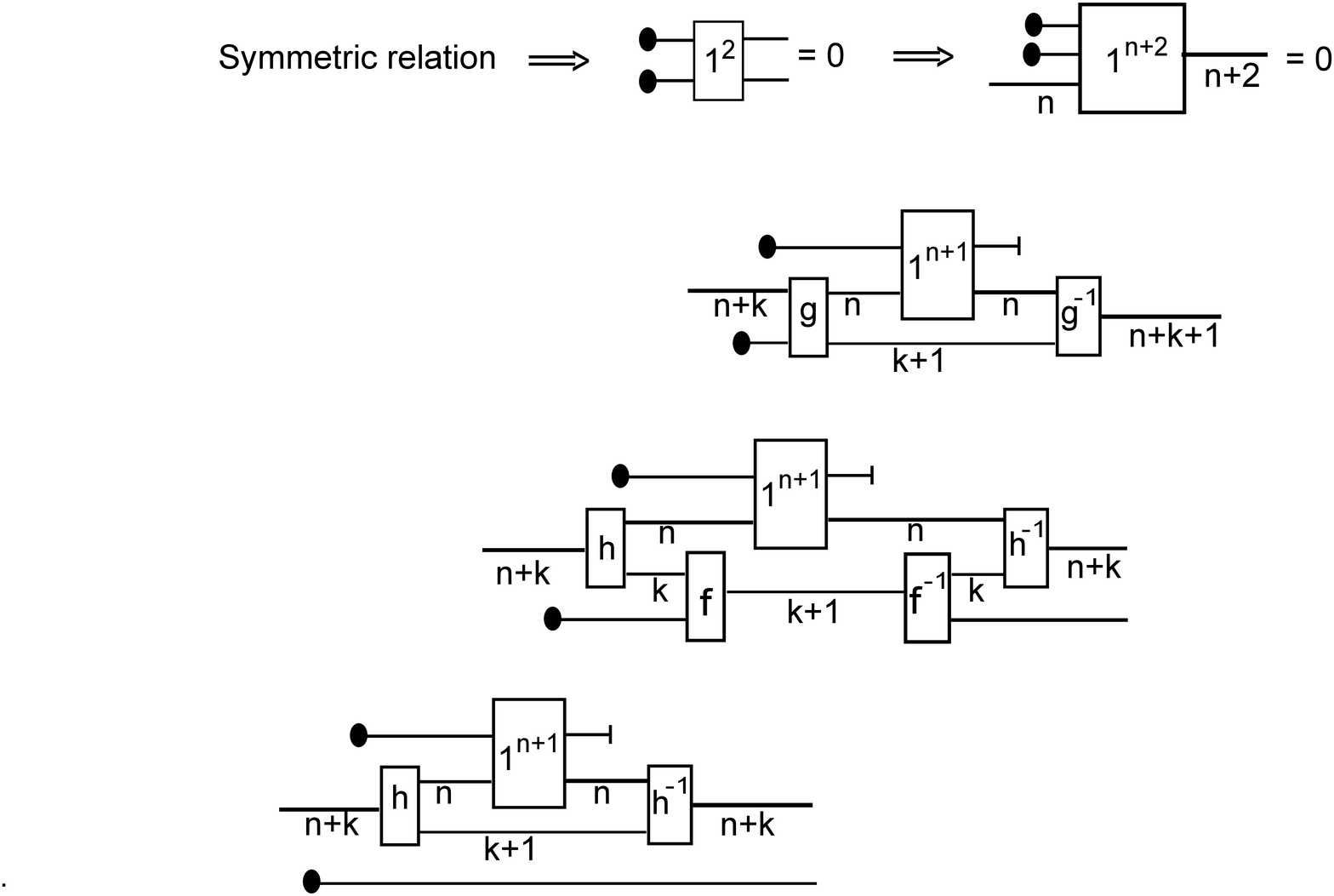}
\put(0,45){${\scriptstyle v_1(m+1)f_n(1(m+1))}=\frac{n+1}{(k+1)!}\sum\limits_{\scriptscriptstyle g\in S(m+1)}$}
\put(1,25){$=\frac{n+1}{(k+1)!}\sum\limits_{\scriptscriptstyle f\in D(k+1)}\sum\limits_{\scriptscriptstyle h\in S(m)}$}
\put(0,8){$=\frac{n+1}{k!}\sum\limits_{\scriptscriptstyle h\in S(m)}$}
\put(65,8){$={\scriptstyle f_n(1(m))v_1(m+1)}$}
\end{overpic}
\caption{The second equality in the third line uses the relation in the top row, and $D(k+1)=\{\mbox{transposition}~ (1,i) \in S(k+1), 1\le i \le k+1\}$.}
\label{s8}
\end{figure}

\begin{prop} \label{prop basis 1,P}
The morphism space $\Hom(B,P)=\prod\limits_{n \ge 0}\K\lan f_n \ran$, where $\K\lan f_n \ran$ is the one dimensional vector space spanned by $f_n$.
\end{prop}
\begin{proof}
For $f\in \Hom(B,P)$, we claim that if $f(1(n-1))=0$ then $f(1(n))=cf_n(1(n))$ for some constant $c$.
The image $f(1(n)) \in 1(n)\cdot P\cdot 1(n+1) \cong \K[S(n+1)]$ as vector spaces.
So $f(1(n))=v_{1}(n+1)g$ for a unique $g \in \K[S(n+1)]$.
The condition $f(1(n-1))=0$ implies that $f(v_i(n))=f(1(n-1)v_i(n))=0$ for $1 \le i \le n$.
Thus, $$v_i(n)v_{1}(n+1)g=v_i(n)f(1(n))=f(v_i(n)1(n))=0 \in B(n-1,n+1).$$
Then $(1,i+1)\cdot g=-g$ by the symmetric relation and Lemma \ref{lem bmn}, where $(1,i+1)$ is the transposition of $1$ and $i+1$.
Since $S(n+1)$ is generated by $(1,i+1)$'s, we have $a\cdot g=(-1)^{\op{sign}(a)}g$ for all $a \in S(n+1)$.
It implies that $g=c~e_{(1^{n+1})}$.
The claim follows from the definition of $f_n$.

The claim implies that for any $f\in \Hom(B,P)$ there exists $c_n \in \K$ such that $f-\sum\limits_{n\ge0}c_nf_n$ maps $1(m)$ to $0$ for all $m\ge0$. So $f=\sum\limits_{n\ge0}c_nf_n$.
It is obvious that $f_n$'s are linearly independent.
\end{proof}

We draw a vertical upwards short strand with a label $n$ to denote $f_n \in \Hom(B,P)$.
The map $f_n$ is understood as zero if $n<0$.
Define its right adjoint $f_n^*=\op{cap}_{PQ} \circ (f_n \ot id_Q) \in \Hom(Q, B)$, see figure \ref{s9}.
The map $f_n^*$ is obtained from $f_n$ by adding the U-turn map $\op{cap}_{PQ}$ on the top.
We can also add another U-turn map $\op{cap}_{QP}$ on the top.

\begin{figure}[h]
\begin{overpic}
[scale=0.25]{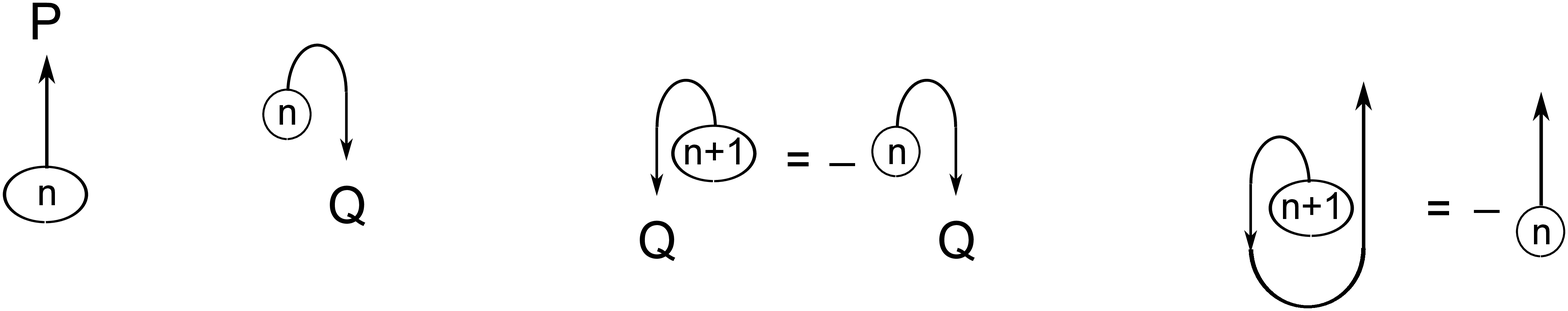}
\put(2,0){$f_n$}
\put(18,0){$f_n^*$}
\end{overpic}
\caption{}
\label{s9}
\end{figure}

\begin{lemma} \label{lem fn n-1}
There is a relation $\op{cap}_{QP} \circ (id_Q \ot f_{n+1})=-\op{cap}_{PQ} \circ (f_{n} \ot id_Q) \in \Hom(Q, B)$.
In particular, $\op{cap}_{QP} \circ (id_Q \ot f_0)=0$.
\end{lemma}
\begin{proof}
We compute values of both maps on generators $1(m)$ of $Q$ for $m=n+k+1$, see figure \ref{s10}.
By definition, both sides maps $1(m)$ to zero for $m \le n$.
\end{proof}

\begin{figure}[h]
\begin{overpic}
[scale=0.2]{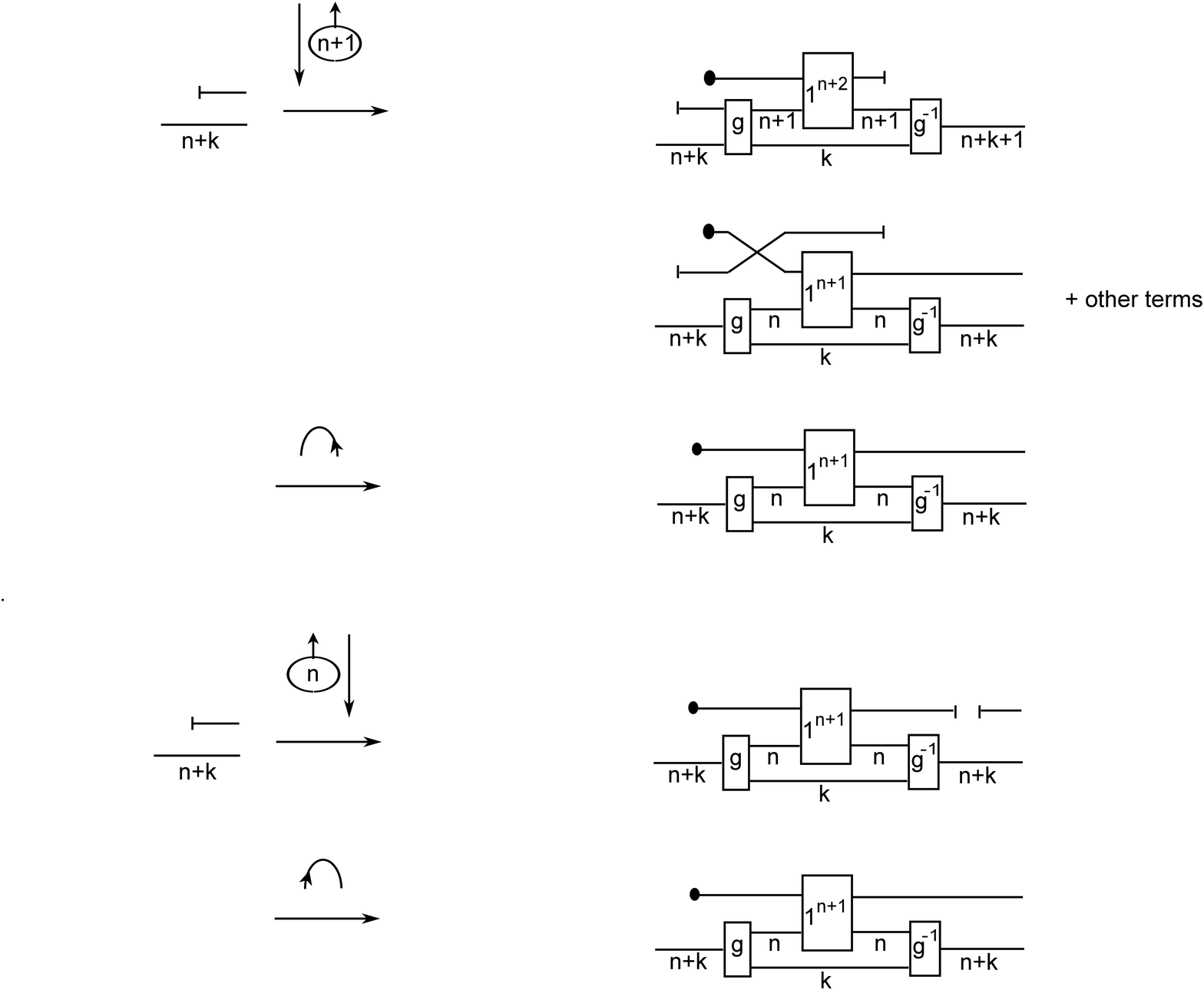}
\put(33,72){$\frac{n+2}{k!}\sum\limits_{g\in S(m)}$}
\put(30,58){$=-\frac{n+1}{k!}\sum\limits_{g\in S(m-1)}$}
\put(33,42){$-\frac{n+1}{k!}\sum\limits_{g\in S(m-1)}$}
\put(35,20){$\frac{n+1}{k!}\sum\limits_{g\in S(m-1)}$}
\put(35,6){$\frac{n+1}{k!}\sum\limits_{g\in S(m-1)}$}
\end{overpic}
\caption{The map $\op{cap}(QP)$ in the third line is the projection onto the summand $B \subset QP$.}
\label{s10}
\end{figure}
The relation in the lemma and another equivalent relation are drawn in figure \ref{s9}.

\subsection{The object right adjoint to $Q$}
We will use the morphisms in the last subsection to construct an object which is a right adjoint to $Q$.

Define a bimodule $\wt{P}_{-1}$ as the cokernel of $f_0: B \ra P$.
It is a quotient of $P$ by the subspace spanned by diagrams whose frozen strand is a short strand.
Although $\wt{P}_{-1}$ is neither left nor right projective, the derived tensor product of $\wt{P}_{-1}$ and $Q$ can still be computed via the ordinary tensor product since $Q$ is both left and right projective.

An analogue of the Heisenberg isomorphism in figure \ref{s4} still holds by replacing $P$ by $\wtp$, since the subspaces under quotient are isomorphic on both sides.
In other words, there is a canonical isomorphism of $B$-bimodules:
$$Q\wtp \cong \wtp Q \oplus B.$$
Let $d_{-1} \in \Hom(Q\wt{P}_{-1}, B)$ denote the projection.
Define a map $c_{-1} \in \Hom(B,  \wt{P}_{-1}Q)$ on the generators $1(m)$ by:
\begin{align*}
c_{-1}(1(m))=& \left\{
\begin{array}{rl}
0, & \quad \mbox{if} \hspace{0.3cm} m=0, \\
\frac{1}{(m-1)!}\sum\limits_{g \in S(m)}\ov{g} \ot_{1(m-1)} g^{-1}, & \quad \mbox{otherwise,}
\end{array}\right.
\end{align*}
where $\ov{g} \in \wtp$ denote the class of $g \in P$ in the quotient $\wtp$.
As in Lemma \ref{lem fn welldefined}, we check that $c_{-1}$ is a $B$-bimodule homomorphism:
$$v_1(m+1)c_{-1}(1(m+1))=c_{-1}(1(m))v_1(m+1),$$
see figure \ref{s11}.
Note that we still use diagrams as before to represent an element of the quotient $\wtp$.
\begin{figure}[h]
\begin{overpic}
[scale=0.25]{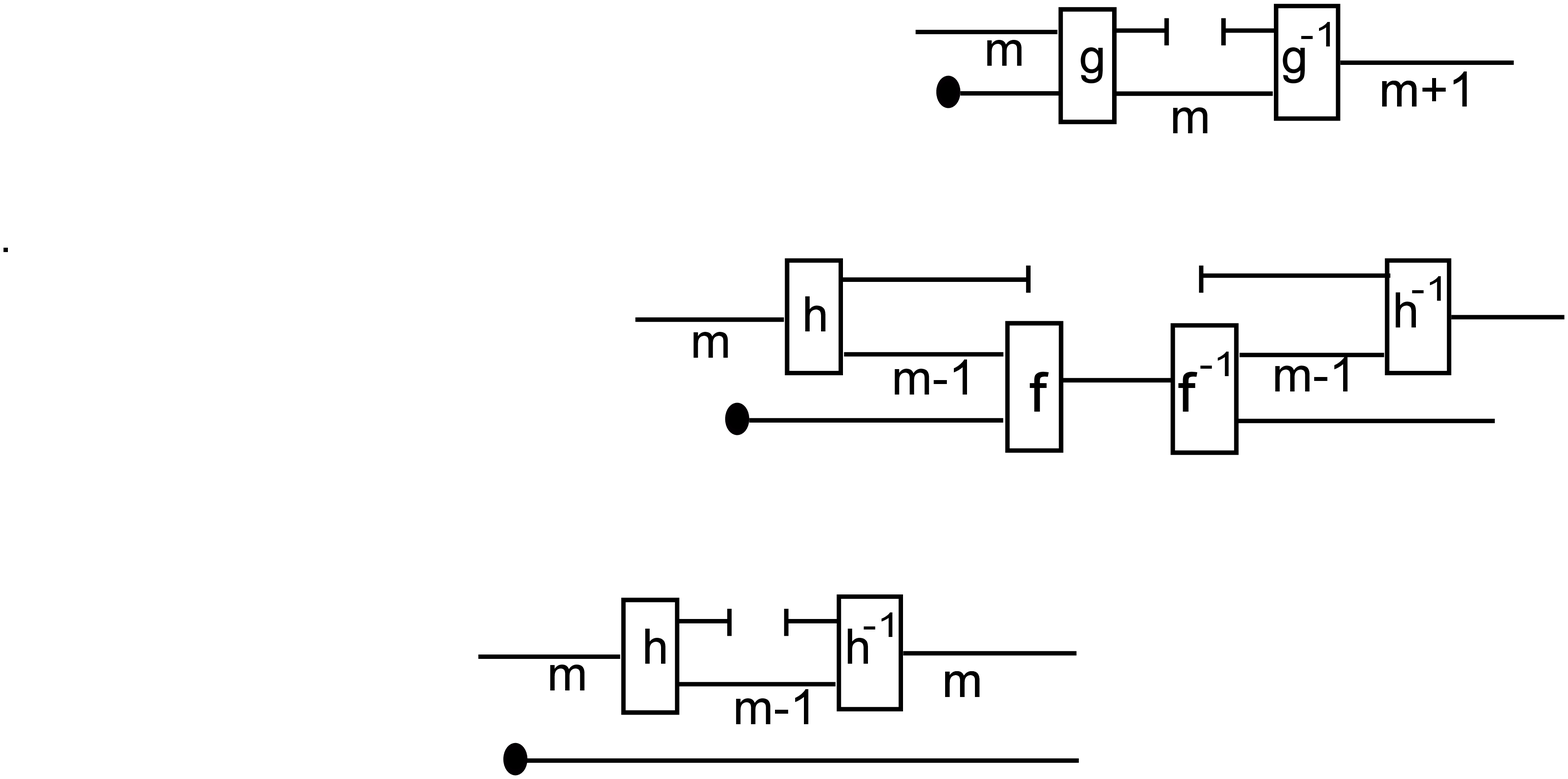}
\put(0,45){${\scriptstyle v_1(m+1)c_{-1}(1(m+1))}=\frac{1}{m!}\sum\limits_{g\in S(m+1)}$}
\put(1,25){$=\frac{1}{m!}\sum\limits_{f\in D(m)}\sum\limits_{h\in S(m)}$}
\put(0,5){$=\frac{1}{(m-1)!}\sum\limits_{h\in S(m)}$}
\put(72,5){$={\scriptstyle c_{-1}(1(m))v_1(m+1)}$}
\end{overpic}
\caption{The second equality holds in the quotient $\wtp Q$, and $D(m)=\{\mbox{transposition}~ (1,i) \in S(m), 1\le i \le m\}$.}
\label{s11}
\end{figure}

We draw a cup and a cap for $c_{-1}$ and $d_{-1}$, respectively.
\begin{figure}[h]
\begin{overpic}
[scale=0.2]{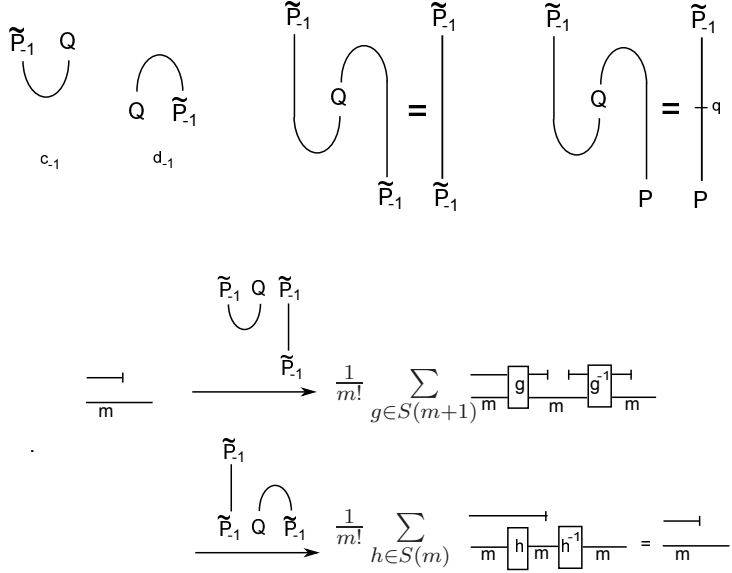}
\put(45,25){$\frac{1}{m!}\sum\limits_{g\in S(m+1)}$}
\put(45,5){$\frac{1}{m!}\sum\limits_{h\in S(m)}$}
\end{overpic}
\caption{Relations involving $c_{-1}, d_{-1}, \op{cap}_{QP}$ and $q$.}
\label{s12}
\end{figure}

\begin{prop} \label{prop adj -1}
The object $\wtp$ is right adjoint to $Q$ in $\ch$.
\end{prop}
\begin{proof}
It amounts to showing that
$$(id_{\wtp} \ot d_{-1}) \circ (c_{-1} \ot id_{\wtp})=id_{\wtp}, \qquad  (d_{-1} \ot id_{Q}) \circ (id_{Q} \ot c_{-1})=id_{Q}.$$
We prove the first equality on generators $\ov{1(m+1)}$ of $\wtp$ in figure \ref{s12}.
The proof for the other one is similar, and we leave it to the reader.
\end{proof}

The adjointness is equivalent to the isotopy relation involving $c_{-1}$ and $d_{-1}$, see figure \ref{s12}.
Another interesting relation is that
\begin{gather*} \label{eq PP-1}
(id_{\wtp} \ot \op{cap}_{QP}) \circ (c_{-1} \ot id_{P})=q \in \Hom(P, \wtp),
\end{gather*}
where $q: P \ra \wtp$ is the quotient map.
This relation will be used in Lemma \ref{lem f0**}.

Since $f_0$ is injective, $\wtp \cong \op{Cone} (B \xra{f_0} P) \in \ch$.
Let $P_{-1}$ denote the cone.
An object right adjoint to $P_{-1}$ can be computed via the following general method.

Let $M,N$ be two objects in a monoidal triangulated category $\cal{C}$ whose unit object is denoted by $\mb$.
Suppose $M$ and $N$ have right adjoint objects $M^*$ and $N^*$, respectively.
Let $c_M: \mb \ra M^*M$ and $d_M: MM^* \ra \mb$ denote the adjunction maps.
For $f \in \Hom_{\C}(M,N)$, let $f^* \in \Hom_{\C}(N^*,M^*)$ denote its right adjoint map.

\begin{lemma} \label{lem adj cone}
The object $L=(M \xra{f} N)$ has a right adjoint object $L^*=(N^* \xra{f^*} M^*)$, where $N$ and $N^*$ are in cohomological degree zero.
\end{lemma}
\begin{proof}
The product $LL^*=(MN^* \xra{(f \ot 1, -1\ot f^*)} (NN^* \oplus MM^*) \xra{(1 \ot f^*, f \ot 1)} NM^*)$.
The sum $(d_N, d_M): NN^* \oplus MM^* \ra \mb$ induces a map $d_L: LL^* \ra \mb$.
Similarly, define $c_L: \mb \ra L^*L$ induced by $c_M \oplus c_N: \mb \ra M^*M \oplus N^*N$.
It is easy to check that $d_L, c_L$ are chain maps.
It remains to show that $(id_{L^*} \ot d_{L}) \circ (c_{L} \ot id_{L^*})=id_{L^*}$, and $(d_{L} \ot id_{L}) \circ (id_{L} \ot c_{L})=id_{L}.$
We check the first equality on the component $N^*$.
The left hand side is
$$N^* \xra{(c_N \oplus c_M) \ot 1_{N^*}} (N^*N\oplus M^*M) \ot N^* \cong (N^*NN^* \oplus M^*MN^*) \xra{(1_{N^*} \ot d_{N}, ~~0)} N^*,$$
which is equal to $1_{N^*}$.
The rest of the proof is similar.
\end{proof}

It follows from the lemma above that $Q_{-1}=(Q \xra{f_0^*} B)$ is right adjoint to $P_{-1}$, where $Q$ is in cohomological degree zero.
Using the lemma again, we obtain $\wt{P}_{-2}=(B \xra{(f_0^*)^*} \wtp)$ is right adjoint to $Q_{-1}$.
Here, $(f_0^*)^*=(id_{\wtp} \ot f_0^*) \circ c_{-1}$.

\begin{lemma} \label{lem f0**}
There is a relation $(f_0^*)^*=q \circ (-f_1) \in \Hom(B ,\wtp)$, where $q: P \ra \wtp$ is the quotient map.
\end{lemma}
\begin{proof}
We give a graphical calculus in figure \ref{s13}, where the first equality is from Lemma \ref{lem fn n-1}, and the last equality is from (\ref{eq PP-1}) as in the top right picture in figure \ref{s12}.
\end{proof}
\begin{figure}[h]
\begin{overpic}
[scale=0.25]{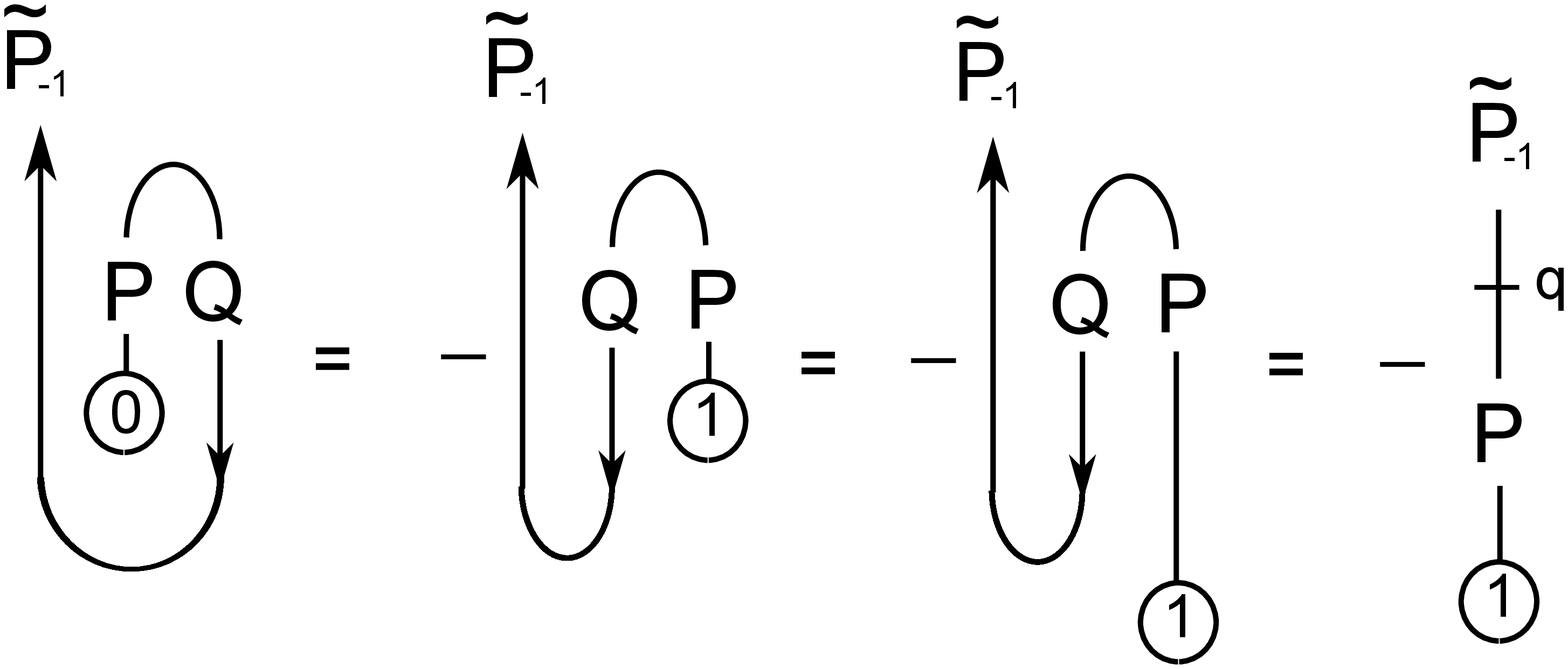}
\end{overpic}
\caption{In the rightmost picture, the little bar labeled by $q$ denotes the quotient map $q: P \ra \wtp$.}
\label{s13}
\end{figure}

Define $P_{-2}=(B\oplus B \xra{(f_0, f_1)} P)$ where $P$ is in cohomological degree zero.
There is a morphism $t: P_{-2} \ra \wt{P}_{-2}$ in $\ch$:
$$\xymatrix{
P_{-2}: \ar[d]^{t} & B \oplus B \ar[d]_{(0, id_B)} \ar[r]^{(f_0, f_1)} & P \ar[d]^{-q} \\
\wt{P}_{-2}:  & B  \ar[r]^{(f_0^*)^*} & \wtp \\
}$$
The map $t$ is a quasi-isomorphism since $B \xra{f_0} P \xra{-q} \wtp$ is a short exact sequence.
Therefore, $P_{-2} \cong \wt{P}_{-2} \in \ch$.

\begin{defn}
For $n \ge 0$, define $P_{-n}$ as a complex $(B^{\oplus n} \xra{(f_0, \dots, f_{n-1})} P)$,
and dually define $Q_{-n}$ as $(Q \xra{f_0^*, \dots, f_{n-1}^*} B^{\oplus n})$, where $P$ and $Q$ are in cohomological degree zero.
In particular, $P_0=P$ and $Q_0=Q$.
\end{defn}

Using a similar argument of proving $P_{-2} \cong \wt{P}_{-2}$, one can prove the following results.
\begin{prop} \label{prop adj negative}
There is a half infinite chain of objects
$$P_0 \leftrightarrow Q_0 \leftrightarrow P_{-1}\leftrightarrow \cdots \leftrightarrow Q_{-n+1} \leftrightarrow P_{-n} \leftrightarrow Q_{-n} \leftrightarrow \cdots,$$
such that each pair of objects form an adjoint pair in $\ch$.
\end{prop}

We will extend the chain to the left in the next two subsections.

\subsection{Extensions of $B$-bimodules}
There is an exact triangle $Q_{-1} \ra Q_0 \xra{f_0^*} B \ra Q_{-1}[1]$ in $\ch$.
Here, $Q_0$ is left adjoint to $P_{-1}$ which is left adjoint to $Q_{-1}$.
Similar exact triangles exist in $\ch$ when we decrease indices of $Q$'s.
Suppose we increase indices by one to obtain a sequence
$$Q_0 \ra Q_1 \ra B \xra{h} Q_0[1]$$
for some $h \in \Hom(B, Q_{0}[1])$, where $Q_1$ is some object to be constructed.
Then $Q_1$ might be a candidate of an object left adjoint to $P_0$ which is left adjoint to $Q_0$.
If the sequence can be proved to be an exact triangle, then $Q_1$ is isomorphic to the cone of $h$, i.e. an extension of $B$ by $Q_0=Q$.
This observation motivates us to study the extension group $\Hom(B, Q[1])$.

We construct a $B$-bimodule $Q_1$ in two steps as follows.
We will show that there is a short exact sequence of $B$-bimodules:
\begin{gather} \label{eq Q1}
0 \ra Q \ra Q_1 \ra B \ra 0.
\end{gather}

\n{\bf Step 1: The left module.}
We first define
$$Q_1=Q \oplus B,$$
as a left $B$-module.
Let $[a]_1 \in B$, and $[b]_Q \in Q$ for $a,b \in B(m,n)$.
Then
$$a \cdot [1(n)]_1=[a]_1; \qquad 1(n) \cdot [1(n+1)]_Q=[1_{n+1}]_Q, \quad b \cdot [1(n+1)]_Q=[b \bt 1_1]_Q.$$
The left $B$-module $Q_1$ is projective.

\vspace{.1cm}
\n{\bf Step 2: The right module.}
The subspace $Q$ of $Q_1$ is a right $B$-submodule:
$$[b]_Q \cdot b'=[bb']_Q, \quad \mbox{for}~~ b \in B(m,n), b' \in B(n,n').$$
The algebra $B$ contains $\bigoplus\limits_{n \ge 0}\K[S(n)]$ as a subalgebra.
The subspace $B$ is a right $\bigoplus\limits_{n \ge 0}\K[S(n)]$-submodule of $Q_1$:
$$[1(n)]_1 \cdot a=[a]_1, \quad \mbox{for}~~ a \in \K[S(n)].$$
The only nontrivial part of the definition of $Q_1$ is right multiplication on $[1(n-1)]_1$ with $v_n(n) \in B(n-1,n)$, for $n \ge 1$:
\begin{gather} \label{eq rt g}
[1(n-1)]_1 \cdot v_n(n)=[v_n(n)]_1 + [1(n)]_Q,
\end{gather}
see figure \ref{s14}.
Here we use the presentation of $B$ in Lemma \ref{lem def B}.
The algebra $B$ is generated by $\K[S(n-1)]$ and $v_n(n)$ for $n \ge 1$.
In particular, $v_{i}(n)=v_n(n) \cdot g_{i}(n)$ for some element $g_{i}(n) \in \K[S(n)]$.
Define $[1(n-1)]_1 \cdot v_i(n)=([1(n-1)]_1 \cdot v_n(n)) \cdot g_{i}(n)$.

Define right multiplication on $[a]_1$ with $b \in B$ as
\begin{gather} \label{eq lf rt comm}[a]_1 \cdot b=a \cdot ([1(n)]_1 \cdot b),\end{gather}
for $a \in B(m,n)$.
Our construction of $Q_1$ is complete.

\begin{figure}[h]
\begin{overpic}
[scale=0.25]{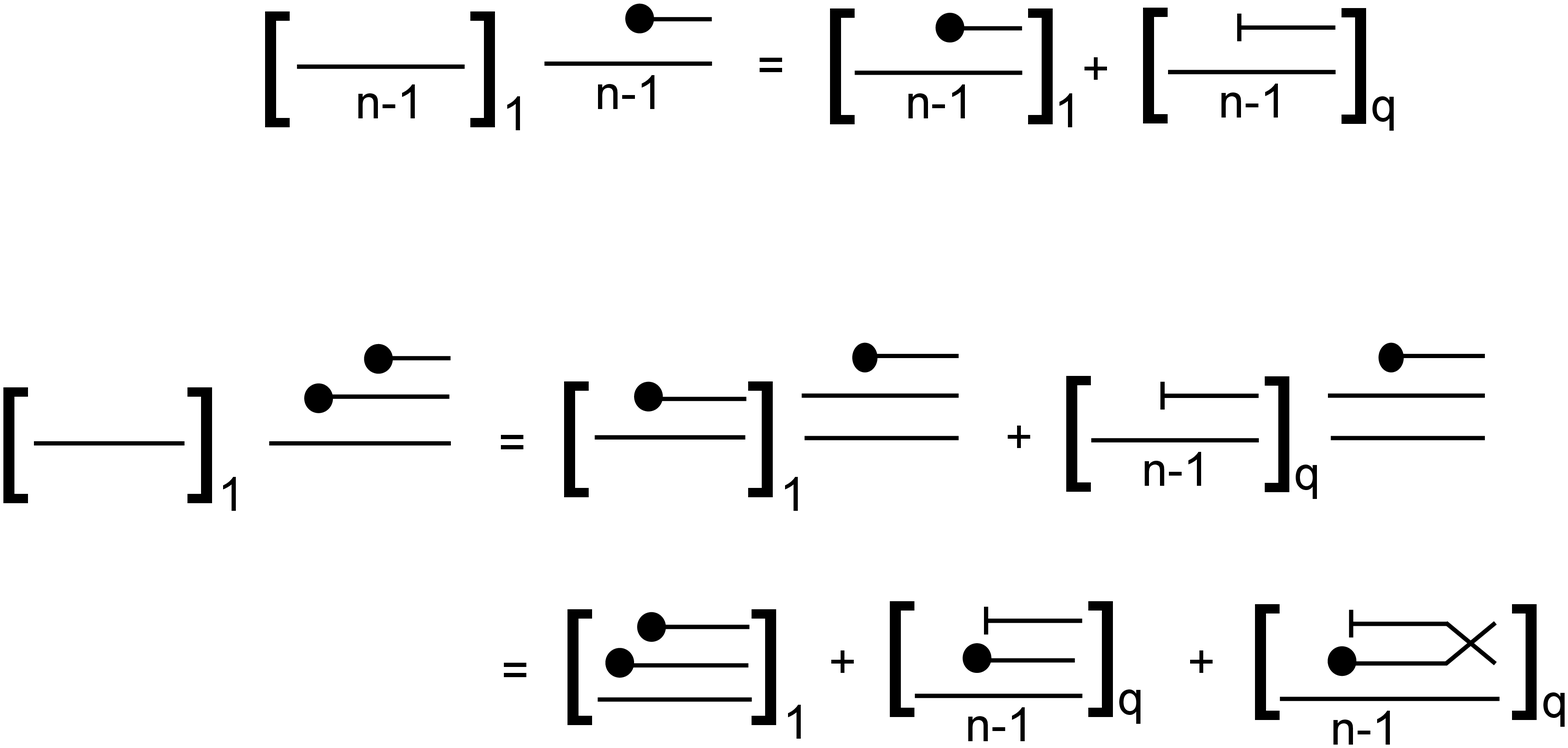}
\end{overpic}
\caption{The first row is the definition of $[1(n-1)]_1 \cdot v_n(n)$; the rest computes $([1(n-1)]_1 \cdot v_n(n)) \cdot v_{n+1}(n+1)$.}
\label{s14}
\end{figure}

\vspace{.1cm}
By definition, $Q$ is a subbimodule of $Q_1$.
As a left $B$-module, $B$ is projective, and generated by $[1(n)]_1$ for $n \ge 0$.
Thus, the right multiplication on the left submodule $B$ of $Q_1$ is determined by the right multiplication on these generators.
It commutes with the left multiplication by (\ref{eq lf rt comm}).

\begin{lemma} \label{lem Q1 well}
The right multiplication on $Q_1$ is well-defined.
\end{lemma}
\begin{proof}
Since the right action of $\K[S(n)]$ on $Q_1$ is the ordinary multiplication in $B$, we only have to check the relations involving $v_n(n)$ in Lemma \ref{lem def B}.
The isotopy relation is easy to verify.
To check the symmetric relation $v_{n}(n) v_{n+1}(n+1)=v_{n}(n) v_{n+1}(n+1) s_{n}(n+1)$, we compute
\begin{align*}
([1(n-1)]_1 \cdot v_n(n)) \cdot v_{n+1}(n+1) &=([v_{n}(n)]_1+[1(n)]_Q) \cdot v_{n+1}(n+1) \\
&=v_{n}(n) \cdot ([1(n)]_1 \cdot v_{n+1}(n+1)) +[v_{n+1}(n+1)]_Q \\
&=v_{n}(n) \cdot ([v_{n+1}(n+1)]_1 + [1(n+1)]_Q) +[v_{n+1}(n+1)]_Q \\
&=[v_{n}(n) v_{n+1}(n+1)]_1 + [v_n(n+1)]_Q +[v_{n+1}(n+1)]_Q,
\end{align*}
see figure \ref{s14}.
Recall that $s_n(n+1) \in S(n+1)$ exchanges $v_n(n+1)$ and $v_{n+1}(n+1)$.
It implies that $([1(n-1)]_1 \cdot v_n(n)) \cdot v_{n+1}(n+1)$ is invariant under right multiplication with $s_n(n+1)$.
\end{proof}

The $B$-bimodule $Q_1$ fits into the short exact sequence (\ref{eq Q1}).

\begin{lemma} \label{lem Q1 nontriv}
The extension $Q_1$ of $B$ by $Q$ is not split.
\end{lemma}
\begin{proof}
Suppose $Q_1$ is split, i.e. there is a commutative diagram
$$
\xymatrix{
 0 \ar[r] & Q \ar[r] \ar[d]^{id} &  Q_1 \ar[r] \ar[d]^{\psi} & B \ar[d]^{id} \ar[r] & 0. \\
 0 \ar[r] & Q \ar[r]  & Q\oplus B  \ar[r] & B \ar[r] & 0. \\
}$$
of $B$-bimodules.
Let $[1(n-1)]'_1$ and $[1(n)]'_Q$ for $n \ge 1$ denote the generators of $B \oplus Q$.
The commutative diagram implies that $\psi([1(n)]_Q)=[1(n)]'_Q$, and $\psi([1(n)]_1)=[1(n)]'_1+[t]'_Q$ for some $t \in 1(n)\cdot Q \cdot 1(n)$.
Since $1(n)\cdot Q \cdot 1(n)\cong B(n+1,n)=0$, we have $\psi([1(n)]_1)=[1(n)]'_1$.
So $\psi([t]_1)=\psi(t \cdot [1(n)]_1)=t \cdot \psi([1(n)]_1)=t \cdot [1(n)]'_1=[t]'_1$.
Similarly, $\psi([t]_Q)=[t]'_Q$.
We have
$$\psi([1(n-1)]_1 \cdot v_n(n))=\psi([v_n(n)]_1+[1(n)]_Q)=[v_n(n)]'_1+[1(n)]'_Q,$$
$$\psi([1(n-1)]_1) \cdot v_n(n)=[1(n-1)]'_1 \cdot v_n(n)=[v_n(n)]'_1.$$
Thus $\psi$ is not a map of right $B$-modules. This gives a contradiction.
\end{proof}

We mimic the definition of $Q_1$ to construct a family of extensions of $B$ by $Q$ as follows.
For $n \geq 0$, define a $B$-bimodule $Q_{1,n}$ in two steps.

\n{\bf Step 1: The left module.}
Define $Q_{1,n}=Q \oplus B$
as a left $B$-module.

\n{\bf Step 2: The right module.}
The subspace $Q$ of $Q_{1,n}$ is a right $B$-submodule.
The subspace $B$ is a right $\bigoplus\limits_{n \ge 0}\K[S(n)]$-submodule of $Q_{1,n}$.
The nontrivial part is to define the right multiplication $[1(m)]_1 \cdot v_{m+1}(m+1)$ as:
\begin{align*}
[v_{m+1}(m+1)]_1, & \quad \mbox{if}~ 0 \le m < n, \\
[v_{m+1}(m+1)]_1+\frac{n+1}{k!}\sum\limits_{g \in S(m)}[(g \bt 1(1)) (1(k) \bt e_{(n+1)}) (g^{-1} \bt 1(1))]_Q, & \quad \mbox{if}~ m=n+k, k\ge 0,
\end{align*}
where $e_{(n+1)}=\frac{1}{(n+1)!}\sum\limits_{g \in S(n+1)}~g$.

\begin{figure}[h]
\begin{overpic}
[scale=0.25]{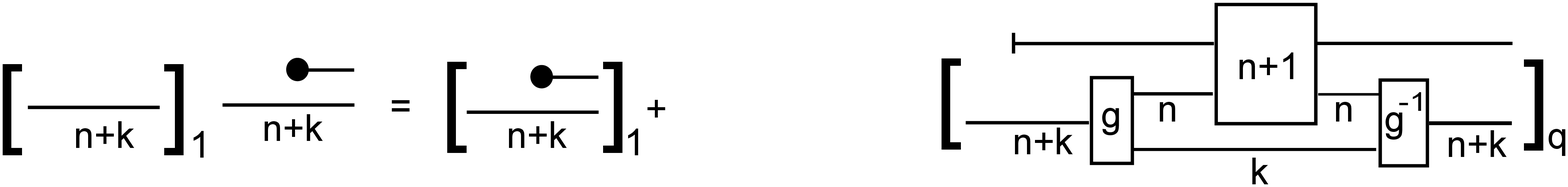}
\put(44,5){$\frac{n+1}{k!}\sum\limits_{\scriptscriptstyle g \in S(m)}$}
\end{overpic}
\caption{The right multiplication $[1(m)]_1 \cdot v_{m+1}(m+1)$ in $Q_{1,n}$ for $m=n+k$.}
\label{s15}
\end{figure}
By definition, $Q_{1,0}=Q_1$.
For $m=n$,
\begin{gather} \label{eq hn}
[1(n)]_1 \cdot v_{n+1}(n+1)=[v_{n+1}(n+1)]_1+(n+1)!~[e_{(n+1)}]_Q \in Q_{1,n}.
\end{gather}
We describe the nontrivial right multiplication in figure \ref{s15}, where a box with $n+1$ inside denotes $e_{(n+1)}$.
To compare figures \ref{s15} and \ref{s7}, the definitions of $Q_{1,n}$ and $f_n \in \Hom(B,P)$ are in some sense dual to each other.

\begin{lemma} \label{lem hn welldefined}
The right multiplication on $Q_{1,n}$ is well-defined.
\end{lemma}
\begin{proof}
It suffices to check the isotopy relation and the symmetric relation involving $v_{m+1}(m+1)$ in Lemma \ref{lem def B}.
The isotopy relation holds since $[1(m)]_1 \cdot v_{m+1}(m+1)$ uses a sum of conjugation over $S(m)$.
For the symmetric relation, it is enough to check that $([1(m)]_1 \cdot v_{m+1}(m+1)) \cdot v_{m+2}(m+2)$ is invariant under right multiplication with $s_{m+1}(m+2)$.
The proof is a generalization of that of Lemma \ref{lem Q1 well} as in figure \ref{s14}, where the analogue of the proof of Lemma \ref{lem fn welldefined} in figure \ref{s8} is needed.
We leave it to the reader.
\end{proof}

The extension $Q_{1,n}$ gives rise to a family of morphisms $h_n \in \Hom(B,Q[1])$ for $n \ge 0$.
We write $\Hom^1(B,Q)$ for $\Hom(B,Q[1])$.
Lemma \ref{lem Q1 nontriv} implies that $h_0 \neq 0 \in \Hom^1(B,Q)$.
A similar argument can be used to prove that $h_n$'s are linearly independent in $\Hom^1(B,Q)$.

\begin{prop} \label{prop basis 1,Q}
The morphism space $\Hom^1(B,Q)=\prod\limits_{n \ge 0}\K\lan h_n \ran$, where $\K\lan h_n \ran$ is the one dimensional vector space spanned by $h_n$.
\end{prop}
\begin{proof}
For $h \in \Hom^1(B,Q)$, let $M$ denote the corresponding extension of $B$ by $Q$.
Since $B$ is projective as a left $B$-module, we can assume that $M=B\oplus Q$ as a left $B$-module.
The decomposition also holds as a right $\bigoplus\limits_{n \ge 0}\K[S(n)]$-module.
Thus the extension $M$ is completely determined by values of right multiplication $[1(m)]_1 \cdot v_{m+1}(m+1)$ for $m \ge 0$.
As in the proof of Proposition \ref{prop basis 1,P}, it suffices to show the following claim.

\n{\bf Claim:} If $[1(n-1)]_1 \cdot v_{n}(n)=[v_{n}(n)]_1$, then $[1(n)]_1 \cdot v_{n+1}(n+1)=[v_{n+1}(n+1)]_1+c[e_{(n+1)}]_Q$ for some constant $c$.
In other words, if the extension is trivial for $[1(n-1)]_1 \cdot v_{n}(n)$, then it is propositional to the extension $h_n$ for $[1(n)]_1 \cdot v_{n+1}(n+1)$, see (\ref{eq hn}).

Suppose that $[1(n-1)]_1 \cdot v_{n}(n)=[v_{n}(n)]_1$, then $[1(n-1)]_1 \cdot v_{i}(n)=[v_{i}(n)]_1$ for $1 \le i \le n$.
Let $[1(n)]_1 \cdot v_{n+1}(n+1)=[v_{n+1}(n+1)]_1+[g]_Q$ for some $g \in 1(n)\cdot Q \cdot 1(n+1) \cong \K[S(n+1)]$.
We compute
\begin{align*}
([1(n-1)]_1 \cdot v_i(n)) \cdot v_{n+1}(n+1)
=&[v_i(n)]_1 \cdot v_{n+1}(n+1) \\
=&v_i(n) \cdot ([1(n)]_1 \cdot v_{n+1}(n+1)) \\
=&v_i(n) \cdot ([v_{n+1}(n+1)]_1+[g]_Q) \\
=&[v_i(n) v_{n+1}(n+1)]_1+[v_i(n+1)g]_Q.
\end{align*}
The resulting element should be invariant under right multiplication with the transposition $g_i=(i,n+1) \in S(n+1)$ since $v_i(n)v_{n+1}(n+1)g_i=v_i(n)v_{n+1}(n+1)$.
In particular, $$v_i(n+1)gg_i=v_i(n+1)g \in B(n,n+1) \cong S(n+1).$$
It implies that $g g_i=g$ for $1 \le i \le n$.
Since $g_i$'s generate $S(n+1)$, we have $ga=g$ for all $a \in S(n+1)$.
So $g=c~e_{(n+1)}$ which proves the claim.
\end{proof}

We draw a vertical downwards short strand with a label $n$ to denote $h_n \in \Hom^1(B,Q)$.
The map $h_n$ is understood as zero if $n<0$.
Define its left adjoint $^*h_n=\op{cap}_{PQ} \circ (id_P \ot h_n) \in \Hom^1(P, B)$, see figure \ref{s16}.
The map $^*h_n$ is obtained from $h_n$ by adding the U-turn map $\op{cap}_{PQ}$ on the top.
We can also add another U-turn map $\op{cap}_{QP}$ on the top.

\begin{figure}[h]
\begin{overpic}
[scale=0.3]{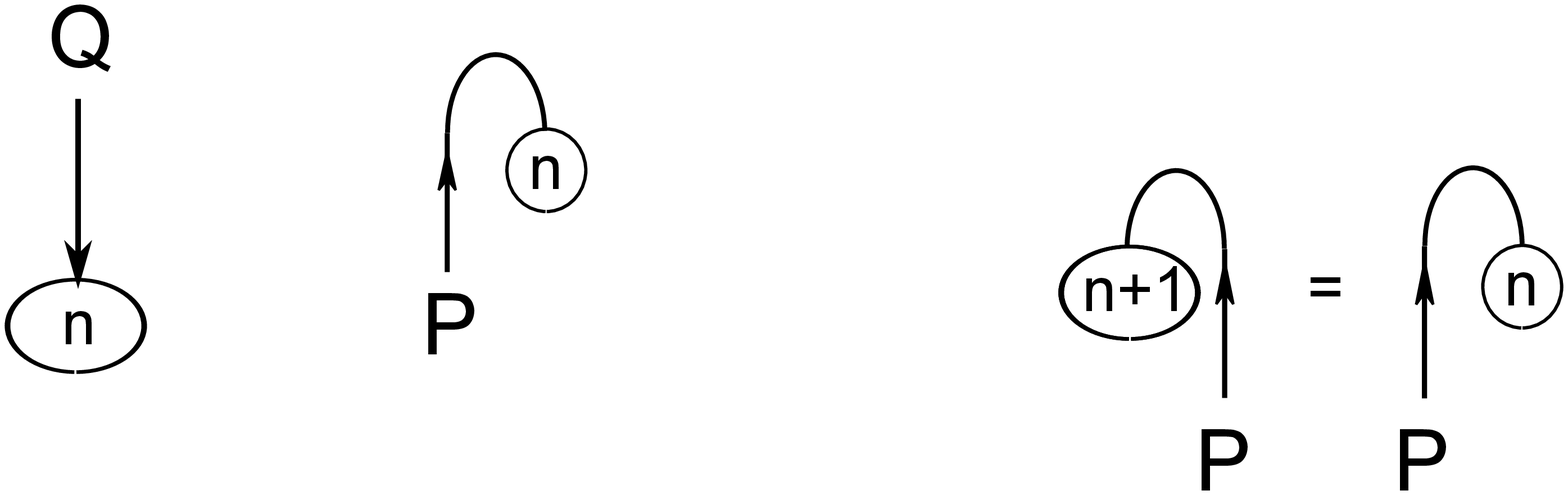}
\put(2,0){$h_n$}
\put(28,0){$^*h_n$}
\end{overpic}
\caption{Definitions of $h_n$ and $^*h_n$.}
\label{s16}
\end{figure}

\begin{lemma} \label{lem hn n-1}
There is a relation $\op{cap}_{QP} \circ (^*h_{n+1} \ot id_P)=\op{cap}_{PQ} \circ (id_P \ot h_n) \in \Hom^1(P,B)$.
In particular, $\op{cap}_{QP} \circ (^*h_0 \ot id_P)=0$.
\end{lemma}
\begin{proof}
Let $M_{n+1}'$ denote the extension of $P$ by $QP$ corresponding to $^*h_{n+1} \ot id_P$, and
$N_{n}'$ denote the extension of $P$ by $PQ$ corresponding to $id_P \ot h_n$.
The morphisms on both sides of the relation correspond to two extensions $M_{n+1}$ and $N_n$ of $P$ by $B$.
They are related by the following commutative diagrams of bimodules.
$$
\xymatrix{
 0 \ar[r] & QP \ar[r] \ar[d]^{\op{cap}_{QP}} &  M_{n+1}' \ar[r] \ar[d] & P \ar[d]^{id} \ar[r] & 0. \\
 0 \ar[r] & B \ar[r]  &  M_{n+1}  \ar[r] & P \ar[r] & 0. \\
}$$
$$
\xymatrix{
 0 \ar[r] & PQ \ar[r] \ar[d]^{\op{cap}_{PQ}} &  N_{n}' \ar[r] \ar[d] & P \ar[d]^{id} \ar[r] & 0. \\
 0 \ar[r] & B \ar[r]  &  N_{n}  \ar[r] & P \ar[r] & 0. \\
}$$

Both extensions $M_{n+1}$ and $N_n$ are equal to $P\oplus B$ as left $B$-modules.
It suffices to show that values of right multiplication $[1(m)]_P \cdot v_{m+1}(m+1)$ for two extensions are equal.
For $m \le n$, $[1(m)]_P \cdot v_{m+1}(m+1)=[v_{m+1}(m+1)]_P$ for both extensions.
We compute the values for $m=n+k+1$ in figure \ref{s17}.
They are equal to each other which implies the lemma.
\end{proof}
\begin{figure}[h]
\begin{overpic}
[scale=0.2]{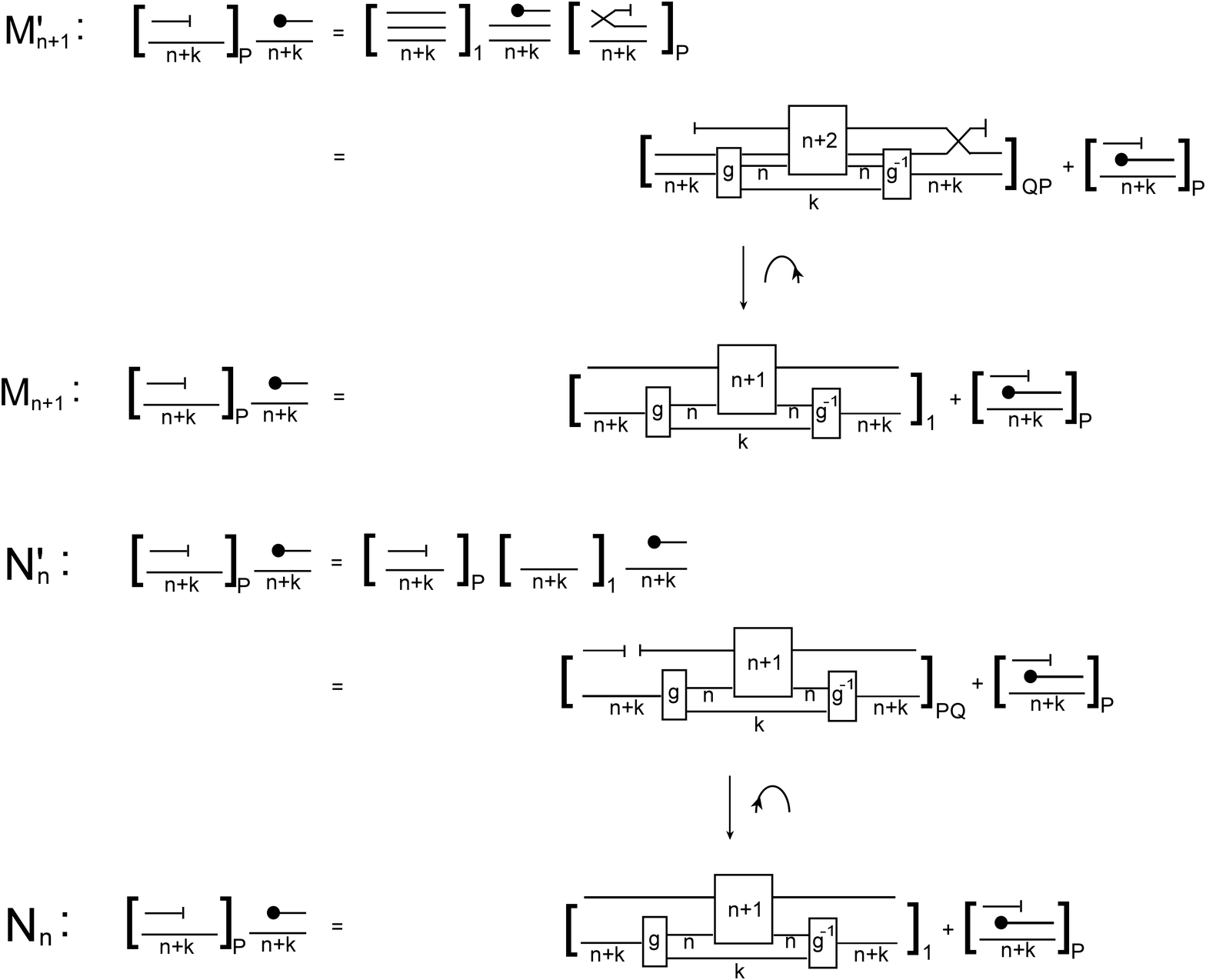}
\put(30,3){$\frac{n+1}{k!}\sum\limits_{\scriptscriptstyle g \in S(n+k)}$}
\put(30,24){$\frac{n+1}{k!}\sum\limits_{\scriptscriptstyle g \in S(n+k)}$}
\put(30,47){$\frac{n+1}{k!}\sum\limits_{\scriptscriptstyle g \in S(n+k)}$}
\put(30,68){$\frac{n+2}{k!}\sum\limits_{\scriptscriptstyle g \in S(n+k+1)}$}
\end{overpic}
\caption{}
\label{s17}
\end{figure}

\subsection{The object left adjoint to $P$}
In order to show that $Q_1$ is left adjoint to $P$, we construct two $B$-bimodule homomorphisms: $d_1: Q_1P \ra B,$ and $ c_1: B \ra PQ_1$ in the following.

Since $Q_1=Q\oplus B$ as left $B$-modules, $Q_1P=QP\oplus P \cong B\oplus PQ \oplus P =B\oplus PQ_1$ as left modules by the Heisenberg isomorphism (\ref{eq H relation}).
The isomorphism on generators of left $B$-modules is given by
$$
\begin{array}{cccc}
\phi: & B \oplus PQ_1 & \ra & Q_1P \\
& [1(m)]_B & \mapsto & [1(m+1)]_Q[1(m+1)]_P,  \\
& [1(m+1)]_P [1_{m}]_1 & \mapsto & [1(m+1)]_1[1(m+1)]_P, \\
& [1(m+1)]_P [1_{m+1}]_Q & \mapsto & [1(m+2)]_Q \cdot s_{m+1}(m+2) \cdot [1(m+2)]_P,
\end{array}
$$
for $n \ge 0$.

\begin{lemma} \label{lem Heisenberg Q1P}
The map $\phi: B \oplus PQ_1 \ra Q_1P$ is an isomorphism of $B$-bimodules.
\end{lemma}
\begin{proof}
Since $Q$ is a subbimodule of $Q_1$, we only have to check that the restriction of $\phi$ to $P \ot B \subset P \ot Q_1$ is a map of right $B$-modules.
We compute
\begin{align*}
&\phi([1(m+1)]_P [1_{m}]_1 \cdot v_{m+1}(m+1))\\
=&\phi([1(m+1)]_P ([v_{m+1}(m+1)]_1 + [1(m+1)]_Q)) \\
=& [v_{m+1}(m+2)]_1 [1(m+2)]_P + [1(m+2)]_Q \cdot s_{m+1}(m+2) \cdot [1(m+2)]_P,\\
&\phi([1(m+1)]_P [1_{m}]_1) \cdot v_{n+1}(n+1)\\
=&[1(m+1)]_1[1(m+1)]_P \cdot v_{m+1}(m+1)) \\
=& [1(m+1)]_1[v_{m+1}(m+2)]_P = [1(m+1)]_1 \cdot v_{m+1}(m+2) \cdot [1_{m+2}]_P \\
=& [v_{m+1}(m+2)]_1 [1(m+2)]_P + [1(m+2)]_Q \cdot s_{m+1}(m+2) \cdot [1(m+2)]_P.
\end{align*}
So $\phi([1(m+1)]_P [1_{m}]_1 \cdot v_{m+1}(m+1))=\phi([1(m+1)]_P [1_{m}]_1) \cdot v_{m+1}(m+1)$.
\end{proof}

Let $d_1: Q_1P \ra B$ denote the projection.
Define a map $c_1: B \ra PQ_1$ of $B$-bimodules on the generators $1(m)$ by:
$$
c_1(1(m))=
[v_{m+1}(m+1)]_P [1(m)]_1+(1-\delta_{m,0})\frac{1}{(m-1)!}\sum\limits_{g \in S(m)} [g]_P [g^{-1}]_Q$$
see figure \ref{s18}.
To show $c_1$ is well-defined, it is enough to check that:
$$v_{m+1}(m+1)c_{1}(1(m+1))=c_{1}(1(m))v_{m+1}(m+1).$$
The proof is similar to the proof in figure \ref{s11}.
We leave it to the reader.

\begin{figure}[h]
\begin{overpic}
[scale=0.3]{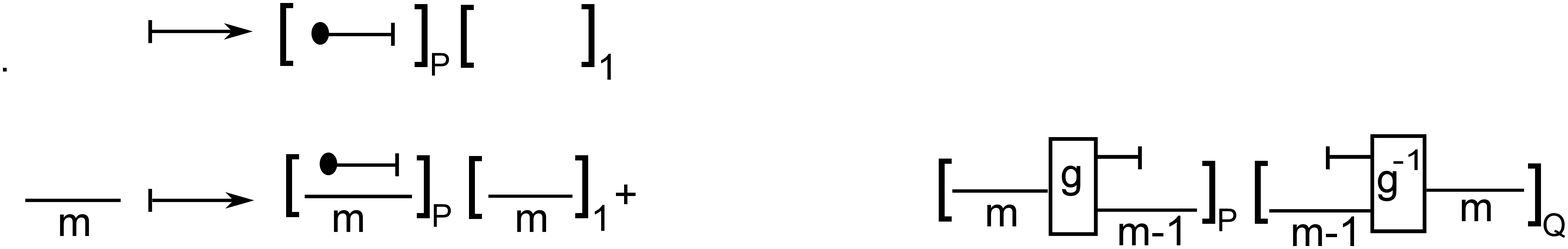}
\put(5,12){$\es$}
\put(33,12){$\es$}
\put(42,2){$\frac{1}{(m-1)!}\sum\limits_{\scriptscriptstyle g \in S(m)}$}
\end{overpic}
\caption{The map $c_1: B \ra PQ_1$ on the generators $1(m)$, where $\es$ denotes the empty diagram for $1(0) \in B$.}
\label{s18}
\end{figure}

\begin{prop} \label{prop adj +1}
The object $Q_1$ is left adjoint to $P$ in $\ch$.
\end{prop}
\begin{proof}
It amounts to showing that
$$(id_{P} \ot d_{1}) \circ (c_{1} \ot id_{P})=id_{P}, \qquad  (d_{1} \ot id_{Q_1}) \circ (id_{Q_1} \ot c_{1})=id_{Q_1}.$$
The proof is similar to that of Proposition \ref{prop adj -1} as in figure \ref{s12}.
\end{proof}

We draw a cup and a cap for $c_{1}$ and $d_{1}$, respectively.
The adjointness is equivalent to the isotopy relation, see figure \ref{s19}.
Another useful relation is that
\begin{gather*} \label{eq QQ1}
(\op{cap}_{QP} \ot id_{Q_1}) \circ (id_Q \ot c_{1})=r \in \Hom(Q, Q_1),
\end{gather*}
where $r: Q \ra Q_1$ is the inclusion map in the short exact sequence (\ref{eq Q1}).

\begin{figure}[h]
\begin{overpic}
[scale=0.2]{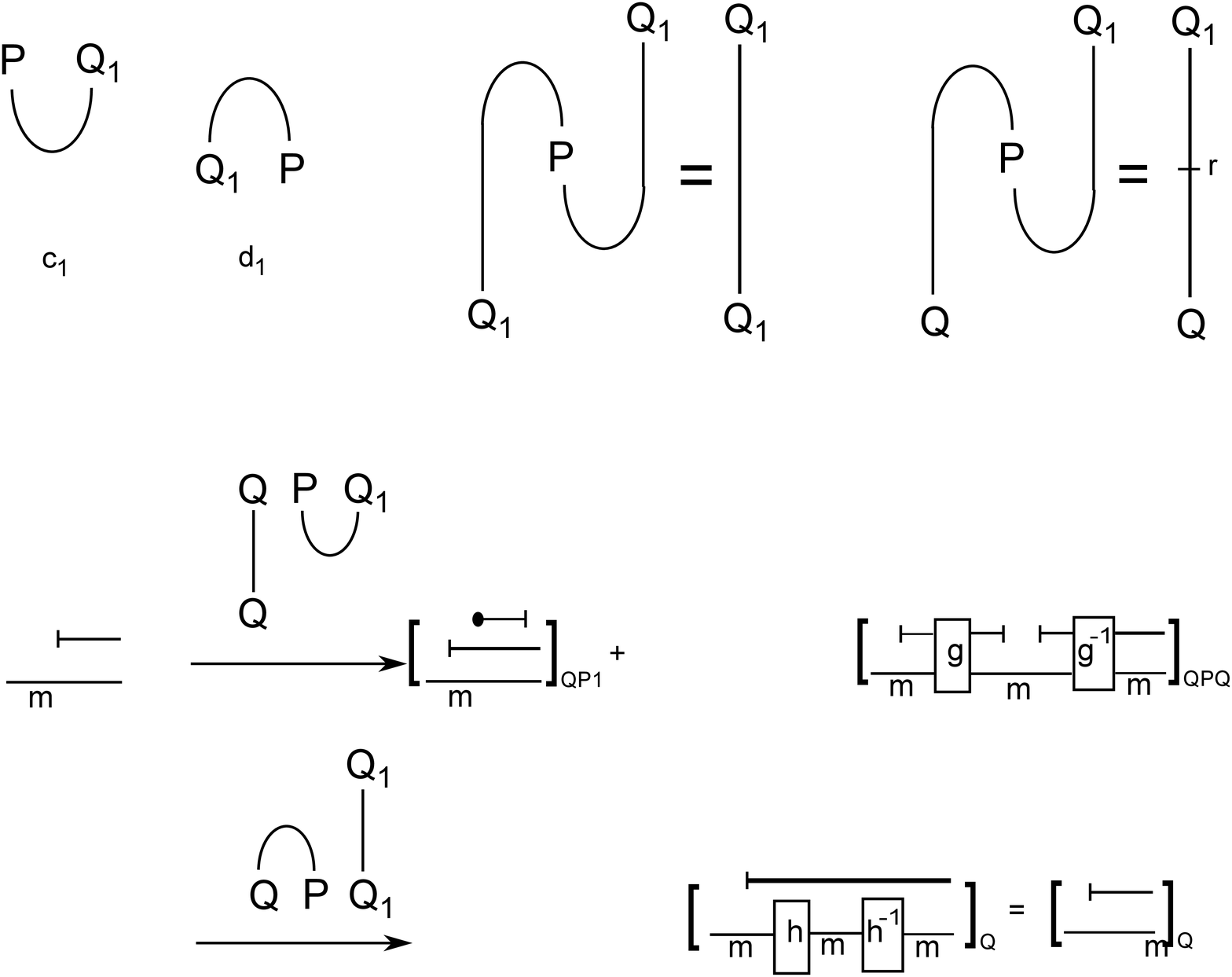}
\put(51,25){$\frac{1}{m!}\sum\limits_{\scriptscriptstyle g \in S(m+1)}$}
\put(38,4){$\frac{1}{m!}\sum\limits_{\scriptscriptstyle h \in S(m)}$}
\end{overpic}
\caption{Relations involving $c_1, d_1, \op{cap}_{QP}$, and $r$, and the proof of $(\op{cap}_{QP} \ot id_{Q_1}) \circ (id_Q \ot c_{1})=r$.}
\label{s19}
\end{figure}

The object $Q_1 \cong (B[-1] \xra{h_0} Q)$ in $\ch$, where $Q$ is in cohomological degree zero.
It follows from Lemma \ref{lem adj cone} that $P_1=(P \xra{^*h_0} B[1])$ is left adjoint to $Q_1$, where $P$ is in cohomological degree zero.
Using the lemma again, we obtain $\wt{Q}_{2}=(B[-1] \xra{^*(^*h_0)} Q_1)$ is left adjoint to $P_1$.
Here, $^*(^*h_0)=(^*h_0 \ot id_{Q_1}) \circ c_1$.

\begin{lemma} \label{lem h0**}
There is a relation $^*(^*h_n)=r \circ (h_{n+1}) \in \Hom(B, Q_1)$, where $r: Q \ra Q_1$ is the inclusion.
\end{lemma}
\begin{proof}
We give a graphical calculus in figure \ref{s20}, where the first equality is from Lemma \ref{lem hn n-1}, and the last equality is from (\ref{eq PP-1}) as in the top right picture in figure \ref{s19}.
\end{proof}
\begin{figure}[h]
\begin{overpic}
[scale=0.25]{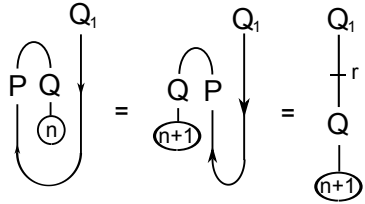}
\end{overpic}
\caption{In the rightmost picture, the little bar labeled by $r$ denotes the inclusion $r: Q \ra Q_1$.}
\label{s20}
\end{figure}

In particular, $^*(^*h_0)=r \circ (h_{1}) \in \Hom(B, Q_1)$.
We express the cone of $^*(^*h_0)$ via the cones of $r$ and $h_1$.
Let $Q_{1,1}=(B[-1] \xra{h_1} Q)$, and $r_1: Q \ra Q_{1,1}$ denotes the inclusion.
Define $$Q_2=(B[-1]\oplus B[-1] \xra{(h_0, h_1)} Q),$$ where $Q$ is in degree zero.
There exists a distinguished triangle $B[-1] \xra{r_1\circ h_0} Q_{1,1} \ra Q_2 \xra{[1]} B[-1]$ in $\ch$.
The octahedral axiom of a triangulated category implies that there exists a diagram:
$$
\xymatrix{
 B[-1] \ar[rr]^{h_1} &               & Q \ar[rr]^{r} \ar[dl]^{r_1} &               &  Q_1 \ar[dl] \\
                 & Q_{1,1} \ar[ul] \ar[dr] & & B \ar[ul]^{h_0} \ar[ll]^{r_1\circ h_0} &   \\
                 &               & Q_2 \ar[ur] &               &
}$$
such that $Q_2$ is isomorphic to $\wt{Q}_2$, the cone of $^*(^*h_0)=r \circ (h_{1}) \in  \Hom(B[-1], Q_1)$.

\begin{defn}
For $n \ge 0$, define $Q_n$ as a complex $(B[-1]^{\oplus n} \xra{(h_0, \dots, h_{n-1})} Q)$,
and dually define $P_n$ as $(P \xra{^*h_0, \dots, ^*h_{n-1}} B[1]^{\oplus n})$, where $P$ and $Q$ are in cohomological degree zero.
In particular, $P_0=P$ and $Q_0=Q$.
\end{defn}

Using a similar argument of proving $Q_{2} \cong \wt{Q}_{2}$, one can prove the following result.
\begin{prop} \label{prop adj positive}
There is a half infinite chain of objects
$$\cdots \leftrightarrow P_n \leftrightarrow Q_n \leftrightarrow P_{n-1}\leftrightarrow \cdots \leftrightarrow Q_{1} \leftrightarrow P_{0} \leftrightarrow Q_{0},$$
such that each pair of objects form an adjoint pair in $\ch$.
\end{prop}

Combining Propositions \ref{prop adj negative} and \ref{prop adj positive}, we finally obtain the following.
\begin{thm} \label{thm adj}
There is an infinite chain of objects
$$\cdots \leftrightarrow P_n \leftrightarrow Q_n \leftrightarrow \cdots \leftrightarrow Q_{1} \leftrightarrow P_{0} \leftrightarrow Q_{0} \leftrightarrow P_{-1} \leftrightarrow \cdots \leftrightarrow P_{-n} \leftrightarrow Q_{-n} \leftrightarrow \cdots,$$
such that

\n(1) $P_0=P, Q_0=Q$;

\n(2) each pair of objects form an adjoint pair in $\ch$;

\n(3) there are exact triangles $B \ra P_n \ra P_{n-1} \xra{[1]} B$, and $Q_n \ra Q_{n+1} \ra B \xra{[1]} Q_n$, for all $n$.
\end{thm}

\section{The fermionic Fock space and the Clifford algebra}
In this section, we extend the construction of $\mathbb{F}_2$-linear DG categories in \cite{Tian} to the ground field $\K$ of characteristic zero.

Recall the integral fermionic Fock space $V_F$ from \cite[Section 3]{Tian}.
Let $V_F$ be a free abelian group with a basis of semi-infinite increasing sequences of even integers:
\begin{gather} \label{eq cim}
\cim = \{\mf{x}=(x_1, x_2, \dots) ~|~ x_1 < x_2 < \cdots \in 2\Z, x_i=x_{i-1}+2 ~\mbox{for}~ i\gg0\}.
\end{gather}
There is a decomposition $\cim = \bigsqcup\limits_{k \in \Z} \cim_{k}$, where
\begin{gather} \label{eq cim_k}
\cim_k = \{(x_1, x_2, \dots) \in \cim ~|~ x_i=2i+2k ~\mbox{for}~ i\gg0\}.
\end{gather}
Elements of $\cim_k$ are called sequences of {\em charge} $k$.

For later use, let $\cm$ denote the set of semi-infinite sequences of even integers:
\begin{gather} \label{eq cm}
\cm = \{\mf{x}=(x_1, x_2, \dots) ~|~ x_i \in 2\Z, x_i=x_{i-1}+2 ~\mbox{for}~ i\gg0\},
\end{gather}
and $\cm_k$ denote its subset of sequences of charge $k$.
%Let $\cal{NDM}$ denote a subset of $\cm$ consisting of non-decreasing sequences.
%There is a chain of inclusion: $\cim \subset \cal{NDM} \subset \cm$.

For $j\in \Z$, let $\psi_j$ and $\psi_j^*$ be the {\em creating} and {\em annihilating} operators acting on the basis of $V_F$ by:
\begin{align*}
\psi_j(x_1, x_2, \dots)=& \left\{
\begin{array}{ll}
0 & \mbox{if} \hspace{0.3cm} 2j=x_n ~\mbox{for some}~ n; \\
(-1)^{n}(x_1, \dots, x_n, 2j, x_{n+1}, \dots) & \mbox{if} \hspace{0.3cm} x_n < 2j < x_{n+1}.
\end{array}\right. \\
\psi_j^*(x_1, x_2, \dots)=& \left\{
\begin{array}{ll}
0 & \quad \mbox{if} \hspace{0.3cm} 2j\neq x_n ~\mbox{for all}~ n; \\
(-1)^{n-1}(x_1, \dots, x_{n-1}, x_{n+1}, \dots) & \quad \mbox{if} \hspace{0.3cm} 2j= x_n.
\end{array}\right. \\
\end{align*}

Define $Cl$ as an algebra with generators $t_i$ for $i \in \Z$ and relations:
\begin{gather*}
t_i^2=0;\\
t_it_j=-t_j t_i \hspace{.2cm} \mbox{if} \hspace{.2cm} |i-j|>1;\\
t_i t_{i+1} + t_{i+1}t_i=1.
\end{gather*}
%The connection between two Clifford algebras is given by the following homomorphism:
%\begin{gather} \label{eq clcl'}
%\begin{array}{ccc}
% Cl& \ra & Cl' \\
 % t_{2j} & \mapsto & \psi_j \\
 % t_{2j-1} & \mapsto & \psi_j^*+\psi_{j-1}^*.
%\end{array}
%\end{gather}

There is an action of $Cl$ on $V_F$ given by:
$$t_{2j}:=\psi_j, \quad\quad t_{2j-1}:=\psi_j^*+\psi_{j-1}^*.$$

%\begin{rmk} The classic Clifford algebra $Cl'$ uses the generators $\psi_j$ and $\psi_j^*$. \end{rmk}

\subsection{The DG algebra $R$}
In this subsection, we define the DG $\K$-algebra $R$ using diagrams.
A category of DG $R$-modules will be used to categorify the fermionic Fock space $V_F$.
The construction extends that in the case of $\mathbb{F}_2$, see \cite[Sections 4, 5]{Tian}.

There are three types of {\em elementary diagrams} with even indices:
\be
\item a vertical strand $\mb_i$ with $i$ on both ends;
\item a crossing $cr_{i,j}$ with $(j, i)$ on the bottom and $(i, j)$ on the top;
\item a dotted strand $dot_i$ with $i+2$ on the bottom and $i$ on the top.
\ee
All indices of diagrams are even integers, see figure \ref{fig v1}.
The grading of an elementary diagram is zero except that $\deg(cr_{i,j})=1$ if $i<j$, and $\deg(cr_{i,j})=-1$ if $i \geq j$.
\begin{figure}[h]
\begin{overpic}
[scale=0.2]{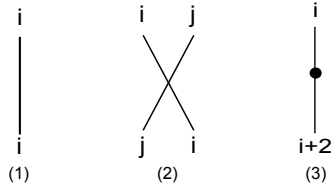}
\end{overpic}
\caption{Elementary diagrams.}
\label{fig v1}
\end{figure}

A generating diagram is obtained from an elementary diagram by horizontally stacking finitely many vertical strands $1_{\mf{a}}$ on the left and infinitely many vertical strands $1_{\mf{x}}$ on the right, where $\mf{a}$ is a finite sequence of even integers, and $\mf{x} \in \cm$ defined in (\ref{eq cm}).
A generating diagram is of type (s) if the corresponding elementary diagram is of type (s), for $s=1,2,3$.
Its grading is the same as the corresponding elementary diagram.
A generating diagram of type (1), (2), or (3) is denoted by $\mb_{\mf{x}}, cr_{i,j}(\mf{a}, \mf{x})$, or $dot_i(\mf{a}, \mf{x})$, respectively, see figure \ref{fig v2}.
\begin{figure}[h]
\begin{overpic}
[scale=0.3]{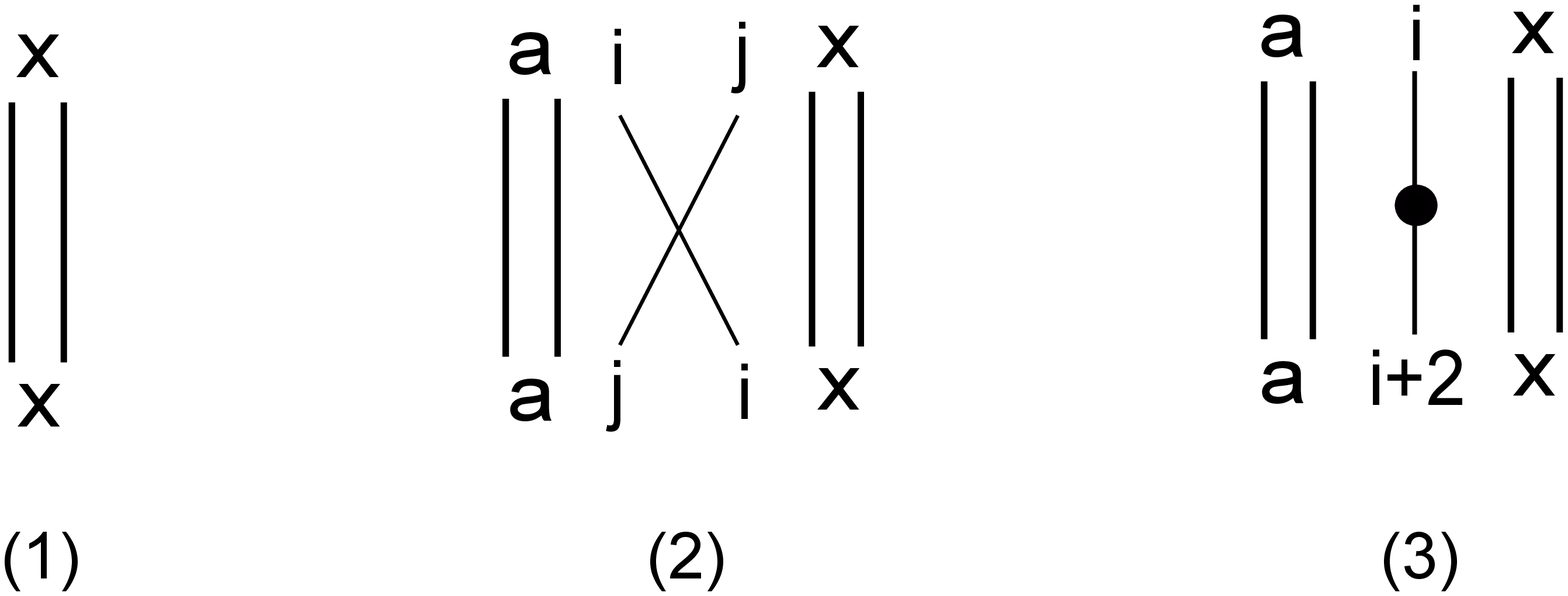}
\end{overpic}
\caption{Generating diagrams.}
\label{fig v2}
\end{figure}

The graded $\K$-algebra $R$ is generated by the generating diagrams, subject to some local relations.
The multiplication $fg$ of two diagrams $f$ and $g$ is the vertical concatenation of $f$ and $g$, where $f$ is at the bottom and $g$ is at the top.
The vertical concatenation of two diagrams is zero unless their endpoints match.

The local relations consist of five groups, see figure \ref{fig v3}:
\be
\item Isotopy relation: a vertical strand is an idempotent; disjoint diagrams super commute.
\item Double dot relation: a strand with two dots is zero.
\item Double crossing relation: $cr_{j,i} \cdot cr_{i,j}=(1-\delta_{i,j})\mb_{i,j}$.
\item Triple intersection moves for all labels.
\item Dot-slide relation: a dot can slide from right to left under a double crossing.
\ee
\begin{figure}[h]
\begin{overpic}
[scale=0.25]{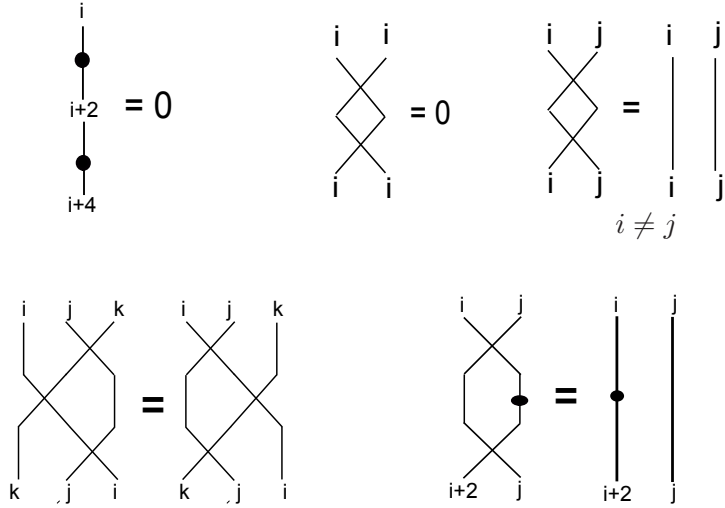}
\put(85,38){$i \neq j$}
\end{overpic}
\caption{Local relations of $R$.}
\label{fig v3}
\end{figure}

A generating diagram $dot_{i}(\mf{a}, \mf{x})$ of type (3) is called {\em restricted} if $x_n > i+2$ for all $n$, where $\mf{x}=(x_1, x_2, \dots) \in \cal{M}$.
Such a diagram is called of type (3r).
Using the dot-slide relation, any generating diagram with a dot can be written as a composition of diagrams such that the dot is moved from left to right.
Thus, the algebra $R$ is generated by generating diagrams of type (1), (2), and (3r).

Define a differential on the generators of $R$ as
\begin{align*}
d(r)=& \left\{
\begin{array}{ll}
\mb_{\mf{y}} & \quad\mbox{if} \hspace{0.3cm} r=cr_{i,i}(\mf{a}, \mf{x}); \\
0 & \quad\mbox{otherwise}.
\end{array}\right.
\end{align*}
where $r$ is a generating diagram of type (1), (2) or (3r), and $\mf{y}=(a_1, \dots, a_n, i, i, x_1, x_2, \dots) \in \cal{M}$.
The differential is extended to $R$ by Leibniz rule $d(r_1r_2)=d(r_1)r_2+(-1)^{\deg(r_1)}r_1d(r_2)$.

\begin{lemma} \label{lem welldef R}
The differential $d$ on $R$ is well-defined, i.e. it preserves the local relations, $\deg(d)=1$, and satisfies $d^2=0$.
\end{lemma}
\begin{proof}
Since $d(r)$ is nontrivial only for $r=cr_{i,i}(\mf{a}, \mf{x})$, it is enough to show that $d$ preserves the double crossing and triple crossing relations involving $cr_{i,i}$.
We check the double crossing relation:
$$d(cr_{i,i} \cdot cr_{i,i})=d(cr_{i,i})\cdot cr_{i,i} - cr_{i,i} \cdot d(cr_{i,i})=\mb_{i,i}\cdot cr_{i,i}-cr_{i,i}\cdot \mb_{i,i}=0=d(\mb_{i,i}).$$
The proof for the triple crossing relation is similar.

By definition $\deg(d)=1$, and $d^2=0$ follows from that $d^2(r)=0$ for any generator $r$.
\end{proof}

\begin{rmk}
(1) The definition of the differential here is simplified comparing to that in \cite[Section 4.2]{Tian}, since we reduce the generators of $R$ from type (3) to type (3r).

\n (2) The differential does not satisfy Leibniz rule with respect to the horizontal concatenation of diagrams.
In other words, the diagrammatic DG category defined by $R$ is not monoidal.

\n (3) The generators only involving $cr_{i,i}$ generate the graded nilCoxeter algebra.
\end{rmk}

Since a double crossing with the same indices is zero, the dot-slide relation is not symmetric with respect to left and right, i.e. a dot might not be able to slide from left to right.
The obstruction is from the double crossing relation $cr_{i,i}^2=0$, see figure \ref{fig v4} for an example.
%On the other hand, a dot in a diagram can always slide from left to right if the diagram is nonzero.
\begin{figure}[h]
\begin{overpic}
[scale=0.4]{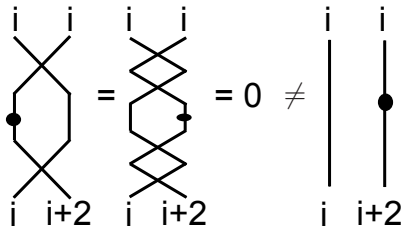}
\put(68,30){$\neq$}
\end{overpic}
\caption{A dot may not slide from left to right.}
\label{fig v4}
\end{figure}

There is a distinguished element $\es(k)=(2k+2, 2k+4, \dots) \in \cim_k$, called the {\em vacuum state} of charge $k$.
The {\em energy} $\op{E}_k$ is defined on $\cm_k$ by
$$\begin{array}{cccc}
\op{E}_k:& \cm_k & \ra & \Z \\
& (x_1, x_2, \dots)  & \mapsto & \frac{1}{2} \sum\limits_{i=1}^{\infty}(2k+2i-x_i),
\end{array}$$
where only finitely many terms are nonzero in the sum.
The energy measures the total difference between the vacuum state $\es(k)$ and $\mf{x} \in \cm_k$.
In particular, the vacuum object has energy zero, and $E_k(\mf{x}) \ge 0$ for $\mf{x} \in \cim_k$.
Note that $\op{E}_k(\mf{x})<0$ implies that $\mf{x}$ contains repetitive numbers: $x_i=x_j$ for some $i \neq j$.

\vspace{.2cm}
The following properties of the DG algebra $R$ are proved in the case of $\mathbb{F}_2$ in \cite[Sections 4, 5]{Tian}.
They also hold in characteristic zero with the same proofs.
We give a summary here, and refer the reader to \cite{Tian} for more detail.

\vspace{.2cm}
\n {\bf Properties of the graded algebra $R$:}
\be
\item There is a decomposition
$R=\bigoplus\limits_{\mf{x}, \mf{y} \in \cm}\mb_{\mf{x}}\cdot R \cdot \mb_{\mf{y}}$
into subspaces, where $\mb_{\mf{x}}\cdot R \cdot \mb_{\mf{y}}$ is spanned by diagrams with indices $\mf{x}$ at the bottom and $\mf{y}$ at the top.
\item The space $\mb_{\mf{x}}\cdot R \cdot \mb_{\mf{y}}=0$ if $\mf{x}$ and $\mf{y}$ have different charges since the generating diagrams do not change the charge.
Thus, there is a decomposition $R=\bigoplus\limits_{k \in \Z}R_k$ into subalgebras, where $R_k=\bigoplus\limits_{\mf{x}, \mf{y} \in \cm_k}\mb_{\mf{x}}\cdot R \cdot \mb_{\mf{y}}$.
Moreover, all $R_k$'s are isomorphic to each other.
\item The space $\mb_{\mf{x}}\cdot R_k \cdot \mb_{\mf{y}}=0$ if $\op{E}_k(\mf{x})>\op{E}_k(\mf{y})$ for $\mf{x}, \mf{y} \in \cm_k$, since the generating diagrams do not decrease the energy from bottom to top.
\item Each subspace $\mb_{\mf{x}}\cdot R \cdot \mb_{\mf{y}}$ is finite dimensional, see \cite[Corollary 4.19]{Tian}.
\item The algebra $\mb_{\mf{x}}\cdot R \cdot \mb_{\mf{x}}$ is one dimensional generated by $\mb_{\mf{x}}$ if $\mf{x}=(x_1, x_2, \dots)$ does not contain repetitive numbers, i.e. $x_i \neq x_j$ for $i \neq j$.
\item If $\mf{x}$ contains repetitive numbers, then the algebra $\mb_{\mf{x}}\cdot R \cdot \mb_{\mf{x}}$ is a graded tensor product of nilCoxeter algebras.
\ee

Let $\op{N}(\mf{x}) \in \cm$ denote the rearrangement of $\mf{x} \in \cm$ such that the sequence $\op{N}(\mf{x})$ is non-decreasing.
It is called the {\em normalization} of $\mf{x}$.
%Let $\eta(\mf{x})=|\{(i,j)\in \Z^2 ~|~ i<j, ~x_i>x_j\}|$ denote the number of decreasing pairs in $\mf{x}$.
%Let $\cal{NDM}$ denote the set of nondecreasing sequences which contains $\cim$ as a subset.
Since the crossings $cr_{i,j}$ and $cr_{j,i}$ are inverse to each other up to grading shifts,
%the projective DG module $P(\mf{x})=R\cdot \mb_{\mf{x}}$ is isomorphic to $P(\op{N}(\mf{x}))$ with a grading shift by $\eta(\mf{x})$.
the algebra $\bigoplus\limits_{\op{N}(\mf{x})=\op{N}(\mf{y})}\mb_{\mf{x}}\cdot R \cdot \mb_{\mf{y}}$ is a graded infinite matrix algebra over its subalgebra $\mb_{N(\mf{x})}\cdot R \cdot \mb_{N(\mf{x})}$.

%(7) The DG algebra $R$ is derived Morita equivalent to its subalgebra $\bigoplus\limits_{\mf{x}, \mf{y} \in \cal{NDM}}\mb_{\mf{x}}\cdot R \cdot \mb_{\mf{y}}$.

\vspace{.2cm}
\n {\bf Properties of the cohomology algebra $H(R)$:}
\be
\item The cohomology $H(\mb_{\mf{x}}\cdot R \cdot \mb_{\mf{y}})=0$ if $\mf{x}$ or $\mf{y}$ contains repetitive numbers, essentially because the cohomology of the nilCoxeter algebra is zero, see \cite[Lemma 5.1]{Tian}.
\item Let $\cal{NRM}$ denote a subset of $\cm$ consisting of sequences containing no repetitive numbers.
The differential on $\mb_{\mf{x}}\cdot R \cdot \mb_{\mf{y}}$ is zero for $\mf{x}, \mf{y} \in \cal{NRM}$.
Thus, the cohomology $H(\mb_{\mf{x}}\cdot R \cdot \mb_{\mf{y}})$ is naturally isomorphic to $\mb_{\mf{x}}\cdot R \cdot \mb_{\mf{y}}$.
\item The cohomology algebra $H(R)$ is naturally isomorphic to $\bigoplus\limits_{\mf{x}, \mf{y} \in \cal{NRM}}\mb_{\mf{x}}\cdot R \cdot \mb_{\mf{y}}$. We will identify them from now on.
\ee
Recall that $\cim$ is the set of strictly increasing sequences, see (\ref{eq cim}).
Any sequence $\mf{x} \in \cal{NRM}$ has a unique representative $N(\mf{x}) \in \cim$.

\begin{prop} \label{prop HR}
The DG algebra $R$ is formal, i.e. it is quasi-isomorphic to its cohomology $H(R)$ with the trivial differential.
The DG algebra $H(R) = \bigoplus\limits_{\mf{x}, \mf{y} \in \cal{NRM}}\mb_{\mf{x}}\cdot R \cdot \mb_{\mf{y}}$ is a graded infinite matrix algebra over its subalgebra $\wt{H}(R)=\bigoplus\limits_{\mf{x}, \mf{y} \in \cim}\mb_{\mf{x}}\cdot R \cdot \mb_{\mf{y}}$.
In particular, $H(R)$ and $\wt{H}(R)$ are derived Morita equivalent.
\end{prop}

The DG algebra $\wt{H}(R)=\bigoplus\limits_{k \in \Z}\wt{H}(R_k)$, where $\wt{H}(R_k)$'s are all isomorphic to each other.
We summarize the relations as:
\begin{gather} \label{eq rel R HR}
R \xrightarrow{\op{quasi-iso}} H(R) \xrightarrow{\op{derived~~Morita}} \wt{H}(R) \cong \bigoplus\limits_{\Z} \wt{H}(R_0)
\end{gather}

\subsection{The algebra $\wt{H}(R_0)$}
The DG algebra $\wt{H}(R_0)$ has a simple presentation in terms of generators and relations.
We give a brief summary here, see \cite[Proposition 5.5]{Tian} for the proofs.

\vspace{.2cm}
\n{\bf Properties of the algebra $\wt{H}(R_0)$:}
\be
\item The space $\mb_{\mf{x}}\cdot \wt{H}(R_0) \cdot \mb_{\mf{y}}$ is at most one dimensional, for $\mf{x}, \mf{y} \in \cim_0$.
\item The space $\mb_{\mf{x}}\cdot \wt{H}(R_0) \cdot \mb_{\mf{y}}$ is one dimensional for $\mf{x}, \mf{y} \in \cim_0$ if and only if there exists a finite index set $\cal{I} \subset \Z_+$ such that $y_i=x_i-2$ for $i \in \cal{I}$ and $y_i=x_i$ for $i \notin \cal{I}$. It is generated by a diagram with $|\cal{I}|$ dots and no crossing. We write such a generator as $\mf{x} \xrightarrow{\cal{I}} \mf{y}$.
\item When $\cal{I}=\es$, the space $\mb_{\mf{x}}\cdot \wt{H}(R_0) \cdot \mb_{\mf{x}}$ is generated by $\mb_{\mf{x}}$.
\item As an algebra, $\wt{H}(R_0)$ is generated by $\mb_{\mf{x}}$, and $\mf{x} \xrightarrow{\cal{I}} \mf{y}$ for $\cal{I}=\{i\}$. The relations consist of two groups:
    \begin{gather} \label{eq rel HR0}
     (\mf{x} \xrightarrow{\{i\}} \mf{y})(\mf{y} \xrightarrow{\{j\}} \mf{z})=(\mf{x} \xrightarrow{\{j\}} \mf{y'})(\mf{y'} \xrightarrow{\{i\}} \mf{z}), ~i \neq j; \qquad
     (\mf{x} \xrightarrow{\{i\}} \mf{y})(\mf{y} \xrightarrow{\{i\}} \mf{z})=0.
     \end{gather}
    The relations are induced by the isotopy relation of dots on different strands, and the double dots relation, respectively.
\item The DG algebra $\wt{H}(R_0)$ with the trivial differential is concentrated at degree zero, since any dot has degree zero.
      So it can be viewed as an ordinary algebra.
\ee
Some generators of $\wt{H}(R_0)$ between states of low energy are described in figure \ref{fig v5}.
\begin{figure}[h]
\begin{overpic}
[scale=0.4]{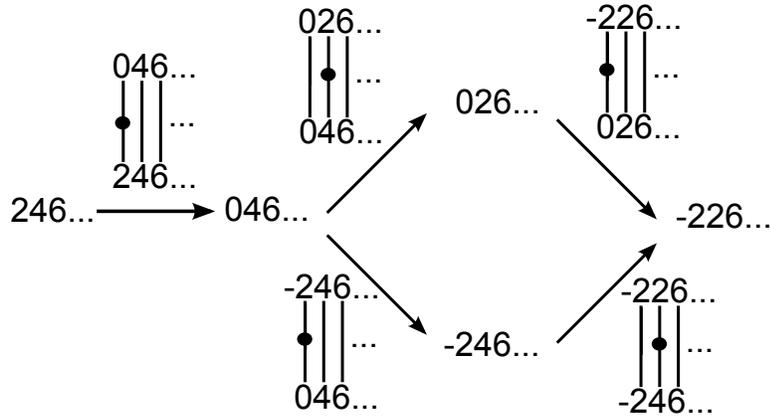}
\end{overpic}
\caption{Some generators of $\wt{H}(R_0)$.}
\label{fig v5}
\end{figure}

According to Property (4) above, $\wt{H}(R_0)$ is a quiver algebra whose set of vertices is $\cim_0$, the set of semi-infinite increasing sequences of charge $0$.
Let $\Par$ denote the set of all partitions.
There is a bijection $\mf{x}: \Par \ra \cim_{0}$:
\begin{gather} \label{eq iso vfk}
\mf{x}(\lm)=(2-2\ov{\lm}_1, 4-2\ov{\lm}_2, \dots, 2k-2\ov{\lm}_k, 2(k+1), \dots),
\end{gather}
where $\ov{\lm}=(\ov{\lm}_1, \ov{\lm}_2, \dots, \ov{\lm}_k)$ is the dual partition of $\lm$.
We fix the bijections $\mf{x}$ from now on.
%Note that this bijection is different from the one used in \cite{Tian}.
Under the bijection, each generator $(\mf{x} \xrightarrow{\{i\}} \mf{y})$ of $\wt{H}(R_0)$ can be rewritten as $r(\lm|\mu)$, where $\mf{x}=\mf{x}(\lm), \mf{y}=\mf{x}(\mu)$.
The relations of $\wt{H}(R_0)$ in (\ref{eq rel HR0}) are exactly those of $F$ in Definition \ref{def F}.

\begin{prop} \label{prop HR0 F}
The algebras $F$ and $\wt{H}(R_0)$ are naturally isomorphic under the bijection $\mf{x}$.
\end{prop}

\subsection{The fermionic Fock space.}
The categorification of the fermionic Fock space $\vf$ is the same as in \cite[Sections 5.2]{Tian}.
We refer to \cite[Section 10]{BL} for an introduction to DG modules and {\em projective} DG modules, and to \cite{Ke} for an introduction to DG categories and their homology categories.

For the DG algebra $R$, let $DG(R)$ denote the DG category of DG $R$-modules.
%The cohomology algebra $H(R)$ is a DG algebra with a trivial differential.
%We are interested in a full subcategory of $DG(R)$ generated by some projective DG modules.
Let $P(\mf{x})=R \cdot \mb_{\mf{x}}$ denote the projective DG $R$-module associated to $\mf{x} \in \cal{M}$.
Define $DGP(R)$ as the smallest full DG subcategory of $DG(R)$ which contains $\{P(\mf{x}); \mf{x} \in \cm\}$ and is closed under the cohomological grading shift functor $[1]$ and taking mapping cones.
\begin{defn} \label{def cf}
Define $\cf$ as the homotopy category of $DGP(R)$. It decomposes as a disjoint union $\bigcup\limits_{k \in \Z} \cf_k$, where $\cf_k$ is the homotopy category of $DGP(R_k)$.
\end{defn}

%Similarly, let $\wt{P}(\mf{x})=\wt{H}(R) \cdot \mb_{\mf{x}}$ denote the projective $\wt{H}(R)$-module associated to $\mf{x} \in \cim$. Define $DGP(\wt{H}(R))$ as the smallest full DG subcategory of $DG(\wt{H}(R))$ generated by $\wt{P}(\mf{x})$'s.

The relations in (\ref{eq rel R HR}) imply that $\cf_k$ is equivalent to the homotopy category of $DGP(\wt{H}(R_k))$.
The later is naturally isomorphic to $\Kom(F)$ for all $k$ by Proposition \ref{prop HR0 F}.
Here, $\Kom(F)$ is the homotopy category of finite dimensional projective $F$-modules.
There are canonical isomorphisms of their Grothendieck groups:
$$K_0(\cf) = \bigoplus\limits_{k \in \Z} K_0(\cf_k) \cong \bigoplus\limits_{k \in \Z} K_0(\Kom(F)).$$
It follows from the representation theory of quiver algebras that $K_0(\Kom(F)) \cong \Z\lan \Par \ran$, the free abelian group with a basis $\Par$.
%See \cite[Lemma 5.8]{Tian} for the proof.
Let $V_F=\bigoplus\limits_{k \in \Z}(V_F)_k$, where $(V_F)_k$ has an integral basis $\cim_k$, see (\ref{eq cim_k}).
To sum up, we have the following categorification of the fermionic Fock space $\vf$.

\begin{prop} \label{prop K0 vf}
There is a canonical isomorphism of abelian groups $K_0(\cf_k) \ra (\vf)_k$ which maps $[P(\mf{x})]$ to $\mf{x} \in \cim_k$.
Taking the union of $k$, we have $K_0(\cf) \cong \vf$.
\end{prop}

\subsection{The Clifford category}

To categorify the Clifford algebra, we extend the construction in \cite[Section 6]{Tian} from $\mathbb{F}_2$ to characteristic zero.
We need carefully deal with the signs.
Recall the definitions of shifts of left and right DG modules from \cite[Sections 10.3, 10.6.3]{BL} as follows.
A DG left $A$ and right $B$ module $M$ is denoted as a $(A,B)$-bimodule.
Let $M[1]$ denote the shift of $M$, i.e.
$$M[1]^i=M^{i+1},  \quad d_{M[1]}=-d_M, \quad a \cdot m \cdot b=(-1)^{|a|}amb,$$
for $m \in M[1], a \in A, b \in B$, where $a \cdot m \cdot b$ is the multiplication in $M[1]$, and $amb$ is the multiplication in $M$.
In other words, the left $A$ action is twisted in $M[1]$.
We will focus on the case $A=B=R$, i.e. DG $R$-bimodules.
Given a homomorphism of $R$-bimodules $u: M \ra N$, the cone $C(u)$ of $u$ is defined as
$(M[1]\oplus N, d=(d_{M[1]}+f, d_N)$.
Consider the DG category $DG(R^e)$ of DG $R$-bimodules.
Let $M \ot_R N$ denote the DG $R$-bimodule of the tensor product over $R$, where the differential is $d(m \ot n)=d(m) \ot n+(-1)^{|m|}m \ot d(n)$.
We will simply write $M \ot N$ for $M \ot_R N$.

For each $i \in \Z$, there is an inclusion of DG algebras:
$$\begin{array} {cccc}
\ii: & R & \ra & R \\
& f & \mapsto & \mb_{2i}\odot f,
\end{array}
$$
where $\mb_{2i}\odot f \in R$ denotes the diagram obtained from a diagram $f \in R$ by adding a vertical strand $\mb_{2i}$ on the left.
More precisely, $\ii$ maps $\mb_{\mf{x}} R \mb_{\mf{y}}$ into $\mb_{(2i)\odot\mf{x}} R \mb_{(2i)\odot\mf{y}}$, where $(2i)\odot\mf{x}$ is the concatenation of sequences $(2i)$ and $\mf{x}$.
Since $(2i)\odot\mf{x} \in \cm_{k-1}$ for $\mf{x} \in \cm_k$, we have $\ii(R_k) \subset R_{k-1}$.

\begin{defn} \label{def Ti}
For any $i \in \Z$, define two DG $R$-bimodules
\begin{gather*}
T(2i)=\bigoplus_{k\in\Z}T(2i)_k=\bigoplus_{k\in\Z}\left(\bigoplus_{\mf{x}\in\cm_{k-1}, \mf{y}\in\cm_k}\mb_{\mf{x}} R \mb_{(2i)\odot\mf{y}}\right),\\
T(2i-1)=\bigoplus_{k\in\Z}T(2i-1)_k=\bigoplus_{k\in\Z}\left(\bigoplus_{\mf{y}\in\cm_k,\mf{z}\in\cm_{k+1}}
\mb_{(2i)\odot\mf{z}} R \mb_{\mf{y}}\right).
\end{gather*}
The actions of $R$ are given by $f \cdot t \cdot g=f~t~\ii(g)$, and $f \cdot t' \cdot g=\ii(f)~t'~g$, where $f~t~\ii(g)$ and $\ii(f)~t'~g$ are multiplications in $R$, for $f,g \in R$, $t \in T(2i)$, and $t' \in T(2i-1)$.
\end{defn}

The bimodules $T(2i)$ and $T(2i-1)$ correspond to the induction and restriction functors with respect to $\ii$, respectively.
The summand $T(2i)_k$ is a $(R_{k-1},R_k)$-bimodule, and $T(2i-1)_k$ is a $(R_{k+1},R_k)$-bimodule.
See figure \ref{fig v6} for a graphic description of $T(2i), T(2i+1)$.
Since the left strand of $T(2i)$ with the index $2i$ is unchanged under the right action, we call it the frozen strand of $T(2i)$, and add a little bar at its top.
Similarly, we call the left strand of $T(2i-1)$ with the index $2i$ the frozen strand of $T(2i-1)$, and add a little bar at its bottom.
\begin{figure}[h]
\begin{overpic}
[scale=0.25]{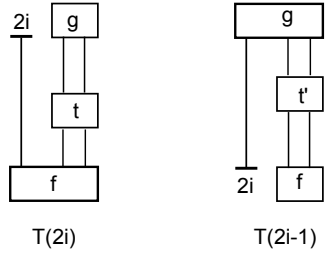}
\end{overpic}
\caption{Two families of $R$-bimodules $T(2i)$ and $T(2i-1)$.}
\label{fig v6}
\end{figure}

\begin{prop} \label{prop Cl relation}
There are canonical isomorphisms of DG $R$-bimodules:

\n(1) $T(2i-1)\ot T(2i) \cong C(u_i)$ for some $u_i: T(2i)\ot T(2i-1) \ra R$.

\n(2) $T(2i+1)\ot T(2i) \cong C(\wt{u}_i)[1]$ for some $\wt{u}_i: R \ra T(2i)\ot T(2i+1)$.
\end{prop}
\begin{proof}
The proofs are almost the same as the discussion in \cite[Section 6.2]{Tian} except that some signs are needed now.
To prove (1), we define two maps $v_i ,w_i$ as:
$$\begin{array} {cccc}
v_i: & T(2i)\ot T(2i-1)[1] & \ra & T(2i-1)\ot T(2i) \\
& m \ot n & \mapsto & (-1)^{|m|}(\mb_{2i}\odot m) \ot cr_{2i,2i}(\mf{z}) \ot (\mb_{2i} \odot n),\\
w_i: & R & \ra & T(2i-1)\ot T(2i) \\
& f & \mapsto &  (2i) \odot f,
\end{array}
$$
for $m \in T(2i), n\in T(2i-1), f \in R$.
See figure \ref{fig v7}.
\begin{figure}[h]
\begin{overpic}
[scale=0.25]{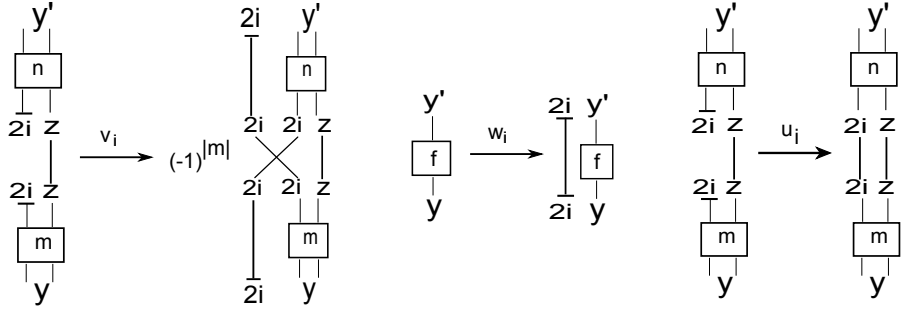}
\end{overpic}
\caption{The maps $v_i, w_i, u_i$.}
\label{fig v7}
\end{figure}

The factor $(-1)^{|m|}$ in the definition of $v_i$ is needed to prove that the map is well-defined:
\begin{align*}
v_i(ma \ot n)= &(-1)^{|ma|} (\mb_{2i}\odot ma) \ot cr_{2i,2i}(\mf{z}) \ot (\mb_{2i} \odot n) \\
= & (-1)^{|ma|} (\mb_{2i}\odot m) \ot  (\mb_{2i,2i} \odot a) cr_{2i,2i}(\mf{z})   \ot (\mb_{2i} \odot n) \\
= & (-1)^{|ma|} (-1)^{|a|} (\mb_{2i}\odot m) \ot   cr_{2i,2i}(\mf{z}) (\mb_{2i,2i} \odot a)  \ot (\mb_{2i} \odot n) \\
= & (-1)^{|ma|+|a|} (\mb_{2i}\odot m) \ot   cr_{2i,2i}(\mf{z})  \ot (\mb_{2i} \odot an) \\
= & (-1)^{|m|} (\mb_{2i}\odot m) \ot   cr_{2i,2i}(\mf{z})  \ot (\mb_{2i} \odot an) = v_i(m \ot an).
\end{align*}
By definition, $v_i ,w_i$ are homogeneous degree zero maps of graded $R$-bimodules.
The two maps are injective.
Moreover, the set of diagrams in $T(2i-1)\ot T(2i)$ is divided into two subsets depending on whether the frozen strand in $T(2i)$ connects to the frozen strand in $T(2i-1)$.
So there is an isomorphism  of graded $R$-bimodules:
\begin{gather} \label{eq iso cl}
T(2i-1)\ot T(2i) = \op{Im}(v_i) \oplus \op{Im}(w_i) \cong (T(2i)\ot T(2i-1)[1] ~\oplus~ R).
\end{gather}
The map $w_i$ is actually a map of DG $R$-bimodules.
Let $\op{pr}_i: T(2i-1)\ot T(2i) \ra  R$ denote the projection onto its summand $R$ with respect to the isomorphism above.
It is a map of graded $R$-bimodules.
Finally, we define $u_i: T(2i)\ot T(2i-1) \ra R$ as the composition
$$T(2i)\ot T(2i-1) \xrightarrow{v_i[-1]} T(2i-1)\ot T(2i)[-1] \xrightarrow{d} T(2i-1)\ot T(2i) \xrightarrow{\op{pr}_i} R.$$
It turns out that $u_i$ has a simple description as connecting the frozen strands in $T(2i-1)$ and $T(2i)$, see figure \ref{fig v7}.
In particular, $u_i$ is a map of DG $R$-bimodules.
The definition of $u_i$ implies that the isomorphism (\ref{eq iso cl}) of graded bimodules can be completed into an isomorphism of DG bimodules: $T(2i-1)\ot T(2i)  \cong (T(2i)\ot T(2i-1) \xrightarrow{u_i} R)$, where $R$ is in degree zero.

The proof for (2) is similar. We leave it for the reader, see \cite[Section 6.2.2]{Tian} for more detail.
\end{proof}

\begin{prop} \label{prop Cl relation2}
(1) There is a canonical isomorphisms of DG $R$-bimodules: $T_{j} \ot T_{i} \cong T_{i} \ot T_{j} [c_{i,j}]$ for $|i-j|>1$, where $c_{i,j}=(-1)^{i-j}$ if $i<j$, otherwise $c_{i,j}=(-1)^{i-j+1}$.

\n(2) The DG $R$-bimodule $T_{j} \ot T_{j}$ is contractible for any $j \in \Z$.
\end{prop}
\begin{proof}
The proof of (1) is similar to that in Proposition \ref{prop Cl relation}, see \cite[Section 6.2.4]{Tian}.

To prove (2), we define a map for $j=2i$:
$$\begin{array} {cccc}
h_{2i}: & T(2i) \ot T(2i) & \ra & T(2i) \ot T(2i)[-1] \\
& m & \mapsto & (-1)^{|m|}m \cdot cr_{2i,2i}(\mf{z}),
\end{array}
$$
for $m \in \mb_{\mf{x}}\cdot T(2i) \ot T(2i) \cdot \mb_{\mf{z}}$, see figure \ref{fig v8}.
The factor $(-1)^{|m|}$ in the definition makes $h_{2i}$ a homogeneous degree zero map of graded $R$-bimodules.
\begin{figure}[h]
\begin{overpic}
[scale=0.3]{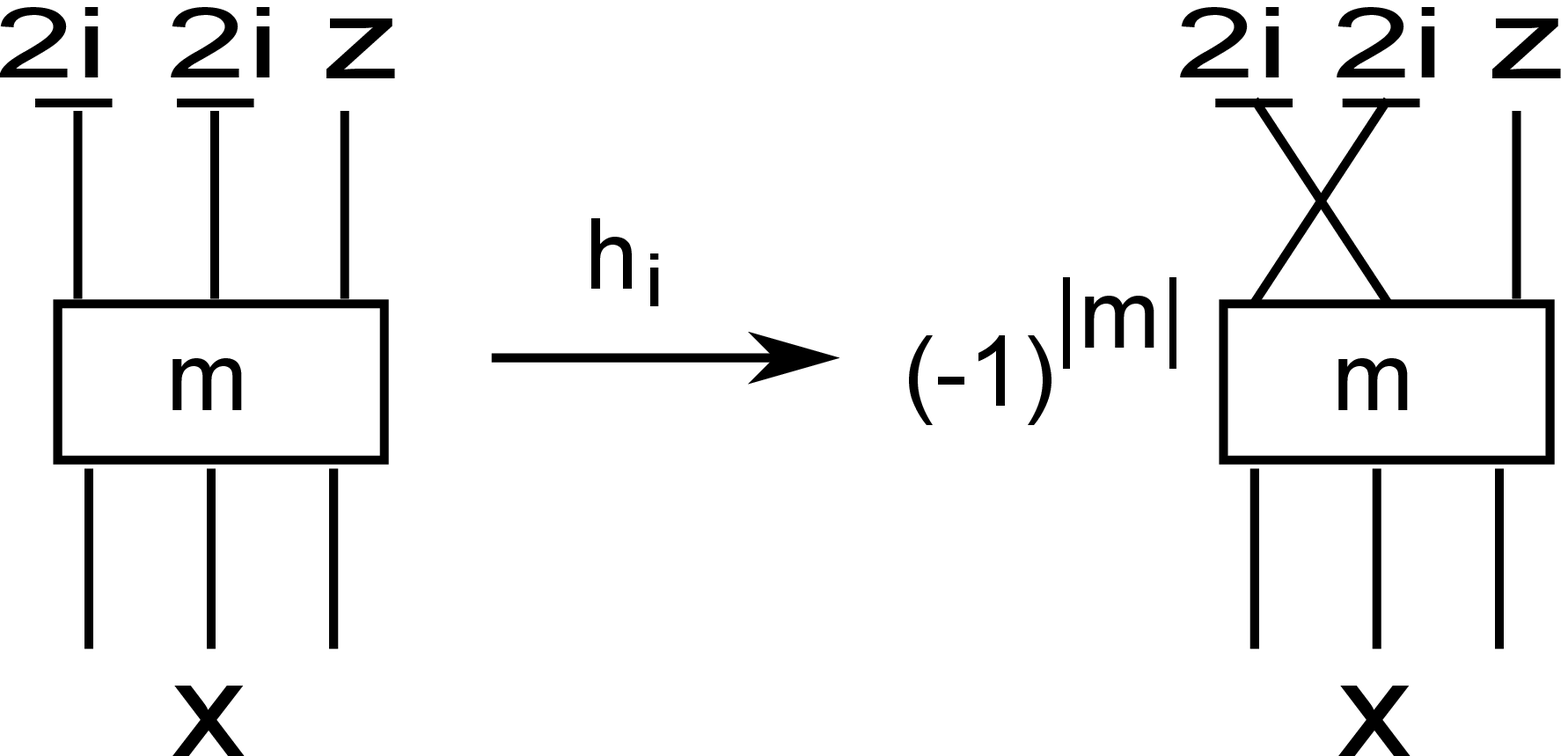}
\end{overpic}
\caption{The map $h_{2i}: T(2i) \ot T(2i) \ra T(2i) \ot T(2i)$.}
\label{fig v8}
\end{figure}

We check that
\begin{align*}
d\circ h_{2i}(m)&=(-1)^{|m|} d(m \cdot cr_{2i,2i}(\mf{z}))\\
&=(-1)^{|m|} d(m) \cdot cr_{2i,2i}(\mf{z})+(-1)^{|m|+|m|}m \cdot d(cr_{2i,2i}(\mf{z}))\\
& = (-1)(-1)^{|d(m)|} d(m) \cdot cr_{2i,2i}(\mf{z}) + m \\
&=-h_{2i} \circ d(m)+m.
\end{align*}
So $d \circ h_{2i} + h_{2i} \circ d=id$, i.e. $h_{2i}$ is the null homotopy.

The proof for the case of $j=2i-1$ is similar, and we leave it to the reader.
\end{proof}

Propositions \ref{prop Cl relation} and \ref{prop Cl relation2} categorify the relations of the Clifford algebra $Cl$.
The following lemma categorifies the Clifford action on the fermionic Fock space, see \cite[Lemma 6.11]{Tian} and the proof there.
\begin{lemma} \label{lem Cl action}
There are canonical isomorphisms of DG $R$-modules:

\n(1) $T(2i) \ot P(\mf{x}) \cong P((2i)\odot \mf{x})$, for any $\mf{x} \in \cm$.

\n(2) $T(2i-1) \ot P(\mf{x})=0$, if the sequence $\mf{x}$ does not contain $2i$ or $2i-2$.

\n(3) $T(2i-1) \ot P(\mf{x}) \in DGP(R)$ for any $\mf{x} \in \cm$.
\end{lemma}

For $\mf{x}$ containing $2i$ or $2i-2$, $T(2i-1) \ot P(\mf{x})$ can be computed using the isomorphisms in Proposition \ref{prop Cl relation} and Lemma \ref{lem Cl action}.
For instance,
\begin{align*}
&T(5) \ot P(2,4,6,8,\dots) \\
 \cong & T(5) \ot T(2) \ot T(4) \ot T(6) \ot P(8,\dots) \\
 \cong & T(2) \ot T(5) \ot T(4) \ot T(6) \ot P(8,\dots) [-1] \\
 \cong & T(2) \ot (R \ra T(4) \ot T(5)) \ot T(6) \ot P(8,\dots)
\end{align*}
\begin{align*}
& \cong  (T(2) \ot T(6) \ot P(8,\dots)) \ra (T(2) \ot T(4) \ot T(5) \ot T(6) \ot P(8,\dots))  \\
& \cong P(2,6,8,\dots) \ra (T(2) \ot T(4) \ot ((T(6) \ot T(5)) \ra R) \ot P(8,\dots)\\
& \cong P(2,6,8,\dots) \ra (T(2) \ot T(4) \ot  P(8,\dots)) \\
& \cong P(2,6,8,\dots) \ra P(2,4,8,\dots)
\end{align*}

In particular, $T(2i-1) \ot P(\mf{x})$ is always a projective DG $R$-module.
Thus, the left DG $R$-module $T(i) \cong \bigoplus\limits_{\mf{x}} T(i) \ot P(\mf{x})$ is projective.
Similarly, one can prove that $T(i)$ is projective as a right DG $R$-module.

\vspace{.2cm}
Let $D(R^e)$ denote the derived category of DG $R$-bimodules.
It is triangulated and admits a monoidal structure given by the derived tensor product over $R$.

\begin{defn} \label{def cl}
Define $\wt{\cl}$ as the smallest monoidal triangulated subcategory of $D(R^e)$ which contains $R$ and $T(i)$'s for all $i \in \Z$.
Let $\cl$ be the Karoubi envelop of $\wt{\cl}$.
\end{defn}

The derived tensor product of $T(i)$ and $T(j)$ is isomorphic to their ordinary tensor product in $\mf{D}(R^e)$.
So the monoidal structure on $\cl$ can be computed by the ordinary tensor product.
The morphisms in Proposition \ref{prop Cl relation} give rise to a chain of adjoint pairs in $\cl$, see \cite[Section 6.3]{Tian}.
\begin{prop} \label{prop Cl adj}
There is a chain of objects:
$$\cdots \leftrightarrow T(i+1) \leftrightarrow T(i) \leftrightarrow T(i-1) \leftrightarrow  \cdots,$$
such that each pair of objects form an adjoint pair in $\cl$.
\end{prop}
A similar chain of adjoint pairs exists in the Heisenberg category $\ch$, see Theorem \ref{thm adj}.

After passing to the derived and homotopy categories, there is a categorical action of $\cl$ on $\cf$ by Lemma \ref{lem Cl action}.
It induces a linear action
\begin{gather} \label{eq cl action}
K_0(\cl) \times K_0(\cf) \ra K_0(\cf).
\end{gather}

\begin{prop} \label{prop K0 cl}
There is an injective homomorphism $\gamma: Cl \ra K_0(\cl)$ of rings given by $\gamma(t_i)=[T(i)]$.
The pull back of the action (\ref{eq cl action}) by $\gamma$ agrees with the Clifford action on the fermionic Fock space $V_F \cong K_0(\cf)$.
\end{prop}
\begin{proof}
The map $\gamma$ is well-defined by Propositions \ref{prop Cl relation} and \ref{prop Cl relation2}.
Under the isomorphism in Proposition \ref{prop K0 vf}, the pull back of $(\ref{eq cl action})$ by $\gamma$ agrees with the Clifford action of $Cl$ on $V_F$ by Lemma \ref{lem Cl action}.
So $\gamma$ is also injective because the Clifford action is faithful.
\end{proof}

We conjecture that the map $\gamma$ is an isomorphism.

\section{Categorical boson-fermion correspondence}
The boson-fermion correspondence provides an isomorphism of the bosonic Fock space $V_B[q,q^{-1}]$ and the fermionic Fock space $V_F$.
Here, $V_B=\Z[x_1,x_2,\dots]$, and $V_F=\bigoplus\limits_{k \in \Z}(V_F)_k$.
In particular, there is an isomorphism $V_B \cong (V_F)_0$ of Fock spaces.
There are categorifications $\cb$ of $V_B$ and $\cf_0$ of $(V_F)_0$, see Propositions \ref{prop K0 vb} and \ref{prop K0 vf}, respectively.

\begin{thm} \label{thm bf fock}
There is an equivalence of triangulated categories $\cb \cong \cf_0$.
\end{thm}
\begin{proof}
Recall that $\cb=\Kom(B)$ and $\cf_0$ is the homotopy category of $DGP(R_0)$.
The equivalence follows from a chain of relations of the algebras:
\begin{gather} \label{eq chain rel}
R_0 \leftrightarrow H(R_0) \leftrightarrow \wt{H}(R_0) \cong F \leftrightarrow B
\end{gather}
where the first arrow tells that the DG algebra $R_0$ is formal, i.e. quasi-isomorphic to its cohomology $H(R_0)$;
the second arrow tells that the DG algebra $H(R_0)$ is derived Morita equivalent to its subalgebra $\wt{H}(R_0)$, see Proposition \ref{prop HR};
the third isomorphism is from Proposition \ref{prop HR0 F};
and the last arrow is a Morita equivalence of the algebras $F$ and $B$, see Proposition \ref{prop Morita BF}.
\end{proof}

%\begin{rmk} \label{rmk rel gl infty} Connection to representation theory of $\mathfrak{gl}(\infty)$. \end{rmk}

In addition to the isomorphism of Fock spaces, the correspondence establishes maps between the Heisenberg and Clifford algebras.
We refer the reader to \cite[Section 14.10]{Kac} for definitions of these maps.
The bases of the Heisenberg and Clifford algebras in this paper are different from those in \cite{Kac}.
We need modify the maps correspondingly.
Recall $q, p$ are the generators of $H$ satisfying the Heisenberg relation $qp=pq+1$, see (\ref{eq H relation}).
Define
\begin{gather} \label{eq def alpha q}
g(q)= \sum \limits_{i\le0}t_{2i}t_{2i-1} - \sum \limits_{i>0}t_{2i-1}t_{2i},
\end{gather}
\begin{gather} \label{eq def alpha p}
g(p)= \sum \limits_{i\le0}t_{2i+1}t_{2i} - \sum \limits_{i>0}t_{2i+1}t_{2i}
\end{gather}
as two elements of a certain completion of $Cl$.
The goal of this section is to lift $g(q), g(p)$ to the categorical level.

The Heisenberg category $\ch$ is a subcategory of $D(B^e)$, and the Clifford category $\cl$ is a subcategory of $D(R^e)$.
The generators $p, q$ are lifted to the $B$-bimodules $P ,Q \in \ch$.
The generators $t_i$ are lifted to the $R$-bimodules $T(i) \in \cl$, for $i \in \Z$.
We fix the equivalence $D(B^e) \ra D(R_0^e)$ induced by (\ref{eq chain rel}).
The restriction onto its subcategory $\ch$ gives an embedding $$\cg: \ch \ra D(R_0^e)$$
of monoidal categories.
Categorification of $g(q), g(p)$ amounts to showing that the objects $\cg(Q), \cg(P)$ lift the expressions in (\ref{eq def alpha q}, \ref{eq def alpha p}).

Recall that $T(2i)=\bigoplus\limits_{k\in\Z}T(2i)_k, T(2i-1)=\bigoplus\limits_{k\in\Z}T(2i-1)_k$, where each summand $T(2i)_k$ is a $(R_{k-1}, R_k)$-bimodule, and $T(2i-1)_k$ is a $(R_{k+1}, R_k)$-bimodule.
Thus, each summand of $T(i)T(j)$ for $i-j$ odd is a $R_k$-bimodule.
Throughout this section we take the summand of $k=0$, and still write $T(i)T(j)$ for the summand of $R_0$-bimodule by abuse of notation.

\begin{defn} \label{def g(q)}
Define $\wtq \in D(R^e_0)$ by
$$\left(\bigoplus\limits_{i \in \Z}T(2i)T(2i-1) \oplus  \bigoplus\limits_{j>0}R_0(j)[-1], \quad d=\bigoplus\limits_{i\le j}d_{i,j} \right),$$
where $R_0(j)=R_0$ for all $j>0$, and $d_{i,j}=\frac{1}{(j-i)!}u_i: T(2i)T(2i-1) \ra R_0(j)$ for $i \le j, j>0$.
Here, $u_i$ is defined in figure \ref{fig v7}.
\end{defn}

A part of the complex $\wtq$ is given by
$$\xymatrix{
T(4)T(3) \ar[r]  & R_0(2) \\
T(2)T(1) \ar[r] \ar[ur]  & R_0(1) \\
T(0)T(-1)  \ar[ur] \ar[uur]  &
}$$

The object $\wtq$ has a finite truncation $\wtq(n)$, where the direct sum is over $0<j \le n, -n \le i \le n$.
Its class $[\wtq(n)]$ in $K_0(\cl)$ is equal to
$$\sum \limits_{-n \le i\le0}t_{2i}t_{2i-1} - \sum \limits_{0<i\le n}(1-t_{2i}t_{2i-1})=\sum \limits_{-n \le i\le0}t_{2i}t_{2i-1} - \sum \limits_{0<i\le n}t_{2i-1}t_{2i},$$
which approximates $g(q)$ as $n$ goes to infinity.
Thus, $\wtq$ can be viewed as a categorification of $g(q)$.
Our main result in this section is the following.

\begin{thm} \label{thm bf HCl}
The objects $\cg(Q)$ and $\wtq$ are isomorphic in $D(R_0^e)$.
\end{thm}

The proof will be a technical computation and occupy the rest of this section.
%It is equivalent to show that $\cg(Q)$ and $\wt{Q}$ are isomorphic as endofunctors of $\cf_0$, the homotopy category of $DGP(R_0)$.
The homotopy category $\cf_0$ of $DGP(R_0)$ is generated by the projectives $P(\mf{x})$ for $\mf{x} \in \cim_0$.
Under the bijection $\mf{x}: \Par \ra \cim_0$ in (\ref{eq iso vfk}), each projective is of the form $P(\mf{x}(\mu))$, for some partition $\mu$.
Morphisms between the projectives are compositions of generators of $\Hom_{R_0}(P(\mf{x}(\lm)),P(\mf{x}(\mu)))$, for $\lm \ra \mu$ in the quiver $\Gamma$.
Each morphism space has a distinguished generator induced by right multiplication with the element $r(\lm|\mu) \in R_0$.
It corresponds to $(\lm|\mu) \in F$ under the relations (\ref{eq chain rel}).
%This element is identified with $(\lm|\mu) \in F$ under Proposition \ref{prop HR0 F}.

For $M=\cg(Q)$ and $\wtq$, we will explicitly compute $M(\mu)=M \ot_{R_0} P(\mf{x}(\mu))$ and the map
$$M(\lm|\mu): M \ot_{R_0} P(\mf{x}(\lm)) \ra M \ot_{R_0} P(\mf{x}(\mu))$$
induced by $r(\lm|\mu)$.
In order to prove Theorem \ref{thm bf HCl}, it suffices to show that there exists isomorphisms $\beta_{\mu}: \cg(Q) \ot_{R_0} P(\mf{x}(\mu)) \cong \wtq \ot_{R_0} P(\mf{x}(\mu))$ in $\cf_0$ making the following diagram commutes:
\begin{align} \label{eq proof of thm BF}
\xymatrix{
\cg(Q)(\lm) \ar[rr]^{\cg(Q)(\lm|\mu)} \ar[d]^{\beta_{\lm}}& & \cg(Q)(\mu) \ar[d]^{\beta_{\mu}}  \\
\wtq(\lm) \ar[rr]^{\wtq(\lm|\mu)}                      &  & \wtq(\mu)
}
\end{align}

\vspace{.2cm}
\n{\bf The case of $\cg(Q)$.}
Under the relations in (\ref{eq chain rel}), the object $P(\mf{x}(\mu)) \in DGP(R_0)$ corresponds to $V_{\mu} \in \Kom(B)$, and the element $r(\lm|\mu) \in R_0$ corresponds to $(\lm|\mu) \in F$ which in turn corresponds to $b_{\lm\mu} \in B$.
We have
$$\cg(Q)(\mu)=\cg(Q) \ot_{R_0} P(\mf{x}(\mu))=\cg(Q \ot_B V_{\mu}), \qquad \cg(Q)(\lm|\mu)=\cg(Q\ot_B V_{\lm} \xra{\cdot b_{\lm\mu}} Q\ot_B V_{\mu}),$$
\begin{gather} \label{eq Q V}
\mbox{where} \quad Q \ot_B V_{\mu} \cong \bigoplus\limits_{\mu' \in \Res(\mu)}V_{\mu'}.
\end{gather}
In order to compute the map $Q(\lm|\mu): Q\ot_B V_{\lm} \xra{\cdot b_{\lm\mu}} Q\ot_B V_{\mu}$, we need to choose an explicit isomorphism for (\ref{eq Q V}).
It is determined by a choice of a generator of the one dimensional vector space
$e_{\mu'}  Q  e_{\mu} \cong (e_{\mu'} \bt 1(1)) \K[S(n)] e_{\mu}$.
The latter space has a generator $f_{\mu'\mu}$, see Definition \ref{def choice f}.
In the following, we choose another generator for $e_{\mu'}  Q  e_{\mu}$.

For $\mu' \in \Res(\mu)$, the coordinate of the extra box $\mu \backslash \mu'$ is $(s+1, t+1)$, i.e. there are $s$ and $t$ boxes at the left and the top of the extra box, respectively, see figure \ref{se2}.
In particular, the content $c(\mu \backslash \mu')=s-t$.
Let $\eta_1$ denote the sub-partition of $\mu$ living at the top right of the extra box, and $\eta_2$ denote the sub-partition of $\mu$ living at the bottom left of the extra box.
Suppose $\eta_1=(x_1, \dots x_t)$, and $\ov{\eta}_2=(\ov{y}_1, \dots, \ov{y}_s)$.
\begin{figure}[h]
\begin{overpic}
[scale=0.35]{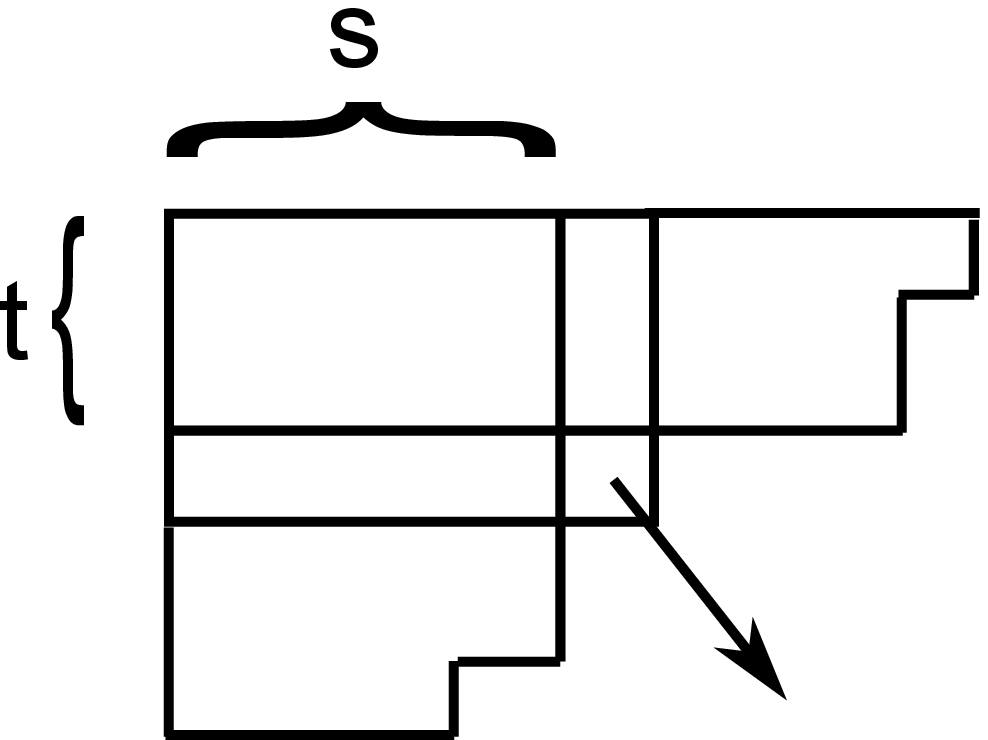}
\put(90,0){$\mu \backslash \mu'$}
\put(75,40){$\eta_1$}
\put(30,10){$\eta_2$}
\end{overpic}
\caption{}
\label{se2}
\end{figure}

Define
\begin{gather} \label{eq def h}
h_{\mu'\mu}=(-1)^s~\prod\limits_{1\le i \le t}(x_i+t+1-i)~\prod\limits_{1\le j \le s}(\ov{y}_j+s+2-j)^{-1} \in \mathbb{Q}\subset \K.
\end{gather}
Suppose $\lm' \ra \lm \oplus \mu' \ra \mu$ in the quiver $\Gamma$, see (\ref{eq def lmmunueta}).
A direct computation shows that
\begin{gather} \label{eq rel h}
h_{\mu'\mu}=\frac{d}{d-1}h_{\lm'\lm},
\end{gather}
for $d=c(\mu' \backslash \lm')-c(\lm \backslash \lm')$.

\begin{defn} \label{def qmu'mu}
For $\mu' \in \Res(\mu)$, define $q_{\mu'\mu} \in e_{\mu'}  Q  e_{\mu}$ which corresponds to $h_{\mu'\mu}^{-1}~ f_{\mu'\mu}$ under the isomorphism $e_{\mu'}  Q  e_{\mu} \cong (e_{\mu'} \bt 1(1)) \K[S(n)] e_{\mu}$.
\end{defn}

The collection of choices $\{q_{\mu'\mu}\}$ determines the isomorphisms (\ref{eq Q V}).
We fix them from now on.
The map
$$Q(\lm|\mu): Q\ot_B V_{\lm}\cong \bigoplus\limits_{\lm' \in \Res(\lm)}V_{\lm'} \ra Q\ot_B V_{\mu}\cong \bigoplus\limits_{\mu' \in \Res(\mu)}V_{\mu'},$$
is determined by its components $V_{\lm'} \ra V_{\mu'}$ for $\lm' \ra \mu'$.
Each component maps the generator $e_{\lm'} \in V_{\lm'}$ to $a_{\lm'\lm}^{\mu'\mu} ~ b_{\lm'\mu'} \in V_{\mu'}$, for some coefficients $a_{\lm'\lm}^{\mu'\mu} \in \K$.
The collection $\{a_{\lm'\lm}^{\mu'\mu}\}$ determines the map $Q(\lm|\mu)$.
The coefficients are computed via the following equation:
\begin{gather} \label{eq Q morphism}
q_{\lm'\lm} \cdot b_{\lm\mu} = \sum\limits_{\mu' \in \Res(\mu)} a_{\lm'\lm}^{\mu'\mu} ~ b_{\lm'\mu'} \cdot q_{\mu'\mu} \in e_{\lm'} Q e_{\mu}.
\end{gather}
Figure \ref{se3} describes the equation, where $f_{\lm\mu}$ corresponds to $b_{\lm\mu}$ under the canonical isomorphism
$(e_{\lm} \bt 1(1)) \K[S(n)] e_{\mu} \cong e_{\lm}Be_{\mu}$, see figure \ref{se1}.

Given $\lm, \mu$ and $\lm'$, the sum contains at most two terms.
If $e_{\lm'} Q e_{\mu}$ is one dimensional, then $\mu'=\lm$.
If $e_{\lm'} Q e_{\mu}$ is two dimensional, then $\mu'=\lm$ or $\nu$, where $\nu$ is the partition satisfying $\lm' \ra \lm \oplus \nu \ra \mu$.

\begin{figure}[h]
\begin{overpic}
[scale=0.3]{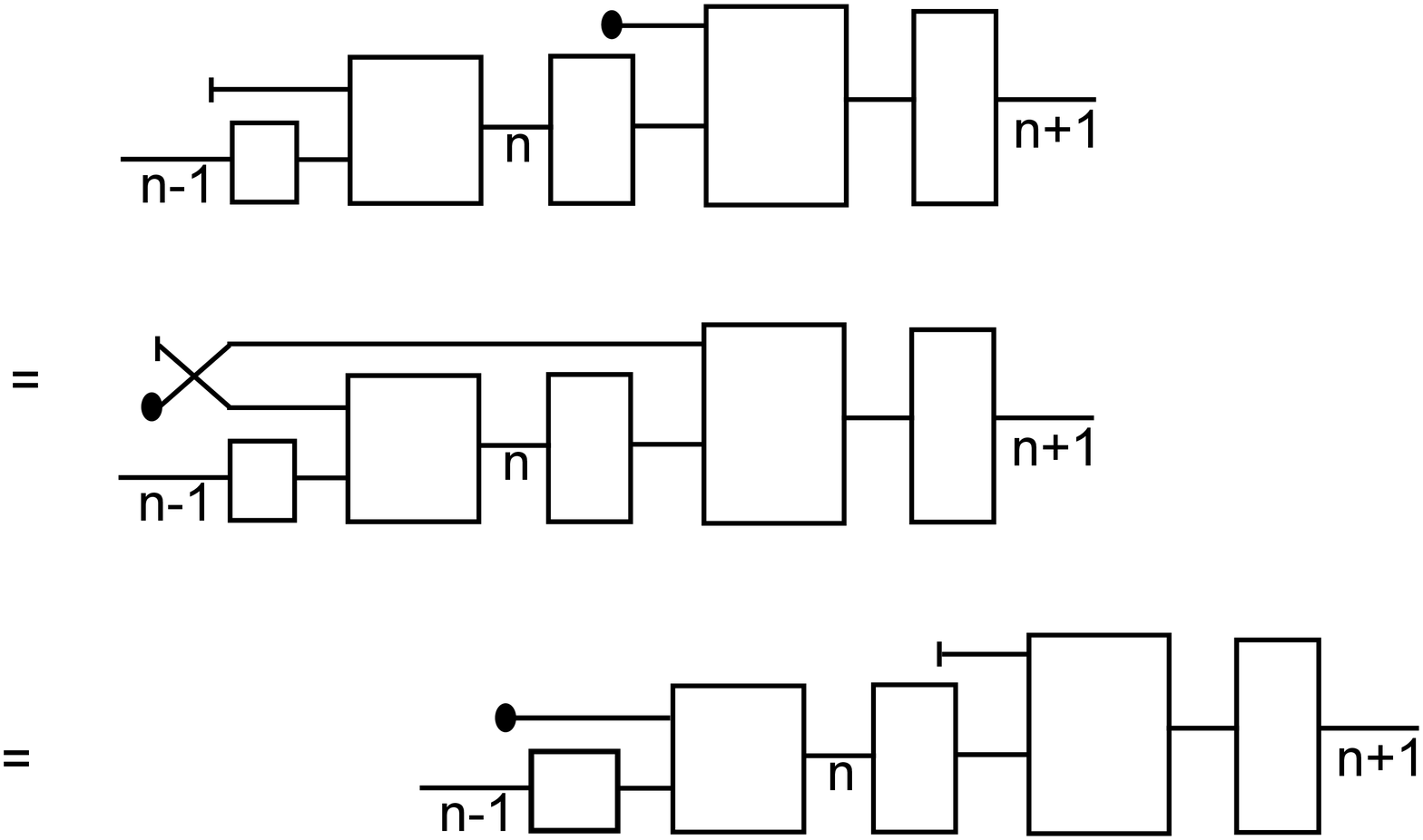}
\put(17,47){$\lm'$}
\put(26.5,50){$q_{\lm'\lm}$}
\put(41,50){$\lm$}
\put(52,52){$f_{\lm\mu}$}
\put(67,52){$\mu$}
\put(17,24){$\lm'$}
\put(26.5,27){$q_{\lm'\lm}$}
\put(41,27){$\lm$}
\put(52,29){$f_{\lm\mu}$}
\put(67,29){$\mu$}
\put(6,5){$\sum\limits_{\mu' \in \Res(\mu)} a_{\lm'\lm}^{\mu'\mu}$}
\put(38.5,1.5){$\lm'$}
\put(47.5,4){$f_{\lm'\mu'}$}
\put(62,4){$\mu'$}
\put(74,5){$q_{\mu'\mu}$}
\put(89,5){$\mu$}
\end{overpic}
\caption{}
\label{se3}
\end{figure}

\begin{lemma} \label{lem a=1}
Under the notation above, if $e_{\lm'} Q e_{\mu}$ is two dimensional, then $a_{\lm'\lm}^{\nu\mu}=1$.
\end{lemma}
\begin{proof}
The equation (\ref{eq sn f}) implies that
$$s_n \cdot f_{\lm'\lm} \cdot f_{\lm\mu}=\frac{1}{d}~f_{\lm'\lm} \cdot f_{\lm\mu}+\frac{d-1}{d}~f_{\lm'\nu} \cdot f_{\nu\mu},$$
where $d=c(\nu \backslash \lm')-c(\lm \backslash \lm')$.
Substituting $q_{\lm\mu}=h_{\lm\mu}^{-1}~f_{\lm\mu}$, the equation (\ref{eq Q morphism}) becomes
$$h_{\lm'\lm}^{-1}~s_n \cdot f_{\lm'\lm} \cdot f_{\lm\mu}=a_{\lm'\lm}^{\lm\mu}h_{\lm\mu}^{-1}~f_{\lm'\lm} \cdot f_{\lm\mu} + a_{\lm'\lm}^{\nu\mu}h_{\nu\mu}^{-1}~f_{\lm'\nu} \cdot f_{\nu\mu},$$
see figure \ref{se3}.
Comparing the coefficients, we have
\begin{gather} \label{eq rel a=h}
a_{\lm'\lm}^{\lm\mu}=\frac{1}{d} ~ \frac{h_{\lm\mu}}{h_{\lm'\lm}}, \qquad a_{\lm'\lm}^{\nu\mu}=\frac{d-1}{d} ~ \frac{h_{\nu\mu}}{h_{\lm'\lm}}.
\end{gather}
The lemma follows from that $\frac{h_{\nu\mu}}{h_{\lm'\lm}}=\frac{d}{d-1}$ by (\ref{eq rel h}).
\end{proof}

We compute $a_{\lm'\lm}^{\lm\mu}$ in the following.
Assume that the box $\lm \backslash \lm'$ is at the top right of the box $\mu \backslash \lm$.
Let $\eta_1, \eta_2, \eta_3$ be the sub-partitions of $\mu$ which are at the top right of $\lm \backslash \lm'$, between $\lm \backslash \lm'$ and $\mu \backslash \lm$, and at the bottom left of $\mu \backslash \lm$, respectively.
The coordinates of $\lm \backslash \lm'$ and $\mu \backslash \lm$ are $(s_1+s_2+2, t_1+1)$ and $(s_1+1, t_1+t_2+2)$, respectively, see figure \ref{se4}.
Let
\begin{gather*}
\eta_1=(x_{1}, \dots, x_{t_1}), \quad \ove_1=(\ovx_{1}, \dots, \ovx_{s_3}), \\
 \eta_2=(z_{1}, \dots, z_{t_2}), \quad \ove_2=(\ovz_{1}, \dots, \ovz_{s_2}), \\
 \eta_3=(y_{1}, \dots, y_{t_3}), \quad \ove_3=(\ovy_{1}, \dots, \ovy_{s_1}).
\end{gather*}

\begin{figure}[h]
\begin{overpic}
[scale=0.35]{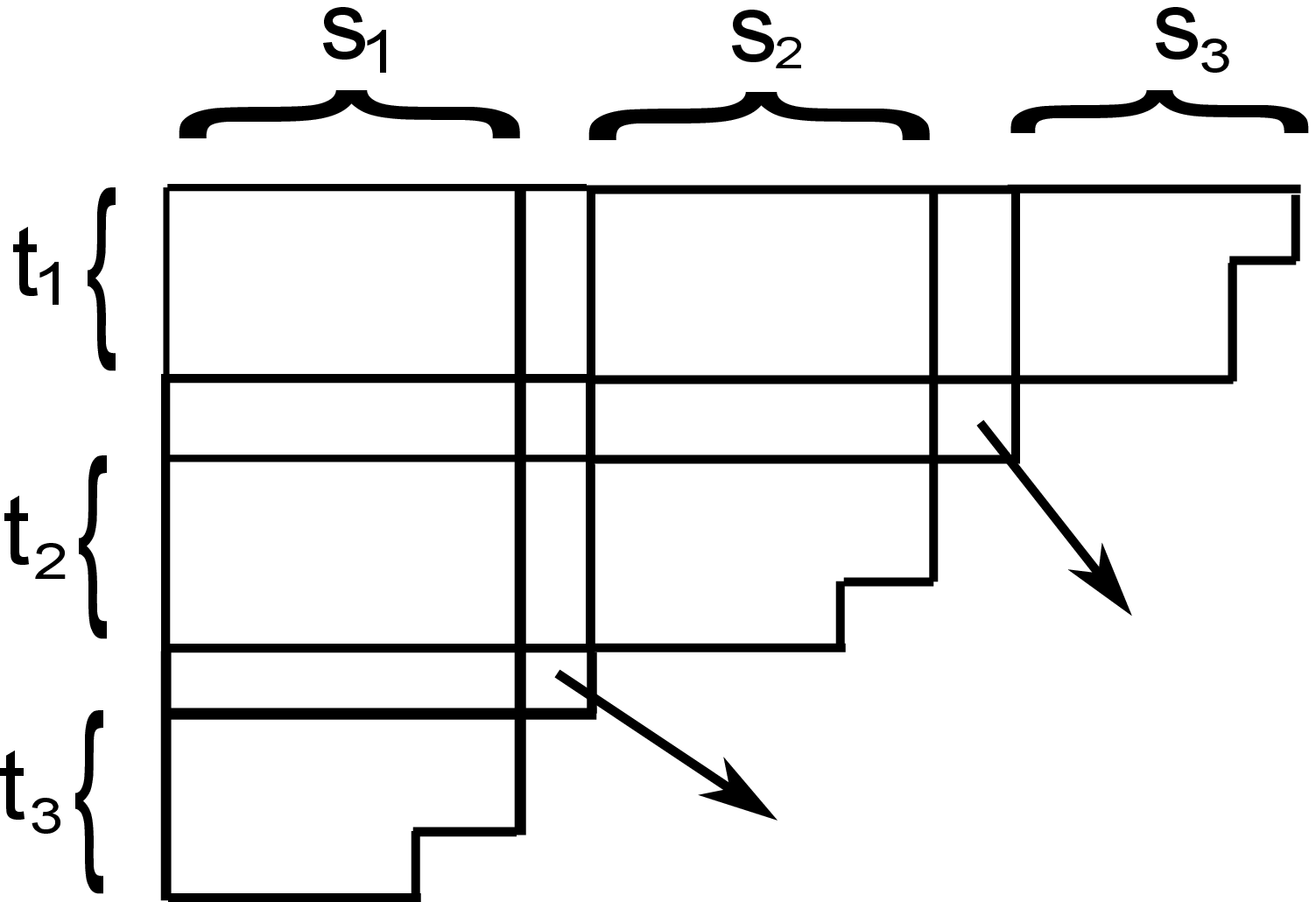}
\put(90,15){$\lm \backslash \lm'$}
\put(62,0){$\mu \backslash \lm$}
\put(85,45){$\eta_1$}
\put(54,25){$\eta_2$}
\put(20,5){$\eta_3$}
\end{overpic}
\caption{}
\label{se4}
\end{figure}

By the definition in (\ref{eq def h}), we have
$$h_{\lm'\lm}=\frac{(-1)^{s_1+s_2+1}}{(s_2+t_2+2)}\prod\limits_{1\le i \le t_1}(x_i+t_1-i+1) \prod\limits_{1\le j \le s_2}(\ovz_j+s_2-j+2)^{-1} \prod\limits_{1\le j \le s_1}(\ovy_j+s_1+s_2+t_2-j+4)^{-1}$$
$$h_{\lm\mu}=\frac{(-1)^{s_1}}{(s_2+t_2+2)^{-1}} \prod\limits_{1\le i \le t_2}(z_i+t_2-i+1) \prod\limits_{1\le i \le t_1}(x_i+s_2+t_2+t_1-i+3)\prod\limits_{1\le j \le s_1}(\ovy_j+s_1-j+2)^{-1}.$$
The constant $d=c(\nu \backslash \lm')-c(\lm \backslash \lm')=c(\mu \backslash \lm)-c(\lm \backslash \lm')=-(s_2+t_2+2)$.
It follows from (\ref{eq rel a=h}) that
\begin{align} \label{eq def a}
a_{\lm'\lm}^{\lm\mu}=(-1)^{s_2}(s_2+t_2+2)\prod\limits_{1\le j \le s_1}(\ovy_j+s_1+s_2+t_2-j+4)
\prod\limits_{1\le j \le s_1}(\ovy_j+s_1-j+2)^{-1}
\end{align}
\begin{align*}
\cdot  \prod\limits_{1\le i \le t_1}(x_i+t_1-i+1)^{-1}\prod\limits_{1\le i \le t_1}(x_i+s_2+t_2+t_1-i+3) \\
\cdot  \prod\limits_{1\le j \le s_2}(\ovz_j+s_2-j+2) \prod\limits_{1\le i \le t_2}(z_i+t_2-i+1).
\end{align*}

If $e_{\lm'} Q e_{\mu}$ is one dimensional, then $q_{\lm'\lm} \cdot b_{\lm\mu} = a_{\lm'\lm}^{\lm\mu} ~ b_{\lm'\lm} \cdot q_{\lm\mu}.$
A computation similar to that in Lemma \ref{lem a=1} shows that
$$a_{\lm'\lm}^{\lm\mu}=\pm \frac{h_{\lm\mu}}{h_{\lm'\lm}},$$
where $\pm$ depends on $\mu\backslash \lm'=(2)$ or $(1^2)$.

\vspace{.2cm}
\n{\bf The case of $\wtq$.}
Recall $\mf{x}(\lm)=(-2\ovlm_1+2, \dots, -2\ovlm_k+2k, 2(k+1), \dots)$, where $\ovlm=(\ovlm_1, \dots, \ovlm_k)$ is the dual partition of $\lm$.
Let $\cal{R}(\lm)=\{l ~|~ \ovlm_l > \ovlm_{l+1}\}$.
It is in bijection with $\Res(\lm)$, where each $l \in \cal{R}(\lm)$ corresponds to $\lm^l \in \Res(\lm)$ whose dual partition is $(\ovlm_1, \dots, \ovlm_l-1, \ovlm_{l+1}, \dots, \ovlm_k)$.

Proposition \ref{prop Cl relation2} and Lemma \ref{lem Cl action} imply that
\begin{align*}
T(2i)T(2i-1)\ot P(\mf{x}(\lm)) \cong 0&, \qquad \mbox{for}~~i<-\ovlm_1+1, \\
(T(2i)T(2i-1) \xra{u_i} R_0(i)) \ot P(\mf{x}(\lm)) \cong 0&, \qquad \mbox{for}~~i>k.
\end{align*}
Therefore, $\wtq \ot P(\mf{x}(\lm))$ is isomorphic to
\begin{align} \label{eq wtq ot P}
\left(\bigoplus\limits_{-\ovlm_1+1 \le i \le k}T(2i)T(2i-1)\ot P(\mf{x}(\lm)) \xra{f_{i,j}}  \bigoplus\limits_{1 \le j \le k}R_0(j) \ot P(\mf{x}(\lm)),\right)
\end{align}
where $f_{i,j}=\frac{1}{(j-i)!}u_i \ot id_{P(\mf{x}(\lm))}$.
The tensor product $T(2i)T(2i-1)\ot P(\mf{x}(\lm))$ for $-\ovlm_1+1 \le i \le k$ has three possibilities:
\be
\item $T(2i)T(2i-1)\ot P(\mf{x}(\lm)) \cong P(\mf{x}(\lm))$ if $i=-\ovlm_j+j$ for $1 \le j \le k$;
\item $T(2i)T(2i-1)\ot P(\mf{x}(\lm)) \cong P(\mf{x}(\lm^l))$ if $i=-\ovlm_l+l+1$ for $l \in \cal{R}(\lm)$;
\item otherwise, $T(2i)T(2i-1)\ot P(\mf{x}(\lm))=0$.
\ee
For the first two cases, there are canonical choices for the isomorphisms, see Proposition \ref{prop Cl relation} and Lemma \ref{lem Cl action}.
We fix them from now on.

In the first case, the map $f_{i,j}=\frac{1}{(j-i)!}\cdot id: P(\mf{x}(\lm)) \ra P(\mf{x}(\lm))$.
In the second case, the map $f_{i,j}=\frac{1}{(j-i)!}\cdot g^l: P(\mf{x}(\lm^l)) \ra P(\mf{x}(\lm))$, where $g^l$ is the right multiplication by the generator $r(\lm^l|\lm) \in R_0$.
Thus, $\wtq \ot P(\mf{x}(\lm)$ contains $k$ copies of $P(\mf{x}(\lm))$ at both degrees zero and one.
The differential between these $2k$ copies is the identity map multiplying with the $k \times k$ scalar matrix $C=(c_{i,j})$, where
\begin{align} \label{eq matrix C}
c_{i,j}= \left\{
\begin{array}{rl}
\frac{1}{(i-(-\ovlm_j+j))!} & \mbox{if}~i \ge -\ovlm_j+j, \\
0 & \mbox{otherwise.}
\end{array}
\right.
\end{align}

\begin{lemma} \label{lem nonsing C}
The matrix $C$ is nonsingular.
\end{lemma}

To prove the lemma, we need some preparations.
Define
\begin{gather} \label{eq def f(s,t)}
f(s,t)=\sum\limits_{0 \le j \le t-1} (-1)^{s+t-1-j} \frac{1}{(s+t-1-j)!~j!},
\end{gather}
for $t \ge 1$.
Here $0!=1$, and $\frac{1}{s!}$ is understood as zero if $s<0$.
By induction on $t$, one can show that
\begin{align} \label{eq f(s,t)}
f(s,t)= \left\{
\begin{array}{rl}
(-1)^{s} \frac{1}{(s-1)!(t-1)!(s+t-1)} & \mbox{if}~s>0, \\
\delta_{t+s,1} & \mbox{if}~s \le 0.
\end{array}
\right.
\end{align}

\begin{lemma} \label{lem det}
Let $A=(a_{ij})$ be a $k \times k$ matrix such that $a_{ij}=(a_i+b_j)^{-1}$.
The determinant $|A|=\prod\limits_{i,j}(a_i+b_j)^{-1} \prod\limits_{i<i'}(a_i-a_{i'}) \prod\limits_{j<j'}(b_j-b_{j'})$.
\end{lemma}
\begin{proof}
Multiplying the $i$-th row of $A$ by $\prod\limits_{j}(a_i+b_j)$, we obtain a matrix $\wt{A}=(\wt{a}_{ij})$, where $\wt{a}_{ij}=\prod\limits_{j'\neq j}(a_i+b_{j'})$ is a homogeneous polynomial of degree $k-1$.
If $a_i=a_{i'}$, then $|\wt{A}|=0$. So $(a_i-a_{i'})$ divides $|\wt{A}|$.
Similarly, $(b_j-b_{j'})$ divides $|\wt{A}|$.
Thus, $|\wt{A}|=c \prod\limits_{i<i'}(a_i-a_{i'}) \prod\limits_{j<j'}(b_j-b_{j'})$ for some polynomial $c$.
By comparing degrees, $c$ has degree zero, hence is a constant.
It is easy to compute that $c=1$ which implies the lemma.
\end{proof}

\n {\em Proof of Lemma \ref{lem nonsing C}:}
Consider a nonsingular $k \times k$ lower triangular matrix $B=(b_{ij})$, where
\begin{align} \label{eq def B}
b_{ij}= \left\{
\begin{array}{rl}
(-1)^{i-j}\frac{1}{(i-j)!} & \mbox{if}~i \ge j, \\
0 & \mbox{otherwise.}
\end{array}
\right.
\end{align}
The product $BC=A=(a_{ij})$ can be computed via (\ref{eq f(s,t)}): $a_{ij}=(-1)^{\ovlm_j-j+i}f(\ovlm_j-j+1,i)$ for $1 \le i,j \le k$.
The sequence $\{\ovlm_j-j\}$ is strictly decreasing with respect to $j$.
Let $p \in [1,k]$ such that $\ovlm_j-j \ge 0$ for $j \in [1,p]$, and $\ovlm_j-j < 0$ for $j \in [p+1,k]$.
We have
\begin{align} \label{eq def mat a}
a_{ij}= \left\{
\begin{array}{rl}
(-1)^{i+1} \frac{1}{(\ovlm_j-j)!(i-1)!(\ovlm_j-j+i)} & \mbox{if}~ j \in [1,p], \\
\delta_{i,j-\ovlm_j} & \mbox{if}~ j \in [p+1,k].
\end{array}
\right.
\end{align}
For any $j \in [p+1, k]$, the $j$-th column is of the form $(0, \dots, 0,1,0, \dots, 0)^T$, where $1$ is at the $(j-\ovlm_j)$-th row.
After removing these columns and rows, the resulting $p \times p$ matrix denoted by $\wt{A}$, is in the form of a matrix in Lemma \ref{lem det} up to a constant.
Let $\cal{I}=[1,k] \backslash \{j-\ovlm_j ~|~ j \in [p+1,k]\}$.
We have $|A|=(-1)^{u}|\wt{A}|$, where $u=\sum\limits_{j \in [p+1,k]}(j+j-\ovlm_j)=\sum\limits_{j \in [1,p]}\ovlm_j$.
It follows from Lemma \ref{lem det} that
\begin{gather} \label{eq det a}
|A|=(-1)^v \prod\limits_{j}\frac{1}{(\ovlm_j-j)!}\prod\limits_{i}\frac{1}{(i-1)!}\prod\limits_{i,j}\frac{1}{\ovlm_j-j+i}
\prod\limits_{j>j'}\frac{1}{\ovlm_j-j-\ovlm_{j'}+j'}\prod\limits_{i>i'}\frac{1}{i-i'},
\end{gather}
where the products are taken over $j,j' \in [1,p], i,i' \in \cal{I}$, and $v=\sum\limits_{i\in \cal{I}}(i+1)+\sum\limits_{j \in [p+1,k]}\ovlm_j$.
Then $|C|=|A|$ is nonzero so that $C$ is nonsingular.
\qed

\begin{prop} \label{prop wtQ obj}
There is an isomorphism $\wtq \ot P(\mf{x}(\lm)) \cong \bigoplus\limits_{\lm' \in \Res(\lm)}P(\mf{x}(\lm'))$. \end{prop}
\begin{proof}
It follows from Lemma \ref{lem nonsing C} that $\wtq \ot P(\mf{x}(\lm)$ is isomorphic to its quotient complex obtained by removing the $2k$ copies of $P(\mf{x}(\lm))$. The quotient complex is concentrated at degree zero, and isomorphic to $\bigoplus\limits_{l \in \cal{R}(\lm)}P(\mf{x}(\lm^l))$.
\end{proof}

The projection
$$\op{pr}: \wtq \ot P(\mf{x}(\lm)) \ra \bigoplus\limits_{\lm' \in \Res(\lm)}P(\mf{x}(\lm'))$$ is a chain map.
In the inverse direction, there is a unique chain map
\begin{gather} \label{eq eq g}
g: \bigoplus\limits_{\lm' \in \Res(\lm)}P(\mf{x}(\lm')) \ra \wtq \ot P(\mf{x}(\lm)),
\end{gather}
such that $\op{pr} \circ g=id$, and $g \circ \op{pr}$ is homotopic to $id$.
We fix the isomorphisms from now on.
The map $$\wtq(\lm|\mu):  \wtq \ot P(\mf{x}(\lm)) \cong \bigoplus\limits_{\lm' \in \Res(\lm)}P(\mf{x}(\lm')) \ra  \wtq \ot P(\mf{x}(\mu)) \cong \bigoplus\limits_{\mu' \in \Res(\mu)}P(\mf{x}(\mu'))$$
is determined by its components $P(\mf{x}(\lm')) \ra  P(\mf{x}(\mu'))$ for $\lm' \ra \mu'$.
Each component maps the generator $\mb_{\mf{x}(\lm')}$ to $\wta_{\lm'\lm}^{\mu'\mu} \cdot \mb_{\mf{x}(\mu')}$.
The collection $\{\wta_{\lm'\lm}^{\mu'\mu}\}$ of scalars determines the map $\wtq(\lm|\mu)$.

To prove Theorem \ref{thm bf HCl}, define $\beta_{\mu}$ as a composition
$$\cg(Q) \ot P(\mf{x}(\mu))=\cg(Q \ot V_{\mu}) \cong \cg(\bigoplus\limits_{\mu' \in \Res(\mu)}V_{\mu'})=\bigoplus\limits_{\mu' \in \Res(\mu)} P(\mf{x}(\mu')) \cong \wtq \ot P(\mf{x}(\mu)),$$
of the chosen isomorphisms, where the first isomorphism is determined by $\{q_{\mu'\mu}\}$ in Definition \ref{def qmu'mu}, and the second isomorphism is $g$ in (\ref{eq eq g}).
It is enough to show that $\wta_{\lm'\lm}^{\mu'\mu}=a_{\lm'\lm}^{\mu'\mu}$, by (\ref{eq proof of thm BF}).

\begin{lemma} \label{lem wta=1}
Under the notation before Lemma \ref{lem a=1}, if $e_{\lm'} Q e_{\mu}$ is two dimensional, then $\wta_{\lm'\lm}^{\nu\mu}=1$.
\end{lemma}
\begin{proof}
The map is the composition
$$P(\mf{x}(\lm')) \ra \wtq \ot P(\mf{x}(\lm)) \ra \wtq \ot P(\mf{x}(\mu)) \ra P(\mf{x}(\nu)).$$
There exist $i$ and $j$ such that $$T(2i)T(2i-1)\ot P(\mf{x}(\lm))\cong P(\mf{x}(\lm')), \qquad T(2j)T(2j-1)\ot P(\mf{x}(\mu))\cong P(\mf{x}(\nu)).$$
If $\nu \neq \lm$, then $i=j$.
The map $T(2i)T(2i-1)\ot P(\mf{x}(\lm)) \ra T(2i)T(2i-1)\ot P(\mf{x}(\mu))$ induced by $r(\lm|\mu)$ is equal to the map $P(\mf{x}(\lm')) \ra P(\mf{x}(\nu))$ induced by $r(\lm'|\nu)$.
Hence, $\wta_{\lm'\lm}^{\nu\mu}=1$.
\end{proof}

It remains computing $\wta_{\lm'\lm}^{\lm\mu}$.
Assume that $e_{\lm'} Q e_{\mu}$ is two dimensional, and $\lm', \lm, \mu$ are in figure \ref{se4}.
Let $j_0=s_1+1, j_1=s_1+s_2+2$.
We have $\ovmu_{j_0}=\ovlm_{j_0}+1$, and $\ov{\lm'}_{j_1}=\ovlm_{j_1}-1$.
The relevant components are
$$T(-2\ovlm_{j_0}+2j_0)T(-2\ovlm_{j_0}+2j_0-1) \ot P(\mf{x}(\mu)) \cong P(\mf{x}(\lm)),$$
$$T(-2\ovlm_{j_0}+2j_0)T(-2\ovlm_{j_0}+2j_0-1) \ot P(\mf{x}(\lm)) \cong P(\mf{x}(\lm)),$$
$$T(-2\ovlm_{j_1}+2j_1+2)T(-2\ovlm_{j_1}+2j_1+1) \ot P(\mf{x}(\lm)) \cong P(\mf{x}(\lm')).$$
The composition $P(\mf{x}(\lm')) \ra \wtq \ot P(\mf{x}(\lm)) \ra \wtq \ot P(\mf{x}(\mu)) \ra P(\mf{x}(\lm))$ is equal to
\begin{align*}
P(\mf{x}(\lm')) &\ra T(-2\ovlm_{j_0}+2j_0)T(-2\ovlm_{j_0}+2j_0-1) \ot P(\mf{x}(\lm)) \\
&\ra T(-2\ovlm_{j_0}+2j_0)T(-2\ovlm_{j_0}+2j_0-1) \ot P(\mf{x}(\mu)) \\
&\ra P(\mf{x}(\lm)),
\end{align*}
where the second and third maps are the chosen isomorphisms, and the first map is $g_{j_0} \cdot r(\lm'|\lm)$ for the distinguished generator $r(\lm'|\lm) \in \Hom(P(\mf{x}(\lm')), P(\mf{x}(\lm)))$.
The first map is a component of the unique chain map $g$ in (\ref{eq eq g}).
More precisely, $g|_{P(\mf{x}(\lm'))}=id \oplus \bigoplus\limits_{1\le j \le k}g_{j} \cdot r(\lm'|\lm)$,
where
\begin{gather*}
id: P(\mf{x}(\lm')) \ra T(-2\ovlm_{j_1}+2j_1+2)T(-2\ovlm_{j_1}+2j_1+1) \ot P(\mf{x}(\lm)) \cong P(\mf{x}(\lm')), \\
g_{j} \cdot r(\lm'|\lm): P(\mf{x}(\lm')) \ra T(-2\ovlm_{j}+2j)T(-2\ovlm_{j}+2j-1) \ot P(\mf{x}(\lm)) \cong P(\mf{x}(\lm)).
\end{gather*}
The condition $g$ being a chain map requires that
\begin{gather} \label{eq eq C}
C \cdot \mf{g}=C \cdot (g_1, \dots, g_k)^T = (f_1, \dots, f_k)^T=\mf{f},
\end{gather}
where $f_i=-((i-(-\ovlm_{j_1}+j_1+1))!)^{-1}$, and the matrix $C$ is defined in (\ref{eq matrix C}).
Note that $\frac{1}{s!}$ is set to be zero if $s<0$.
Since $C$ is nonsingular, there is a unique solution.
The coefficient $\wta_{\lm'\lm}^{\lm\mu}=g_{j_0}$.

Recall from the proof of Lemma \ref{lem nonsing C} that there exists $p \in [1,k]$ such that $\ovlm_j-j \ge 0$ for $j \in [1,p]$, and $\ovlm_j-j < 0$ for $j \in [p+1,k]$.
We further assume that $j_0 \in [1,p]$ and $j_1 \in [p+1, k]$.

\begin{lemma} \label{lem wta lmmu}
Under the assumption above, $\wta_{\lm'\lm}^{\lm\mu}=a_{\lm'\lm}^{\lm\mu}$ which is given in (\ref{eq def a}).
\end{lemma}
\begin{proof}
It is equivalent to solving $A \cdot \mf{g}=BC\cdot \mf{g}=B \cdot \mf{f}$, where $B$ is defined in (\ref{eq def B}), and $A$ is given in (\ref{eq def mat a}).
The assumption $j_1 \in [p+1, k]$ implies that $\ovlm_{j_1}-{j_1} < 0$.
So $B \cdot \mf{f}=(0,\dots,0,-1,0,\dots,0)^T$, where $-1$ is in the $(j_1-\ovlm_{j_1}+1)$-th row.
Let $i_0=j_1-\ovlm_{j_1}+1$.
Then $-g_{j_0}$ is equal to the $(j_0, i_0)$-th entry of the matrix $A^{-1}$, i.e.
$$g_{j_0}=-\frac{A_{i_0j_0}}{|A|},$$
where $A_{i_0j_0}$ is the algebraic cofactor.
The determinant $A_{i_0j_0}$ has a similar expression as $|A|$ in (\ref{eq det a}) but with different sets of indices.
More precisely,
\begin{gather*}
A_{i_0j_0}=(-1)^{v'} \prod\limits_{j}\frac{1}{(\ovlm_j-j)!}\prod\limits_{i}\frac{1}{(i-1)!}\prod\limits_{i,j}\frac{1}{\ovlm_j-j+i}
\prod\limits_{j>j'}\frac{1}{\ovlm_j-j-\ovlm_{j'}+j'}\prod\limits_{i>i'}\frac{1}{i-i'},
\end{gather*}
where the products are taken over $j,j' \in [1,p]\backslash \{j_0\}, i,i' \in \cal{I} \backslash \{i_0\}$, and $$v'=i_0+j_0+\sum\limits_{i \neq i_0}(i+1)+\sum\limits_{j \in [p+1,j_1]}(\ovlm_j-1)+\sum\limits_{j \in [j_1+1,k]}\ovlm_j=\sum\limits_{j \in [p+1,k]}\ovlm_j+\sum\limits_{i \in \cal{I}}(i+1)+j_0+j_1-p+1.$$
We have
\begin{gather*}
g_{j_0}=(-1)^{v-v'+1} \frac{(\ovlm_{j_0}-{j_0})!~(i_0-1)!~(\ovlm_{j_0}-j_0+i_0)\prod\limits_{j\neq j_0}(\ovlm_j-j+i_0)\prod\limits_{i\neq i_0}(\ovlm_{j_0}-j_0+i)}
{\prod\limits_{j>j_0}(\ovlm_j-j-\ovlm_{j_0}+j_0)\prod\limits_{j_0>j'}(\ovlm_{j_0}-j_0-\ovlm_{j'}+j')\prod\limits_{i>i_0}(i-i_0)
\prod\limits_{i_0>i'}(i_0-i')}
\end{gather*}
The $p-1$ factors $(\ovlm_j-j-\ovlm_{j_0}+j_0), (\ovlm_{j_0}-j_0-\ovlm_{j'}+j')$ of the denominator are negative.
So $g_{j_0}=(-1)^{v-v'+1+p-1} |g_{j_0}|$, and $v-v'+1+p-1=j_0+j_1-p+1+1+p-1=j_0+j_1+1$.
We rewrite $g_{j_0}$ as
\begin{align*}
g_{j_0}=& (-1)^{j_0+j_1+1}(\ovlm_{j_0}-j_0+i_0)\prod\limits_{j<j_0}(\ovlm_j-j+i_0)\prod\limits_{j<j_0}(\ovlm_{j}-j-\ovlm_{j_0}+j_0)^{-1}\\
&\cdot \prod\limits_{i>i_0}(i-i_0)^{-1} ~\prod\limits_{i > i_0} (\ovlm_{j_0}-j_0+i) \\
&\cdot (i_0-1)!~\prod\limits_{i<i_0}(i_0-i)^{-1} \prod\limits_{j > j_0}(\ovlm_{j}-j+i_0) \\
&\cdot (\ovlm_{j_0}-{j_0})!~ \prod\limits_{j>j_0}(\ovlm_{j_0}-j_0-\ovlm_{j}+j)^{-1} \prod\limits_{i< i_0}(\ovlm_{j_0}-j_0+i)
\end{align*}
where the products are taken over $j \in [1,p]$, and $i \in \cal{I}$.

\vspace{.2cm}
\n{\bf Claim:} The first two lines are equal to those in (\ref{eq def a}), and the last two lines are equal to the two factors of the last line in (\ref{eq def a}).

\vspace{.2cm}
We list the parameters as
\begin{align*}
j_0=s_1+1, \qquad j_1=s_1+s_2+2, \qquad k= s_1+s_2+s_3+2  \\
\ovlm_{j_0}=t_1+t_2+1, \qquad \ovlm_{j_1}=t_1+1, \\
i_0=j_1-\ovlm_{j_1}+1= s_1+s_2-t_1+2. \\
\ovlm_j=\left\{
\begin{array}{rl}
\ovlm_{j_0}+\ovy_j+1 &~\mbox{if}~1\le j<j_0, \\
\ovlm_{j_0}-(t_2-\ovz_{j-j_0}) &~\mbox{if}~j_0<j<j_1, \\
\ovlm_{j_1}-(t_1-\ovx_{j-j_1}+1) &~\mbox{if}~j_1<j \le k.
\end{array}\right.
\end{align*}
The claim about the first line follows from
\begin{gather*}
(-1)^{j_0+j_1+1}=(-1)^{s_2}, \quad \ovlm_{j_0}-j_0+i_0=s_2+t_2+2, \\
\ovlm_{j}-j-\ovlm_{j_0}+j_0=\ovy_j+1-j+j_0=\ovy_j+s_1+2-j, ~~\mbox{for}~~ j<j_0, \\
\ovlm_{j}-j+i_0=(\ovlm_{j}-j-\ovlm_{j_0}+j_0) +(\ovlm_{j_0}-j_0+i_0)=\ovy_j+s_1+s_2+t_2+4-j, ~~\mbox{for}~~ j<j_0.
\end{gather*}
For the second line, since $i \in \cal{I}=[1,k] \backslash \{j-\ovlm_j~|~j\in [p+1,k]\}$, we have
\begin{gather} \label{eq i-i0}
\prod\limits_{i>i_0}(i-i_0) \prod\limits_{j \in [j_1+1,k]}(j-\ovlm_{j}-i_0) = (k-i_0)!=(s_3+t_1)!,
\end{gather}
$$\mbox{where}~~\prod\limits_{j \in [j_1+1,k]}(\ovlm_{j}-j-i_0)=\prod\limits_{j \in [1,s_3]}(j+t_1-\ovx_j).$$

\vspace{.2cm}
\n{\bf Subclaim:} For any partition $\mu=(\mu_1, \dots, \mu_t)$, its dual partition $\ovmu=(\ovmu_1, \dots, \ovmu_s)$,
the following identity holds:
$$\prod\limits_{j\in [1,s]}(j+t-\ovmu_j) \prod\limits_{i\in [1,t]}(x_i+t+1-i) = (s+t)!.$$
This subclaim is easy to prove by induction on $s$ or $t$.

Applying the subclaim to the partition $\eta_1$ at the top right of the box $\lm \backslash \lm'$, see figure \ref{se4}, we have
$$\prod\limits_{j \in [1,s_3]}(j+t_1-\ovx_j) \prod\limits_{1\le i \le t_1}(x_i+t_1-i+1)=(s_3+t_1)!$$
Comparing with (\ref{eq i-i0}), we get
$$\prod\limits_{i>i_0}(i-i_0)= \prod\limits_{1\le i \le t_1}(x_i+t_1-i+1).$$
It implies that
$$\prod\limits_{i > i_0} (\ovlm_{j_0}-j_0+i)=\prod\limits_{i > i_0} (\ovlm_{j_0}-j_0+i_0+i-i_0)=\prod\limits_{1\le i \le t_1}(s_2+t_2+2+x_i+t_1-i+1).$$
Hence, the second line of $g_{j_0}$ agrees with that of (\ref{eq def a}).

The proof for the rest of the claim is similar, and we leave it for the reader.
It follows from the claim that $\wta_{\lm'\lm}^{\lm\mu}=g_{j_0}=a_{\lm'\lm}^{\lm\mu}$.
\end{proof}

Using a similar computation as in the proof of Lemma \ref{lem wta lmmu}, one can show that
$\wta_{\lm'\lm}^{\lm\mu}=a_{\lm'\lm}^{\lm\mu}$ holds for any case in general.
We complete the proof of Theorem \ref{thm bf HCl}.

\begin{example} \label{example}
For $\lm=(2)$ and $\mu=(2,1)$, $\mf{x}(\lm)=(0,2,6,\dots)$, and $\mf{x}(\mu)=(-2,2,6,\dots)$.
The nontrivial part of $\wtq \ot P(\mf{x}(\lm))$ is given by:
$$\xymatrix{
T(4)T(3)\ot P(0,2,6,\dots)&=P(0,4,6,\dots) \ar[r]^{1}  & P(0,2,6,\dots) \\
T(2)T(1)\ot P(0,2,6,\dots)&=P(0,2,6,\dots) \ar[r]^(.5){1} \ar[ur]^{1}  & P(0,2,6,\dots) \\
T(0)T(-1)\ot P(0,2,6,\dots)&=P(0,2,6,\dots)  \ar[ur]^{1} \ar[uur]^(.2){\frac{1}{2}}  &
}$$
where the differential is some multiples of the distinguished generator or the identity map, and the multiples are drawn above the arrows.
We have $\Res(\lm)=\{\lm'\}$, where $\lm'=(1)$ and $\mf{x}(\lm')=(0,4,6,\dots)$.
The isomorphism $P(0,4,6,\dots) \ra \wtq \ot P(0,2,6,\dots)$ is given by
$$P(0,4,6,\dots) \xra{(2,-2,1)} P(0,2,6,\dots) \oplus P(0,2,6,\dots) \oplus P(0,4,6,\dots),$$
where the codomain is the degree zero part of $\wtq \ot P(0,2,6,\dots)$.

Similarly, the nontrivial part of $\wtq \ot P(-2,2,6,\dots)$ is given by:
$$\xymatrix{
T(4)T(3)\ot P(-2,2,6,\dots)&=P(-2,4,6,\dots) \ar[r]^{1}  & P(-2,2,6,\dots) \\
T(2)T(1)\ot P(-2,2,6,\dots)&=P(-2,2,6,\dots) \ar[r]^(.5){1} \ar[ur]^{1}  & P(-2,2,6,\dots) \\
T(0)T(-1)\ot P(-2,2,6,\dots)&=P(0,2,6,\dots)  \ar[ur]^(.4){1} \ar[uur]^(.2){\frac{1}{2}}  & \\
T(-2)T(-3)\ot P(-2,2,6,\dots)&=P(-2,2,6,\dots)  \ar[uur]_{\frac{1}{2}} \ar[uuur]^(.15){\frac{1}{6}}  &
}$$
We have $\Res(\mu)=\{\nu, \lm\}$, where $\nu=(1^2)$ and $\mf{x}(\nu)=(-2,4,6,\dots)$.
The map of projection $\wtq \ot P(-2,2,6,\dots) \ra P(0,2,6,\dots) \oplus P(-2,4,6,\dots) $ gives the isomorphism.
The composition $P(0,4,6,\dots) \ra \wtq \ot P(\mf{x}(\lm)) \ra \wtq \ot P(\mf{x}(\mu)) \ra P(0,2,6,\dots) \oplus P(-2,4,6,\dots)$
is $(2\cdot r_1, 1 \cdot r_2)$, where $r_1 \in \Hom(P(0,4,6,\dots), P(0,2,6,\dots)), ~ r_2 \in \Hom(P(0,4,6,\dots), P(-2,4,6,\dots))$ are the distinguished generators.
In particular, $\wta_{\lm'\lm}^{\nu\mu}=1$, and $\wta_{\lm'\lm}^{\lm\mu}=2$.
On the other hand, $a_{\lm'\lm}^{\nu\mu}=1$ by Lemma \ref{lem a=1}, and $a_{\lm'\lm}^{\lm\mu}=\frac{1}{d}\frac{h_{\lm\mu}}{h_{\lm'\lm}}=2$, where $d=-2, h_{\lm\mu}=2, h_{\lm'\lm}=-\frac{1}{2}$. \end{example}

Since $P$ is left adjoint to $Q$ in $D(B^e)$, and $T(i)$ is left adjoint to $T(i-1)$ in $D(R_0^e)$, taking left adjoints on both sides of $\cg(Q) \cong \wtq$ gives an isomorphism $\cg(P) \cong \ov{P}$, where
$$\ov{P}=\left(\bigoplus\limits_{j>0}R_0(j)[1] \oplus \bigoplus\limits_{i \in \Z}T(2i)T(2i+1) , \quad d=\bigoplus\limits_{i \le j}d_{j,i} \right),$$
where $R_0(j)=R_0$ for all $j>0$, $d_{j,i}=\frac{1}{(j-i)!}\wt{u}_i: R_0(j) \ra T(2i)T(2i+1)$ for $i \le j, j>0$, and $\wt{u}_i$ is in Proposition \ref{prop Cl relation}.

Similarly, taking right adjoints of $\cg(Q) \cong \wtq$ gives another isomorphism $\cg(P_{-1}) \cong \ov{P_{-1}}$, where
$$\ov{P_{-1}}=\left(\bigoplus\limits_{j\ge 0}R_0(j)[1] \oplus \bigoplus\limits_{i \in \Z}T(2i)T(2i+1) , \quad d=\bigoplus\limits_{i \le j}d_{j,i} \right),$$
where $R_0(j)=R_0$ for all $j \ge 0$, $d_{j,i}=\frac{1}{(j-i)!}\wt{u}_i: R_0(j) \ra T(2i)T(2i+1)$ for $i \le j, j\ge0$.
In particular, $\ov{P}$ is a subcomplex of $\ov{P_{-1}}$, and there is an exact triangle
$$\ov{P} \ra \ov{P_{-1}} \ra R_0[1] \xra{[1]} \ov{P},$$
in $D(R_0^e)$.
This corresponds to the triangle $B \ra P \ra P_{-1} \xra{[1]} B$ on the Heisenberg side as in Theorem \ref{thm adj}.

\section{Lifting vertex operators}
%There are two vertex operators
%$$\psi(z)=\sum\limits_{i \in \Z} \psi_i z^i, \qquad \psi^*(z)=\sum\limits_{i \in \Z} \psi_i^* z^{-i},$$
%associated to the classical Clifford algebra $Cl'$.
Consider two generating series associated to $Cl$:
$$\overline{t}(z)=\sum\limits_{i\in \Z} t_{2i+1} z^{i}, \qquad t(z)=\sum\limits_{i\in \Z} t_{2i} z^{-i}.$$

%So, $\overline{t}(z)$ and $t(z)$ are mapped to $\psi(z)(1+z^{-1})$ and $\psi^*(z)$, respectively, under the map in (\ref{eq clcl'}).
Let $\tau: (V_F)_i \ra (V_F)_{i+1}$ denote the isomorphism of translation, i.e. $\tau(\mf{x})=(x_1+2, x_2+2, \dots)$ for $\mf{x}=(x_1, x_2, \dots) \in (V_F)_i$.
For $n \geq 0$, define a family of infinite partial sums
\begin{gather} \label{eq sn}
\overline{S}_n=\tau^{-1} \circ \sum\limits_{i=-\infty}^{n}(-1)^i t_{2i+1}, \qquad
S_n=\tau \circ \sum\limits_{i=-n}^{+\infty}(-1)^i t_{2i}.
\end{gather}
There are only finitely many nonzero terms in sums $\overline{S}_n(\mf{x}) \in (V_F)_{i+1}$ and $S_n(\mf{x}) \in (V_F)_{i-1}$, for any $\mf{x} \in (V_F)_i$.
Therefore, $\overline{S}_n$ and $S_n$ are linear transformations of $(V_F)_i$.
We focus on the case of $(V_F)_{0}$.
The bijection $\mf{x}$ in (\ref{eq iso vfk}) can be rewritten as
\begin{gather} \label{eq iso vf0}
\mf{x}(\overline{\lm})=(2-2\lm_1, 4-2\lm_2, \dots, 2k-2\lm_k, 2(k+1), \dots),
\end{gather}
where $\overline{\lm}$ is the dual partition of $\lm=(\lm_1, \lm_2, \dots, \lm_k)$.
Let $\Par_n$ denote the subset of $\Par$ consisting of partitions of at most $n$ rows, and $\ov{\Par_n}$ denote the subset of $\Par$ consisting of partitions of at most $n$ columns.

\begin{defn} \label{def lm+n}
Define $\lm \cup 1^n$ as the partition obtained from $\lm$ by adding $1$ on each row of the first $n$ rows.
Then $\lm \cup 1^n \in \Par_n$ if $\lm \in \Par_n$.
\end{defn}

Let $\ov{\Gamma}_n$ and $\Gamma_n$ denote the full subquivers of $\Gamma$ whose sets of vertices are $\ov{\Par_n}$ and $\Par_n$, respectively.
For instance, a part of $\ov{\Gamma}_2$ is:
$$
\xymatrix{
            &                    &                      &                     & (2^2)    &                 \\
            &                    &    (2) \ar[r]        & (2,1) \ar[ur]\ar[r] & (2,1^2)  & \cdots           \\
(0)\ar[r]   & (1) \ar[ur]\ar[r]  &  (1^2) \ar[ur]\ar[r] & (1^3) \ar[ur]\ar[r] & (1^4)    &
}$$
and a part of ${\Gamma}_2$ is:
$$
\xymatrix{
(0)\ar[r]   & (1) \ar[dr]\ar[r]  &  (2) \ar[dr]\ar[r] & (3) \ar[dr]\ar[r] & (4)    &                       \\
            &                    &    (1^2) \ar[r]        & (2,1) \ar[dr]\ar[r] & (3,1)  & \cdots           \\
            &                    &                      &                     & (2^2)    &
            }$$

Define two families of subalgebras of $F$:
\begin{gather} \label{eq Fn filtration}
\br_n= \bigoplus \limits_{\lm, \mu \in \Par_n} (\overline{\lm}) F (\overline{\mu}),
\qquad F_n= \bigoplus \limits_{\lm, \mu \in \Par_n} (\lm) F (\mu).
\end{gather}
In other words, $\br_n$ and $F_n$ are the subalgebras of $F$ generated by $\ov{\Gamma}_n$ and $\Gamma_n$, respectively.
Any of them is an infinite dimensional $\K$-vector space except that $F_0=\overline{F}_0$ which is one dimensional generated by $(0)$.
There are two chains of inclusions:
$$ \overline{F}_0 \subset \overline{F}_1 \subset \dots \subset \overline{F}_n \subset \cdots \subset F,
\qquad F_0 \subset F_1 \subset \dots \subset F_n \subset \cdots \subset F.$$

\subsection{The case of $\br_n$}
The algebra $\br_n$ has a $\K$-basis $\{(\overline{\lm}||\overline{\mu}); \lm, \mu \in \Par_n\}$ by Lemma \ref{lem R}.
A key feature of $\br_n$ is the following lemma.

\begin{lemma} \label{lem br_n}
For $\lm, \mu \in \Par_n$, $(\overline{\lm})\br_n(\overline{\mu})\neq0$ if and only if $(\overline{\mu})\br_n(\ov{\lmn})\neq0$.
In particular, $(\overline{\lm})\br_n(\ov{\lmn})\neq 0$. Moreover, $(\overline{\lm}||\overline{\mu})~(\overline{\mu}||\ov{\lmn})=(\overline{\lm}||\ov{\lmn})$.
\end{lemma}
\begin{proof}
Let $\lm=(\lm_1, \dots, \lm_n) \in \Par_n$, where $\lm_i \geq 0$ for $1 \leq i \leq n$, and $\ov{\lm}=(\ov{\lm}_1, \dots, \ov{\lm}_k)$, where $\ov{\lm}_j \geq 1$ for $1 \leq j \leq k$.
Here $\lm_i=|\{j; ~1 \leq j \leq k, \ov{\lm}_j \geq i\}|$, and $\ov{\lm}_j=|\{i; ~ 1 \leq i \leq n, \lm_i \geq j\}|$.
%In particular, $\lm_1=k$, and $\ov{\lm}_1 \leq n$.
Then $\ov{\lmn}=(n, \ov{\lm}_1, \dots, \ov{\lm}_{k-1})$.

Note that $\ov{\mu}_1 \leq n$ for any $\mu \in \Par_n$.
Lemma \ref{lem lmRmu} implies that $(\overline{\lm})\br_n(\overline{\mu})\neq0$ if and only if
$$n \geq \overline{\mu}_1 \geq \overline{\lm}_1 \geq \overline{\mu}_2 \geq \overline{\lm}_2 \cdots \geq \overline{\mu}_k \geq \overline{\lm}_k.$$
The same inequalities are equivalent to saying that $(\overline{\mu})\br_n(\overline{\lmn})\neq0$.
The rest of lemma directly follows from Lemma \ref{lem R}.
\end{proof}

Note that $\ov{\lmn}$ is obtained from $\ov{\lm}$ by adding one box on each column of the first $n$ columns.
Lemma \ref{lem br_n} says that the path from $\ov{\lm}$ to $\ov{\lmn}$ is the longest nonzero path among pathes starting from $\ov{\lm}$ in $\ov{\Gamma}_n$.
This property is the analogue of that of {\em Serre functors} of triangulated categories.
See \cite[Section 2.6]{Ke1} for the definition of Serre functors.
For our purpose, consider the special case for the bounded derived category $\mf{D}^b(A)$, where $A$ is a finite dimensional $\K$-algebra of finite global dimension.
The Serre functor is an endofunctor of $\mf{D}^b(A)$ defined as the derived functor of tensoring with the $A$-bimodule $DA=\Hom_{\K}(A,\K)$.
It is an autoequivalence of $\mf{D}^b(A)$. See \cite[Theorem 3.1]{Ke1}.

Motivated by the Serre functors above, consider $D\br_n=\Hom_{\K}(\br_n, \K)$, the linear dual of $\br_n$.
So $D\br_n$ has a dual basis $$\{[\mu | \lm]; (\lm)\br_n(\mu)\neq0, \overline{\lm}, \overline{\mu} \in \Par_n\},$$
by Lemma \ref{lem R}.
Moreover, $D\br_n$ is a $\br_n$-bimodule where the action of $\br_n$ is given by
$$(\mu'||\mu)[\mu | \lm]=[\mu'|\lm], ~~\mbox{if}~ (\lm||\mu')(\mu'||\mu)=(\lm||\mu),$$
$$[\mu | \lm](\lm||\lm')=[\mu|\lm'], ~~\mbox{if}~ (\lm||\lm')(\lm'|\mu)=(\lm||\mu).$$

Let $P(\ov{\lm})=\br_n \cdot (\ov{\lm})$ denote the projective $\br_n$-module associated to $\lm \in \Par_n$.

\begin{lemma} \label{lem bar sn}
There is an isomorphism $D\br_n \otimes_{\br_n} P(\ov{\lm}) \cong P(\ov{\lmn})$ of $\br_n$-modules.
\end{lemma}
\begin{proof}
The tensor product $D\br_n \otimes_{\br_n} P(\ov{\lm})$ has a $\K$-basis
$$\{[\ov{\mu}|\ov{\lm}]; (\ov{\lm})\br_n(\ov{\mu})\neq0, \mu \in \Par_n\}.$$
It follows from Lemma \ref{lem br_n} that this basis is bijection to
$$\{(\ov{\mu}||\ov{\lmn}); (\ov{\mu})\br_n(\ov{\lmn})\neq0, \mu \in \Par_n\},$$
which is a basis of $P(\ov{\lmn})$.
Moreover, the left $\br_n$-module $D\br_n \otimes_{\br_n} P(\ov{\lm})$ is generated by $[\ov{\lmn}|\ov{\lm}]$ such that $(\ov{\mu}||\ov{\lmn})[\ov{\lmn} | \ov{\lm}]=[\ov{\mu}|\ov{\lm}]$.
\end{proof}

Let $\Kom(\br_n)$ denote the homotopy category of finite dimensional projective $\br_n$-modules.
The endofunctor $D\br_n \ot^L_{\br_n} -$ preserves $\Kom(\br_n)$ by Lemma \ref{lem bar sn}.
Define $$\ov{\cs}_n=D\br_n \ot^L_{\br_n} -: \Kom(\br_n) \ra \Kom(\br_n).$$

\begin{rmk}
Since $\br_n$ is of infinite global dimension, $\ov{\cs}_n$ is not an equivalence.
\end{rmk}

Let $K_0(\br_n)$ denote the Grothendieck group of $\Kom(\br_n)$.
There is an isomorphism $K_0(\br_n) \ra \Z\lan\ov{\Par_n}\ran$ of abelian groups which maps $[P(\lm)]$ to $\lm$, where $\Z\lan\ov{\Par_n}\ran$ is the free abelian group with a basis $\Par_n$.
We fix this isomorphism.

\begin{thm} \label{thm bar sn}
%The linear map $K_0(\ov{\cs}_n): K_0(\br_n) \ra K_0(\br_n)$ agrees with the restriction of $\ov{S}_n$ to $\ov{\Par_n} \subset \Par \cong \vf^{(0)}$ under the isomorphism $K_0(\br_n) \cong \ov{\Par_n}$, where the isomorphism $\Par \cong \vf^{(0)}$ is in (\ref{eq iso vf0}).

There is a commutative diagram
$$
\xymatrix{
\Kom(\br_n) \ar[r]^{K_0}\ar[d]^{\ov{\cs}_n} & K_0(\br_n) \ar[r]^{\cong}\ar[d]^{K_0(\ov{\cs}_n)} &  \Z\lan\ov{\Par_n}\ran \ar[r]^{\subset} & \Z\lan\Par\ran \ar[r]^{\mf{x}} & (\vf)_0 \ar[d]^{\ov{S}_n}\\
\Kom(\br_n) \ar[r]^{K_0} & K_0(\br_n) \ar[r]^{\cong} &  \Z\lan\ov{\Par_n}\ran \ar[r]^{\subset} & \Z\lan\Par\ran \ar[r]^{\mf{x}} & (\vf)_0
}$$
where the map $\mf{x}$ is induced by the bijection in (\ref{eq iso vf0}).
\end{thm}
\proof
It is enough to prove that $\mf{x}(\ov{\lmn})=\ov{S}_n(\mf{x}(\ov{\lm}))$ for any $\lm=(\lm_1, \dots, \lm_n) \in \Par_n$ by Lemma \ref{lem bar sn}.
We have
\begin{align*}
\ov{S}_n(\mf{x}(\ov{\lm})) & = \tau^{-1} \circ \sum\limits_{i=-\infty}^{n}(-1)^i t_{2i+1}(\mf{x}(\ov{\lm})) \\
& = \tau^{-1} \circ \sum\limits_{i=-\infty}^{n}(-1)^i t_{2i+1} (2-2\lm_1, 4-2\lm_2, \dots, 2n-2\lm_n, 2(n+1), \dots) \\
& = \tau^{-1} \circ \sum\limits_{i=-\lm_1}^{n}(-1)^i t_{2i+1} (2-2\lm_1, 4-2\lm_2, \dots, 2n-2\lm_n, 2(n+1), \dots)
\end{align*}
The action of $t_{2i+1}$ is equal to the sum of two contractions $(2i)\lrcorner + (2i+2)\lrcorner$.
Therefore, all terms in the last line of the alternating sum are canceled in pairs except for the contraction $(2n+2)\lrcorner$.
So
\begin{align*}
\ov{S}_n(\mf{x}(\ov{\lm})) &= \tau^{-1} \circ (-1)^n~ (2n+2)\lrcorner (2-2\lm_1, 4-2\lm_2, \dots, 2n-2\lm_n, 2(n+1), \dots) \\
& = \tau^{-1} \circ (-1)^n~ (-1)^n (2-2\lm_1, 4-2\lm_2, \dots, 2n-2\lm_n, 2(n+2), \dots) \\
& = (2-2(\lm_1+1), 4-2(\lm_2+1), \dots, 2n-2(\lm_n+1), 2(n+1), \dots) \\
& = \mf{x}(\ov{\lmn}). \qed
\end{align*}

Since $\Par=\bigcup\limits_{n}\ov{\Par_n}$ and $\displaystyle{\lim_{n\to +\infty}}\ov{S}_n=\tau^{-1}\circ \ov{t}(z)|_{z=-1}$ as in (\ref{eq sn}), we say that the family $\ov{\cs}_n$ categorify $\ov{t}(z)|_{z=-1}$ as $n\ra +\infty$.

\subsection{The case of $F_n$}
Consider the endofunctor $\RHom_{F_n}(DF_n, -)$ of the derived category $\mf{D}(F_n)$ of $R_n$-modules.
It is right adjoint to the derived functor of tensor product:
\begin{gather} \label{eq s bars}
\Hom_{F_n}(DF_n \otimes^L_{F_n} M, N) \cong \Hom_{F_n}(M, \RHom_{F_n}(DF_n, N)),
\end{gather}
for any $M, N \in \mf{D}(F_n)$.
We want to use this adjoint property to compute $\RHom_{F_n}(DF_n, -)$.

Let $\mu=(\mu_1, \dots, \mu_n) \in \Par_n$.
To compute $DF_n \otimes_{F_n} P(\mu)$, we need some infinite dimensional projective $F_n$-modules.
Lemma \ref{lem lmRmu} implies that the longest nonzero path starting from $\mu$ in $\Gamma_n$ is an infinite path to the limit of $\mu^{\infty,k}=(k, \mu_1, \dots, \mu_{n-1})$ as $k \ra +\infty$.
Let $\mu^{\infty}=(\infty, \mu_1, \dots, \mu_{n-1})$ denote the limit.
There is a chain of maps of projective $F_n$-modules: $P(\mu^{\infty,k}) \ra P(\mu^{\infty,k+1})$ for $k \geq \mu_1$.
Each of them is an inclusion.
Let $P(\mu^{\infty})$ be the direct limit of the system of projectives $P(\mu^{\infty,k})$'s.
Then $P(\mu^{\infty})$ is an infinite dimensional projective $F_n$-module.
As a set, $P(\mu^{\infty})$ is the union of $P(\mu^{\infty,k})$'s with respect to the inclusions.

The module $DF_n \otimes_{F_n} P(\mu)$ has a $\K$-basis $\{[\eta|\mu]; ~ \eta \in \Par_n, (\mu)F_n(\eta)\neq0\}$.
If $(\mu)F_n(\eta)\neq0$, then $(\eta)F_n(\mu^{\infty,k}) \neq 0$ for some $k \geq \mu_1$ by Lemma \ref{lem lmRmu}.
So there is a chain of $F_n$ homomorphisms: $P(\mu^{\infty,k}) \ra DF_n \otimes_{F_n} P(\mu)$ which maps the generator
$(\mu^{\infty,k})$ to $[\mu^{\infty,k}|\mu]$.
They induce a map
$$\op{pr}(\mu): P(\mu^{\infty}) \ra DF_n \otimes_{F_n} P(\mu)$$
of $F_n$-modules which is surjective.

Let $W(\mu)=\Ker(\op{pr}(\mu))$.
Then $W(\mu)$ has a $\K$-basis $\{(\eta||\mu^{\infty}); (\mu)F_n(\eta)=0\}$.

If $\mu_n=0$, then $(\mu)F_n(\eta)=0$ if and only if $(\eta)F_n(\mu^{\infty})=0$.
Thus $\op{pr}(\mu)$ is an isomorphism of $F_n$-modules.

If $\mu_n \neq 0$, let $\mu^{1,i}=(\mu_1, \dots, \mu_{n-1}, \mu_n-1-i)$, for $0 \leq i \leq \mu_n-1$.
Then $W(\mu)$ has a $\K$-basis $\{(\mu^{1,i}|| \mu^{1,0}); 0 \leq i \leq \mu_n-1\}$, and is a submodule of $P(\mu^{1,0})$.
Let $\op{pr^1}(\mu): P(\mu^{1,0}) \ra W(\mu)$ denote the projection.
For $1 \leq j \leq n-1$, let
$$\mu^{2,j}=(\mu_1, \dots, \mu_j-1, \dots, \mu_{n-1}, \mu_n-1) \in \Par_n$$
be the partition obtained from $\mu^{1,0}$ by removing one box from the $j$-th row if the resulting sequence is decreasing.
The $F_n$-module $\Ker(\op{pr^1}(\mu))$ is generated by $(\mu^{2,j}| \mu^{1,0}) \in P(\mu^{1,0})$.
So there is a surjective map of $F_n$-modules $\op{pr^2}(\mu): \bigoplus\limits_{j}P(\mu^{2,j}) \ra \Ker(\op{pr}^1(\mu))$.
By repeating this procedure, we eventually obtain a projective resolution of $DF_n \otimes_{F_n} P(\mu)$:
\begin{gather}  \label{eq proj res}
0 \ra P(\mu^{n,\mf{j}_n}) \ra \cdots \bigoplus\limits_{|\mf{j}_k|=k-1}P(\mu^{k,\mf{j}_k}) \cdots \ra \bigoplus_{j}P(\mu^{2,j}) \ra P(\mu^{1,0}) \ra P(\mu^{\infty}).
\end{gather}
Here, $\mf{j}_k\subset\{1,\dots,n-1\}$ is a subset of $k-1$ elements, and $\mu^{k,\mf{j}_k}$ is the partition obtained from $\mu^1$ by removing one box from the $j$-th row for $j \in \mf{j}$ if the resulting sequence is decreasing.
For the last term when $k=n$, we have $\mf{j}_k=\{1,\dots,n-1\}$, and
$$\mu^{n,\mf{j}_n}=(\mu_1-1, \dots, \mu_{n-1}-1, \mu_n-1),$$
i.e. $\mu^{n,\mf{j}_n} \cup (1^n)=\mu$.

Let $L(\lm)$ denote the simple $F_n$-module associated to the idempotent $(\lm)$, for $\lm \in \Par_n$.
\begin{lemma} \label{lem sn}
There is an isomorphism $\RHom_{F_n}(DF_n, L(\lm)) \cong L(\lmn)[-n]$ in $\mf{D}(F_n)$.
\end{lemma}
\begin{proof}
We have $\Hom_{F_n}(DF_n \otimes P(\mu), L(\lm)[n]) \cong \Hom_{F_n}(P(\mu), \RHom_{F_n}(DF_n, L(\lm)[n])$ by the adjoint property.
The projective resolution (\ref{eq proj res}) implies that the left hand side of the isomorphism is one dimensional if $\lm \cup (1^n)=\mu$; otherwise, it is zero.
It follows from the right hand side that $\RHom_{F_n}(DF_n, L(\lm)[n]) \cong L(\lmn)$.
\end{proof}

\begin{lemma} \label{lem sn finite}
The left global dimension of $F_n$ is $n$.
\end{lemma}
\begin{proof}
Let $\lm=(\lm_1, \dots, \lm_k) \in \Par_n$, where $\lm_s>0$ for $1 \leq s \leq k \leq n$.
Using a similar argument as in the computation of $DF_n \otimes P(\mu)$, the simple module $L(\lm)$ has a projective resolution of length $k$:
\begin{gather}  \label{eq proj res simple}
0 \ra P(\wt{\lm}^{k, \mf{j}_k}) \ra \cdots \bigoplus\limits_{|\mf{j}_s|=s}P(\wt{\lm}^{s,\mf{j}_s}) \cdots \ra \bigoplus\limits_{j}P(\wt{\lm}^{1,j}) \ra P(\lm) \ra L(\lm) \ra 0.
\end{gather}
Here, $\mf{j}_s\subset\{1,\dots,k\}$ is a subset of $s$ elements, and $\wt{\lm}^{s,\mf{j}_s}$ is the partition obtained from $\lm$ by removing one box from the $j$-th row for $j \in \mf{j}_s$ if the resulting sequence is decreasing.
For the last term, we have $\mf{j}_k=\{1,\dots,k\}$, and
$\wt{\lm}^k=(\lm_1-1, \dots, \lm_{k}-1).$
Thus, any simple module has a projective resolution of length not greater than $n$.
The lemma follows from the Horseshoe Lemma.
\end{proof}

In order to compute $\RHom_{F_n}(DF_n, P(\lm))$, we need some finite dimensional truncations of $F_n$.
Let $\Par_n^m=\{\lm \in \Par_n, |\lm| \leq m\}$.
Define $$F_n^m=\bigoplus\limits_{\lm, \mu \in \Par_n^m} (\lm)F_n(\mu)$$ as a family of subalgebras of $F_n$.
Each $F_n^m$ is finite dimensional and has finite global dimension.
For a given $\lm \in \Par_n^m$, the simple $F_n^m$-module associated to $\lm$ is still denoted by $L(\lm)$ by abuse of notation.
For sufficiently large $m$, $\lmn \in \Par_n^m$, and the analogue of Lemma \ref{lem sn} for $F_n^m$-modules still holds:
$$\RHom_{F_n^m}(DF_n^m, L(\lm)[n]) \cong L(\lmn).$$
Since $F_n^m$ is finite dimensional and has finite global dimension, the endofuntor $DF_n^m \otimes^L -$ is an autoequivalence of $\mf{D}^b(F_n^m)$.
Its inverse is isomorphic to $\RHom_{F_n^m}(DF_n^m, -)$. See \cite[Theorem 3.1]{Ke1}.
Therefore, the functor $\RHom_{F_n^m}(DF_n^m, -)$ is an autoequivalence.

We take $m$ to be sufficiently large from now on.
The projective $F_n^m$-module $P(\lm)$ has a $\K$-basis $\{(\eta||\lm); (\eta)F_n^m(\lm)\neq 0\}$.
It is an extension of the simple modules $L(\eta)$'s.
Lemma \ref{lem sn} implies that $\RHom_{F_n^m}(DF_n^m, P(\lm)[n])$ is isomorphic to an $F_n^m$-module which is an extension of simple modules $$\RHom_{F_n^m}(DF_n^m, L(\eta)[n]) \cong L(\eta\cup(1^n)).$$
We denote this $F_n^m$-module by $Q(\lm)$.
Therefore, $Q(\lm)$ has a $\K$-basis $$\{(\eta\cup(1^n)||\lmn); (\eta)F_n^m(\lm)\neq 0\}.$$
The dimension of $\op{Ext}^i(L(\eta), L(\eta'))$ for any $\eta, \eta'$ and $i$ is at most one dimensional.
Moreover, the autoequivalence $\RHom_{F_n^m}(DF_n^m, -)$ induces isomorphisms on the Ext groups.
So the $F_n^m$-module structure on $Q(\lm)$ is given by the multiplication in $F_n^m$.
In particular, $Q(\lm)$ is generated by $(\lmn||\lmn)$ as a $F_n^m$-module.
So there is a surjective map
$$\wtpr(\lm): P(\lmn) \ra Q(\lm).$$
Let $\lm^0=\lmn=(\lm_1+1, \dots, \lm_n+1)$ for $\lm=(\lm_1,\dots, \lm_n)$.
We define
\begin{gather} \label{eq lmt}
\lm^t=(\lm_1+1, \dots, \lm_{n-t}+1,\lm_{n-t+2},\dots,\lm_n,0) \in \Par_n,
\end{gather}
for $1 \leq t \leq n$, where $\lm_{n+1}$ is understood as zero.
Then $(\lm^{t})F_n^m(\lm^{t-1}) \neq 0$.
The right multiplication with $(\lm^{t}||\lm^{t-1})$ induces a homomorphism $\wtpr (\lm)^{t}: P(\lm^{t}) \ra P(\lm^{t-1})$.

\begin{lemma} \label{lem sn proj}
The $F_n^m$-module $Q(\lm)$ has a projective resolution of length $n$:
\begin{gather} \label{eq proj q}
0 \ra P(\lm^n) \ra \cdots \ra P(\lm^{t+1}) \xrightarrow{\wtpr(\lm)^{t+1}} P(\lm^t) \xrightarrow{\wtpr(\lm)^{t}} P(\lm^{t-1}) \ra \cdots \ra P(\lm^0) \xrightarrow{\wtpr(\lm)} Q(\lm) \ra 0.
\end{gather}
\end{lemma}
\begin{proof}
The $F_n^m$-module $\Ker(\wtpr (\lm))$ has a $K$-basis
$$\{(\mu||\lmn);~ \mu \neq \eta\cup(1^n)~~\mbox{for any}~~(\eta)F_n^m(\lm) \neq0 \}.$$
Then $(\mu)F_n^m(\lm^1)\neq 0$ for any $\mu$ in the basis, where $\lm^1=(\lm_1+1, \dots, \lm_{n-1}+1,0)$.
So there is a surjective map $P(\lm^1) \ra \Ker(\wtpr (\lm))$.

The partitions $\lm^t$ and $\lm^{t-1}$ differ only their $(n-t+1)$-th terms, for $1\leq t \leq n$.
More precisely,
\begin{align*}
\lm^{t+1}=(\lm_1+1, \dots,  & ~\lm_{n-t+1},  \quad \lm_{n-t+2}, \dots,\lm_n,0), \\
\lm^t=(\lm_1+1, \dots,  & ~\lm_{n-t}+1,  ~\lm_{n-t+2},\dots,\lm_n,0), \\
\lm^{t-1}=(\lm_1+1, \dots,  & ~\lm_{n-t}+1,  ~\lm_{n-t+1}+1,\dots,\lm_n,0).
\end{align*}
The multiplication $(\lm^{t+1}||\lm^{t})(\lm^{t}||\lm^{t-1})=0$ by Lemma \ref{lem lmRmu}.
So the composition $\wtpr (\lm)^{t} \circ \wtpr (\lm)^{t+1}$ is zero.
The $F_n^m$-module $\Ker(\wtpr (\lm)^{t})$ is generated by $(\lm^{t+1}||\lm^{t})$.
Thus, the sequence is exact at $P(\lm^t)$.
Similarly, one can prove that $\wtpr(\lm)^{n}$ is injective.
\end{proof}

Lemma \ref{lem sn proj} remains true as $m$ tends to infinity.
So $\RHom_{F_n}(DF_n, P(\lm)[n])$ is isomorphic to an $F_n$-module which is still denoted by $Q(\lm)$.
It has a similar projective resolution as $(\ref{eq proj q})$, where all projectives are $F_n$-modules.

Let $\Kom(F_n)$ denote the homotopy category of finite dimensional projective $F_n$-modules.
The endofunctor $\RHom_{F_n}(DF_n, -)$ preserves $\Kom(F_n)$.
Define $$\cs_n=\RHom_{F_n}(DF_n, -): \Kom(F_n) \ra \Kom(F_n).$$

Let $K_0(F_n)$ denote the Grothendieck group of $\Kom(F_n)$.
There is an isomorphism $K_0(F_n) \ra \Z\lan\Par_n\ran$ of abelian groups which maps $[P(\lm)]$ to $\lm$.
We fix this isomorphism.

\begin{thm} \label{thm sn}
There is a commutative diagram
$$
\xymatrix{
\Kom(F_n) \ar[r]^{K_0}\ar[d]^{{\cs}_n} & K_0(F_n) \ar[r]^{\cong}\ar[d]^{K_0({\cs}_n)} &  \Z\lan\Par_n\ran \ar[r]^{\subset} & \Z\lan\Par\ran \ar[r]^{\mf{x}} & (\vf)_0 \ar[d]^{{S}_n}\\
\Kom(F_n) \ar[r]^{K_0} & K_0(F_n) \ar[r]^{\cong} &  \Z\lan\Par_n\ran \ar[r]^{\subset} & \Z\lan\Par_n\ran \ar[r]^{\mf{x}} & (\vf)_0
}$$
where the map $\mf{x}$ is induced by the bijection in (\ref{eq iso vf0}).
\end{thm}
\begin{proof}
Let $\lm=(\lm_1,\dots,\lm_n) \in \Par_n$ where $\lm_n \geq 0$, and $\ov{\lm}=(\ov{\lm}_1,\dots,\ov{\lm}_k)$ where $k=\lm_1$.
It is enough to prove that $\mf{x}([\cs_n(P(\lm))])={S}_n(\mf{x}(\lm))$.
The left hand side
\begin{gather}
\label{eq sn lhs}
\mf{x}([\cs_n(P(\lm))])=\mf{x}([Q(\lm)[-n]])=\sum\limits_{t=0}^{n}(-1)^{t+n} \mf{x}(\lm^t)
\end{gather}
by Lemma \ref{lem sn proj}.
Here $\lm^t$ is defined in (\ref{eq lmt}).

Note that $\mf{x}(\lm) \in \vf^{(0)}$ is expressed using the dual partition $\ov{\lm}$, see (\ref{eq iso vf0}).
Since $\lm^{t+1}$ and $\lm^t$ only differ at one place, we describe $\ov{\lm^t}$ by induction on $t$ as follows.
For $t=0$,
$$\lm^0=\lmn=(\lm_1+1,\dots,\lm_n+1), \qquad \ov{\lm^0}=(\ov{\lm^0_1}, \dots, \ov{\lm^0_{k+1}})=(n, \ov{\lm}_1,\dots,\ov{\lm}_k),$$
where $\ov{\lm^0_s}=n$ for $1 \leq s \leq \lm_{n}+1$.
Given $\ov{\lm^{t}}=(\ov{\lm^{t}_1}, \dots, \ov{\lm^{t}_{k+1}})$, we have $\ov{\lm^{t+1}}=(\ov{\lm^{t+1}_1}, \dots, \ov{\lm^{t+1}_{k+1}})$, where
\begin{align*}
\ov{\lm^{t+1}_s}=& \left\{
\begin{array}{ll}
\ov{\lm^{t}_s}-1 & \quad\mbox{if} \hspace{0.3cm} \lm_{n-t+1}+1 \leq s \leq \lm_{n-t}+1; \\
\ov{\lm^{t}_s} & \quad\mbox{otherwise}.
\end{array}\right.
\end{align*}
Here, $\lm_{n+1}$ is understood as zero.
One can prove that $\ov{\lm^t_s}=n-t$ for $s \in [\lm_{n-t+1}+1, \lm_{n-t}+1]$ by induction on $t$.
The monomials $\mf{x}(\lm^{t}), \mf{x}(\lm^{t+1})$ are
$$\begin{array}{cccccc}
\mf{x}(\lm^{t})&=( \dots,  & 2(\lm_{n-t+1}+1-(n-t)), & \dots, & 2(\lm_{n-t}+1-(n-t))&,\dots), \\
\mf{x}(\lm^{t+1})&=( \dots,  & 2(\lm_{n-t+1}+1-(n-t-1)),& \dots, & 2(\lm_{n-t}+1-(n-t-1))&,\dots).
\end{array}$$
Let $s_t=\lm_{n-t+1}-(n-t)$ for $0 \leq t \leq n$.
The subsequence of $\mf{x}(\lm^{t})$ from $2(s_t+1)$ to $2s_{t+1}$ consists of consecutive even integers.
Moreover, $\mf{x}(\lm^{t})$ and $\mf{x}(\lm^{t+1})$ only differ by two factors
$$2(s_t+1) \in \mf{x}(\lm^{t}) \backslash \mf{x}(\lm^{t+1}), \quad 2(s_{t+1}+1) \in \mf{x}(\lm^{t+1}) \backslash \mf{x}(\lm^{t}).$$
We have
$$\mf{x}(\lm^{t+1})=(-1)^{s_{t+1}-s_{t}+1} 2(s_{t+1}+1) \wedge \left(  2(s_t+1)\lrcorner  \mf{x}(\lm^{t}) \right),$$
$$\mf{x}(\lm^{t})=(-1)^{s_{t}-s_{0}+t} 2(s_{t}+1) \wedge \left(  2(s_0+1)\lrcorner  \mf{x}(\lm^{0}) \right),$$
by induction on $t$.
For $t=0$, $$\mf{x}(\lm^{0})=(2-2n,4-2\ov{\lm}_1,\dots,2(k+1)-2\ov{\lm}_k,2(k+2),\dots),$$ and $s_0=\lm_{n-0+1}-(n-0)=-n$. So
\begin{align*}
\mf{x}(\lm^{t})=&(-1)^{s_{t}-s_0+t} 2(s_{t}+1) \wedge (4-2\ov{\lm}_1,\dots,2(k+1)-2\ov{\lm}_k,2(k+2),\dots) \\
=& (-1)^{n+t} (-1)^{s_t} 2(s_{t}+1) \wedge (\tau \circ (2-2\ov{\lm}_1,\dots,2k-2\ov{\lm}_k,2(k+1),\dots) \\
=& (-1)^{n+t}~ \tau \circ ((-1)^{s_t} 2s_{t} \wedge (2-2\ov{\lm}_1,\dots,2k-2\ov{\lm}_k,2(k+1),\dots)) \\
=& (-1)^{n+t}~ \tau \circ ((-1)^{s_t} t_{2s_{t}}(\mf{x}(\lm)).
\end{align*}
Now (\ref{eq sn lhs}) becomes
$$\mf{x}([\cs_n(P(\lm))])=\tau \circ \sum\limits_{t=0}^{n}(-1)^{s_t} t_{2s_{t}}(\mf{x}(\lm)).$$
On the other hand, $${S}_n(\mf{x}(\lm))=\tau \circ \sum\limits_{s=-n}^{+\infty}(-1)^{s} t_{2s}(\mf{x}(\lm)).$$
There are exactly $n+1$ terms in the infinite sum are nonzero which are in bijection with the $n+1$ terms in $\mf{x}([\cs_n(P(\lm))])$.
\end{proof}

\end{document}